\documentclass[a4paper]{amsart}
\usepackage{graphicx}
\usepackage[colorlinks, linkcolor= blue, citecolor= red]{hyperref}
\usepackage{pdfsync}
\usepackage{color}

\usepackage[T1]{fontenc}
\usepackage[utf8]{inputenc} 

\usepackage{amsmath, amssymb, amsfonts, amscd, amsthm, mathrsfs}

\newcommand{\C}{\mathbb{C}}
\newcommand{\z}{\mathbb{Z}}
\renewcommand{\r}{\mathbb{R}}

\newcommand{\N}{\mathscr{N}}
\newcommand{\q}{\mathbb{Q}}

\newcommand{\g}{\mathscr{G}}
\newcommand{\topo}[1]{\left| #1 \right|}
\newcommand{\B}{\mathscr{B}}
\newcommand{\DD}{\mathscr{D}}
\newcommand{\F}{\mathscr{F}}
\newcommand{\W}{\mathscr{W}}
\newcommand{\cell}{\mathscr{C}}
\newcommand{\p}{\mathbb{P}^1}
\newcommand{\D}{\mathbb{D}}
\newcommand{\DO}{\dot{\mathbb{D}}}
\newcommand{\pminus}{\p \smallsetminus \{ 0, 1, \infty \}}

\newcommand{\sets}{\mathfrak{Sets}}
\newcommand{\covs}{\mathfrak{Cov}(\p)}
\newcommand{\belyi}{\mathfrak{Belyi}}
\newcommand{\et}{\mathfrak{Etale}(\C(x))}
\newcommand{\qb}{ {\overline{\q}} }
\newcommand{\etq}{\mathfrak{Etale}(\qb (x))}
\newcommand{\cat}{\mathfrak{Dessins}}
\newcommand{\ucat}{\mathfrak{UDessins}}
\newcommand{\m}{\mathscr{M}}
\newcommand{\act}[2]{{}^{#1} #2}
\newcommand{\gal}{\operatorname{Gal}(\qb / \q)}
\newcommand{\Gal}{\operatorname{Gal}}
\newcommand{\gt}{\Gamma}
\newcommand{\GT}{\mathcal{GT}}
\newcommand{\rGT}{\widehat{\GT}}
\newcommand{\farout}{{Out}}
\newcommand{\bs}{\backslash}

\newcommand{\image}[1]{\smallskip\begin{center} #1 \end{center}\smallskip}
\newcommand{\figurehere}[2]{\smallskip\begin{center} 
\includegraphics[width=#1\textwidth]{#2}\end{center}\smallskip}

\newtheoremstyle{pedro}{}{}{\itshape}{}{\sc}{~--}{ }{\thmname{#1}\thmnumber{ #2}\thmnote{ (#3)}}

\newtheoremstyle{pedrodef}{}{}{}{}{\sc}{~--}{ }{\thmname{#1}\thmnumber{ #2}\thmnote{ (#3)}}

\theoremstyle{pedro}
\newtheorem{lem}{Lemma}[section]

\newtheorem{thm}[lem]{Theorem}

\newtheorem{prop}[lem]{Proposition}

\newtheorem{coro}[lem]{Corollary}

\theoremstyle{remark}

\newtheorem{rmk}[lem]{Remark}

\theoremstyle{pedrodef}

\newtheorem{ex}[lem]{Example}

\title[Dessins d'enfants]{An elementary approach to dessins d'enfants and the
  Grothendieck-Teichmüller group}

\author{Pierre Guillot}

\address{
Universit\'{e} de Strasbourg \& CNRS\\
Institut de Recherche Math\'{e}matique Avanc\'{e}e\\
7~Rue Ren\'{e} Descartes\\
67084 Strasbourg, France}

\email{guillot@math.unistra.fr}

\let\oldtocsection=\tocsection
\let\oldtocsubsection=\tocsubsection
\let\oldtocsubsubsection=\tocsubsubsection

\renewcommand{\tocsection}[2]{\hspace{0em}\oldtocsection{#1}{#2}}
\renewcommand{\tocsubsection}[2]{\hspace{2em}\oldtocsubsection{#1}{#2}}
\renewcommand{\tocsubsubsection}[2]{\hspace{2em}\oldtocsubsubsection{#1}{#2}}

\numberwithin{equation}{section}

\begin{document}

\maketitle

\begin{abstract}

We give an account of the theory of dessins d'enfants which is both
elementary and self-contained. We describe the equivalence of many
categories (graphs embedded nicely on surfaces, finite sets with
certain permutations, certain field extensions, and some classes of
algebraic curves), some of which are naturally endowed with an action
of the absolute Galois group of the rational field. We prove that the
action is faithful. Eventually we prove that~$\gal$ embeds into the
Grothendieck-Teichmüller group~$\rGT_0$ introduced by Drinfeld. There
are explicit approximations of~$\rGT_0$ by finite groups, and we hope
to encourage computations in this area.

Our treatment includes a result which has not appeared in the
literature yet: the action of~$\gal$ on the subset of {\em regular}
dessins -- that is, those exhibiting maximal symmetry -- is also
faithful.

\end{abstract}


{\footnotesize
\noindent{\bfseries Status.} This paper should be identical, or almost identical, to that which is to appear in {\em L'enseignement mathématique}. If you do not intend to print this document, then you may want to \href{http://www-irma.u-strasbg.fr/~guillot/research/dessins.pdf}{click here} and try a PDF file which is optimized for on-screen reading.
}

\section*{Introduction}

The story of {\em dessins d'enfants} (children's drawings) is best told
in two episodes.

The first side of the story is a surprising unification of
different-looking theories: graphs embedded nicely on surfaces, finite
sets with certain permutations, certain field extensions, and some
classes of algebraic curves (some over~$\C$, some over~$\qb$), all
turn out to define equivalent categories. This result follows from
powerful and yet very classical theorems, mostly from the 19th
century, such as the correspondence between Riemann surfaces and their
fields of meromorphic functions (of course known to Riemann himself),
or the basic properties of the fundamental group (dating back to
Poincaré).

One of our goals with the present paper is to give an account of this
theory that sticks to elementary methods, as we believe it
should. (For example we shall never need to appeal to ``Weil's
rigidity criterion'', as is most often done in the literature on the
subject; note that it is also possible, in fact, to read most of this
paper without any knowledge of algebraic curves.) Our development is
moreover as self-contained as is reasonable: that is, while this paper
is not the place to develop the theory of Riemann surfaces, Galois
extensions or covering spaces from scratch -- we shall refer to basic
textbooks for these -- we give complete arguments from there. Also, we
have striven to state the results in terms of actual equivalences of
categories, a slick language which unfortunately is not always
employed in the usual sources.

The term {\em dessins d'enfants} was coined by Grothendieck in
\cite{esquisse}, in which a vast programme was laid out, giving the
theory a new thrust which is the second side of the story we wish to
tell. In a nutshell, some of the categories mentioned above naturally
carry an action of~$\gal$, the absolute Galois group of the rational
field. This group therefore acts on the set of isomorphism classes of
objects in any of the equivalent categories; in particular one can
define an action of the absolute Galois group on graphs embedded on
surfaces. In this situation however, the nature of the Galois action
is really very mysterious - it is hoped that, by studying it, light
may be shed on the structure of~$\gal$. It is the opportunity to bring
some kind of basic, visual geometry to bear in the study of the
absolute Galois group that makes {\it dessins d'enfants} -- embedded
graphs -- so attractive.


In this paper we explain carefully, again relying only on elementary
methods, how one defines the action, and how one proves that it is
{\it faithful}. This last property is clearly crucial if we are to
have any hope of studying~$\gal$ by considering graphs. We devote some
space to the search for invariants of dessins belonging to the same
Galois orbit, a major objective in the field.


When a group acts faithfully on something, we can usually obtain an
embedding of it in some automorphism group. In our case, this leads to
the {\em Grothendieck-Teichmüller group}~$\rGT$, first introduced by
Drinfeld in~\cite{drinfeld}, and proved to contain~$\gal$ by Ihara
in~\cite{ihara}. While trying to describe Ihara's proof in any detail
would carry us beyond the scope of this paper, we present a complete,
elementary argument establishing that~$\gal$ embeds into the slightly
larger group~$\rGT_0$ also defined by Drinfeld. In fact we work with
a group~$\GT$ isomorphic to~$\rGT_0$, and which is an inverse limit
\[ \GT = \lim_n \GT(n) \, ;   \]
here~$\GT(n)$ is a certain subgroup of~$Out(H_n)$ for an explicitly
defined finite group~$H_n$. So describing~$H_n$ and~$\GT(n)$ for
some~$n$ large enough gives rough information about~$\gal$ -- and it
is possible to do so in finite time.

In turn, we shall see that understanding~$H_n$ amounts, in a sense, to
understanding all finite groups generated by two elements, whose order
is less than~$n$. We land back on our feet: from the first part of
this paper, those groups are in one to one correspondence with some
embedded graphs, called regular, exhibiting maximal symmetry. The
classification of ``regular maps'', as they are sometimes called, is a
classical topic which is still alive today. 
\[ \star \star \star  \]

Let us add a few informal comments of historical nature, not written
by an expert in the history of mathematics.

The origin of the subjet is the study of ``maps'', a word meaning
graphs embedded on surfaces in a certain way, the complement of the
graph being a disjoint union of topological discs which may be
reminiscent of countries on a map of the world. Attention has focused
quickly on ``regular maps'', that is, those for which the automorphism
group is as large as possible. For example, ``maps'' are mentioned in
the 1957 book~\cite{coxeter} by Coxeter and Moser, and older
references can certainly be found. The 1978 paper~\cite{jns} by Jones
and Singerman has gained a lot of popularity; it gave the field
stronger foundations, and already established bijections between
``maps'' and combinatorial objects such as permutations on the one
hand, and also with compact Riemann surfaces, and thus
complex algebraic curves, on the other hand. For a recent survey on
the classification of ``maps'', see~\cite{jozef}. 

Then came the {\em Esquisse d'un programme}~\cite{esquisse}, written
by Grothendieck between 1972 and 1984. Dessins can be seen as
algebraic curves over~$\C$ with some extra structure (a morphism
to~$\p$ with ramification above~$0, 1$ or~$\infty$ only), and
Grothendieck knew that such a curve must be defined over~$\qb$. Since
then, this remark has been known as ``the obvious part of Belyi's
theorem'' by people working in the field, even though it is not
universally recognized as obvious, and has little to do with Belyi
(one of the first complete and rigorous proofs is probably that by
Wolfart~\cite{wolfart}). However, Grothendieck was very impressed by
the simplicity and strength of a result by Belyi~\cite{belyi} stating
that, conversely, {\em any} algebraic curve defined over~$\qb$ can be
equipped with a morphism as above (which is nowadays called a Belyi
map, while it has become common to speak of Belyi's theorem to mean
the equivalence of definability of~$\qb$ on the one hand, and the
possibility of finding a Belyi map on the other hand). Thus the theory
of dessins encompasses all curves over~$\qb$, and Grothendieck pointed
out that this simple fact implied that the action of~$\gal$ on dessins
must be faithful. The {\em esquisse} included many more ideas which
will not be discussed here. For a playful exposition of many examples
of the Galois action on dessins, see~\cite{lando}. 

Later, in 1990, Drinfeld defined~$\rGT$ in~\cite{drinfeld} and
studied its action on braided categories, but did not relate it
explicitly to~$\gal$ although the motivation for the definition came
from the {\em esquisse}. It was Ihara in 1994~\cite{ihara} who proved
the existence of an embedding of~$\gal$ into~$\rGT$; it is interesting
to note that, if dessins d'enfants were the original idea for Ihara's
proof, they are a little hidden behind the technicalities.

The Grothendieck-Teichmüller group has since been the object of much
research, quite often using the tools of quantum algebra in the spirit
of Drinfeld's original approach. See also~\cite{fresse} by Fresse,
which establishes an interpretation of~$\rGT$ in terms of operads.
\[ \star \star \star  \]

Here is an outline of the paper. In section~\ref{sec-dessins}, we
introduce cell complexes, that is, spaces obtained by glueing discs to
bipartite graphs; when the result is a topological surface, we have a
{\em dessin}. In the same section we explain that dessins are entirely
determined by two permutations. In section~\ref{sec-eq-cats}, we quote
celebrated, classical results that establish a number of equivalences
of categories between that of dessins and many others, mentioned
above. In section~\ref{sec-regularity} we study the regularity
condition in detail. The Galois action is introduced in
section~\ref{sec-galois-action}, where we also present some concrete
calculations. We show that the action is faithful. Finally in
section~\ref{sec-GT} we prove that~$\gal$ embeds into the group~$\GT$
described above. 

In the course of this final proof, we obtain seemingly for free the
following refinement: the action of~$\gal$ on {\em regular} dessins is
also faithful. This fact follows mostly from a 1980 result by
Jarden~\cite{jarden} (together with known material on dessins), and it
is surprising that it has not been mentioned in the literature
yet. While this work was in its last stages, I have learned from
Gareth Jones that the preprint~\cite{gabi} by Andrei Jaikin-Zapirain
and Gabino Gonzalez-Diez contains generalizations of Jarden's theorem
while the faithfulness of the Galois action on regular dessins is
explicitly mentioned as a consequence (together with more precise
statements). Also in~\cite{bauer}, a preprint by Ingrid Bauer,
Fabrizio Catanese and Fritz Grunewald, one finds the result stated.

\noindent{\em Acknowledgements.} Nick Gill and Ian Short have followed the development of this paper from the very early stages, and I have benefited greatly from their advice. I also want to thank Gareth Jones for kind words about this work as it was reaching completion. Further corrections have been made based on comments by Olivier Guichard and Pierre de la Harpe, for which I am grateful.

\tableofcontents

\section{Dessins} \label{sec-dessins}

In this section we describe the first category of interest to us,
which is that of graphs embedded on surfaces in a particularly nice
way. These have been called sometimes ``maps'' in the literature, a
term which one should avoid if possible given the other meaning of the
word ``map'' in mathematics. We call them {\em dessins}.

The reader may be surprised by the number of pages devoted to this
first topic, and the level of details that we go into. Would it not
suffice to say that the objects we study are graphs embedded on
surfaces, whose complement is a union of open discs, perhaps with just
a couple of technical conditions? (A topologist would say ``a
CW-complex structure on a surface''.) 

This would not be appropriate, for several reasons. First and
foremost, we aim at proving certain equivalences of categories,
eventually (see next section). With the above definition, whether one
takes as morphisms all continuous maps between surfaces, or restricts
attention to the ``cellular'' ones, in any case there are simply too
many morphisms taken into account (see for example~\cite{jns}). Below,
we get things just right.

Another reason is that we already present {\em two} categories in this
section, not just one: dessins are intimately related to finite sets
endowed with certain permutations. The two categories are equivalent
and indeed so close that we encourage the reader to always think of
these two simultaneously; we take the time to build the intuition for
this. 

Note also that our treatment is very general, including non-orientable
dessins as well as dessins on surfaces with boundary.

Finally, the material below is so elementary that it was possible to
describe it with absolutely no reference to textbooks, an opportunity
we took. We think of the objects defined in this section as the most
down-to-earth of the paper, while the other categories to be
introduced later are here to shed light on dessins.

\subsection{Bipartite graphs}

We start with the definition of {\em bipartite graphs}, or {\em
  bigraphs} for short, which are essentially graphs made of black and
white vertices, such that the edges only connect vertices of different
colours. More formally, a bigraph consists of \begin{itemize}
\item a set~$B$, the elements of which we call the black vertices,
\item a set~$W$, the elements of which we call the white vertices,
\item a set~$D$, the elements of which we call the {\em darts},
\item two maps~$\B \colon D \longrightarrow B$ and~$\W \colon D
  \longrightarrow W$.
\end{itemize}

In most examples all of the above sets will be finite, but in general
we only specify a local finiteness condition, as follows. The degree
of~$w \in W$ is the number of darts~$d$ such that~$\W(d) = w$; the
degree of~$b \in B$ is the number of darts~$d$ such that~$\B(d) =
b$. We require that all degrees be finite.

For example, the following picture will help us describe a bigraph.
\figurehere{0.3}{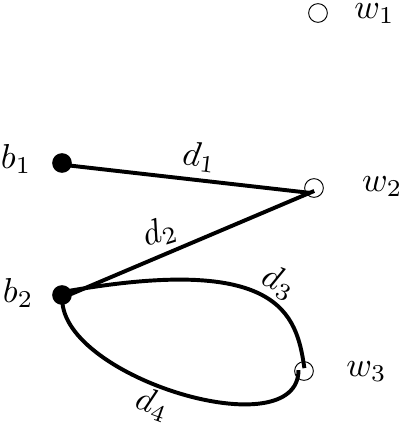}

Here~$B= \{ b_1, b_2 \}$, while~$W= \{ w_1, w_2, w_3 \}$ and~$D= \{
d_1, d_2, d_3, d_4 \}$. The maps~$\B$ and~$\W$ satisfy, for example,
$\B(d_1)= b_1$ and~$\W(d_1) = w_2$. Note that bigraphs according to
this definition are naturally labeled, even though we will often
suppress the names of the vertices and darts in the pictures.

The notion of morphism of bigraphs is the obvious one: a morphism
between~$\g = (B, W, D, \B, \W)$ and~$\g' = (B', W', D', \B', \W')$ is
given by three maps~$B \to B'$, $W \to W'$ and $\Delta \colon D \to
D'$ which are compatible with the maps~$\B, \W, \B', \W'$.
Isomorphisms are invertible morphisms, unsurprisingly. (Pedantically,
one could define an unlabeled bigraph to be an isomorphism class of
bigraphs.) 

To a bigraph~$\g$ we may associate a topological space~$\topo{\g}$, by
attaching intervals to discrete points according to the maps~$\B$
and~$\W$; in the above example, and in all others, it will look just
like the picture. To this end, take for each~$d \in D$ a copy~$I_d$ of
the unit interval~$[0,1]$ with its usual topology. Then consider
\[ Y = \coprod_{d \in D} \, I_d  \]
with the disjoint union topology, and 
\[ X = Y \coprod B \coprod W \, .  \]
(Here~$B$ and~$W$ are given the discrete topology.) On~$X$ there is an equivalence relation corresponding to the identifications imposed by the maps~$\B$ and~$\W$. In other words, the equivalence class~$[b]$ of~$b \in B$ is such that~$[b] \cap I_d = \{ 0 \}$ if~$\B(d) = b$ and~$[b] \cap I_d = \emptyset$ otherwise, while~$[b] \cap B = \{ b \}$ and~$[b] \cap W = \emptyset$; the description of the equivalence class~$[w]$ when~$w \in W$ is analogous, with~$[w] \cap I_d = \{ 1 \}$ precisely when~$\W(d) =w$. All the other equivalence classes are singletons. The space~$\topo{\g}$ is the set of equivalence classes, with the quotient topology. Clearly, an isomorphism of graphs induces a homeomorphism between their topological realizations.

Finally, we point out that usual graphs (the reader
may pick their favorite definition) can be seen as bigraphs by
``inserting a white vertex inside each edge''. We will not formalize
this here, although it is very easy. In what follows we officially
define a graph to be a bigraph in which all white vertices have degree
precisely~$2$; a pair of darts with a common white vertex form an {\em
  edge}.  The next picture, on which you see four edges, summarizes
this.

\figurehere{0.5}{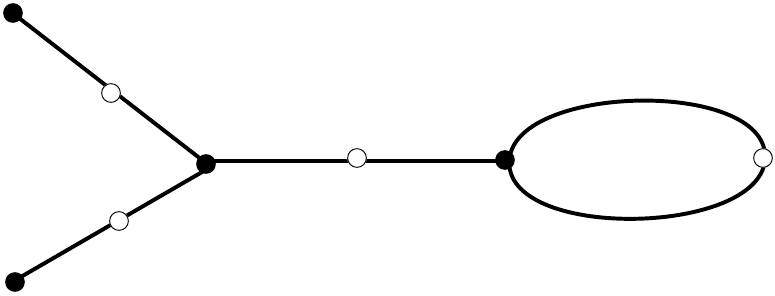}

\subsection{Cell complexes}

Suppose a bigraph~$\g$ is given. A {\em loop} on~$\g$ is a sequence of
darts describing a closed path on~$\g$ alternating between black and
white vertices. More precisely, a loop is a tuple 
\[ (d_1, d_2, \ldots, d_{2n}) \in D^{2n} \]
such that~$\W(d_{2i+1}) = \W(d_{2i+2})$ and~$\B(d_{2i+2}) =
\B(d_{2i+3})$, for~$0 \le i \le n-1$, where~$d_{2n+1}$ is to be
understood as~$d_1$. We think of this loop as starting and ending
with the black vertex~$\B(d_1)$, and visiting along the way the points
$\W(d_2)$, $\B(d_3)$, $\W(d_4)$, $\B(d_5)$, $\W(d_6), \ldots $ (It is
a little surprising to adopt such a convention, that loops always
start at a black vertex, but it does simplify what follows.)

For example, consider the following square:

\figurehere{0.5}{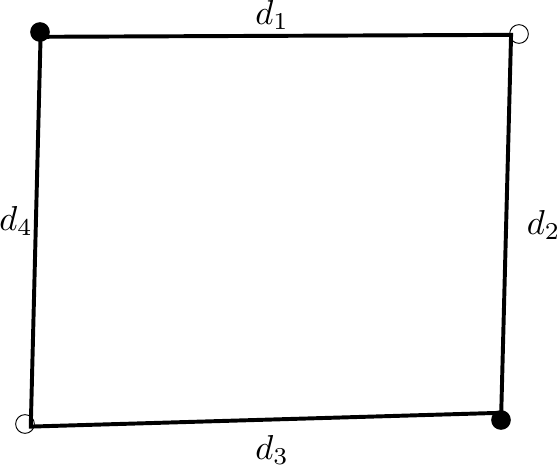}
On this bigraph we have a loop~$(d_1, d_2, d_3, d_4)$ for
example. Note that~$(d_1, d_2, d_2, d_1)$ is also a loop, as well
as~$(d_1, d_1)$.

Loops on~$\g$ form a set~$L(\g)$.
%
%
%
We have reached the definition of a {\em cell complex} (or~$2$-cell
complex, for emphasis). This consists of \begin{itemize}
\item a bigraph~$\g$,
\item a set~$F$, the elements of which we call the {\em faces},
\item a map~$\partial \colon F \to L(\g)$, called the boundary map.
\end{itemize}



The definition of morphisms between cell complexes will wait a little.

A cell complex~$\cell$ also has a topological
realization~$\topo{\cell}$: briefly, one attaches closed discs to the
space~$\topo{\g}$ using the specified boundary maps. In more details,
for each~$f \in F$ we pick a copy~$\D_f$ of the unit disc
\[ \D = \{ z \in \C : |z| \le 1 \} \, .   \]
Consider then 
\[ Z_0 = \coprod_{f \in F} \D_f  \]
and 
\[ Z = \topo{\g} ~\coprod  ~ Z_0 \, .   \]
We define~$\topo{\cell}$ to be the following identification space
of~$Z$, with the quotient topology. Fix~$f \in F$ and let~$\partial f
= (d_1, d_2, \ldots, d_{2n})$. We put~$\omega = e^{\frac{2 i \pi}
  {2n}} \in \D_f$. The discussion will be easier to understand with a
picture: 

\figurehere{0.5}{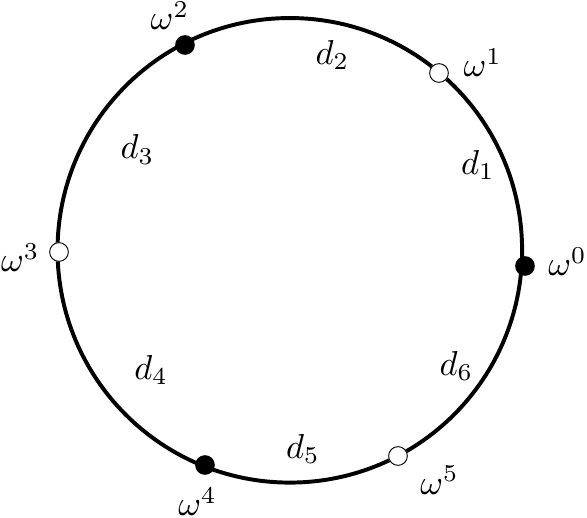}
%
The letters~$d_1, \ldots, d_6$ are simply here to indicate the
intended glueing. Let~$I= [0, 1]$ and consider the homeomorphism 
\[ h_i \colon I \longrightarrow [\omega^{2i}, \omega^{2i+1}] \, ,   \]
where $[\omega^{2i}, \omega^{2i+1}]$ denotes the circular arc
from~$\omega^{2i}$ to~$\omega^{2i+1}$, defined by~$h_i(t) = \omega^{2i
+ t}$. We shall combine it with the continuous map 
\[ g_i \colon I \longrightarrow \topo{\g}  \]
which is obtained as the identification~$I = I_{d_{2i+1}}$ followed by
the canonical map~$I_{d_{2i+1}} \to \topo{\g}$ (see the definition
of~$\topo{\g}$). We can now request, for all~$t \in I$, the
identification of~$g_i(t)$ and~$h_i(t)$, these being both points
of~$Z$.

Similarly there is an identification of the arc~$[\omega^{2i},
  \omega^{2i-1}]$ with the image of~$I_{d_{2i}}$. We prescribe no more
identifications, and this completes the definition of~$\topo{\cell}$.



\begin{ex} \label{ex-carre}
Let us return to the square as above. We add one face~$f$,
with~$\partial f = (d_1, d_2, d_3, d_4)$. We obtain a complex~$\cell$
such that~$\topo{\cell}$ is homeomorphic to the square~$[0, 1] \times
[0, 1]$, and which we represent as follows:

\figurehere{0.5}{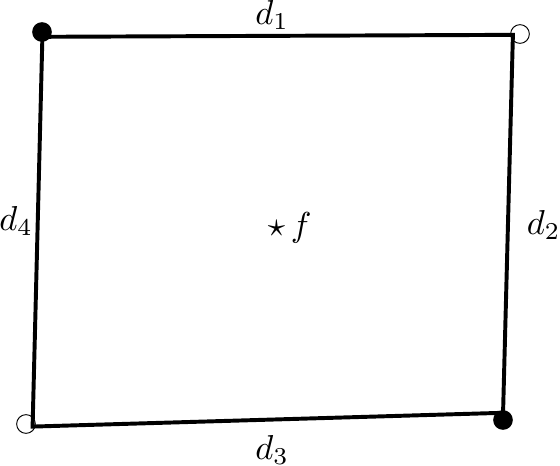}
We shall often place a~$\star$ inside the faces, even when they are
not labeled, to remind the reader to mentally fill in a disc.  The
reader is invited to contemplate how the complex obtained by taking,
say, $\partial f = (d_2, d_1, d_4, d_3)$ instead, produces indeed a
homeomorphic realization. These two complexes ought to be isomorphic,
when we have defined what isomorphisms are.
\end{ex}

\begin{ex} \label{ex-fundamental}
This example will be of more importance later than is immediately
apparent. Let~$B, W, D$ and~$F$ all have one element, say~$b, w, d$
and~$f$ respectively; and let~$\partial f = (d,
d)$. Then~$\topo{\cell}$ is homeomorphic to the sphere~$S^2$.

\image{\includegraphics[width=0.5\textwidth]{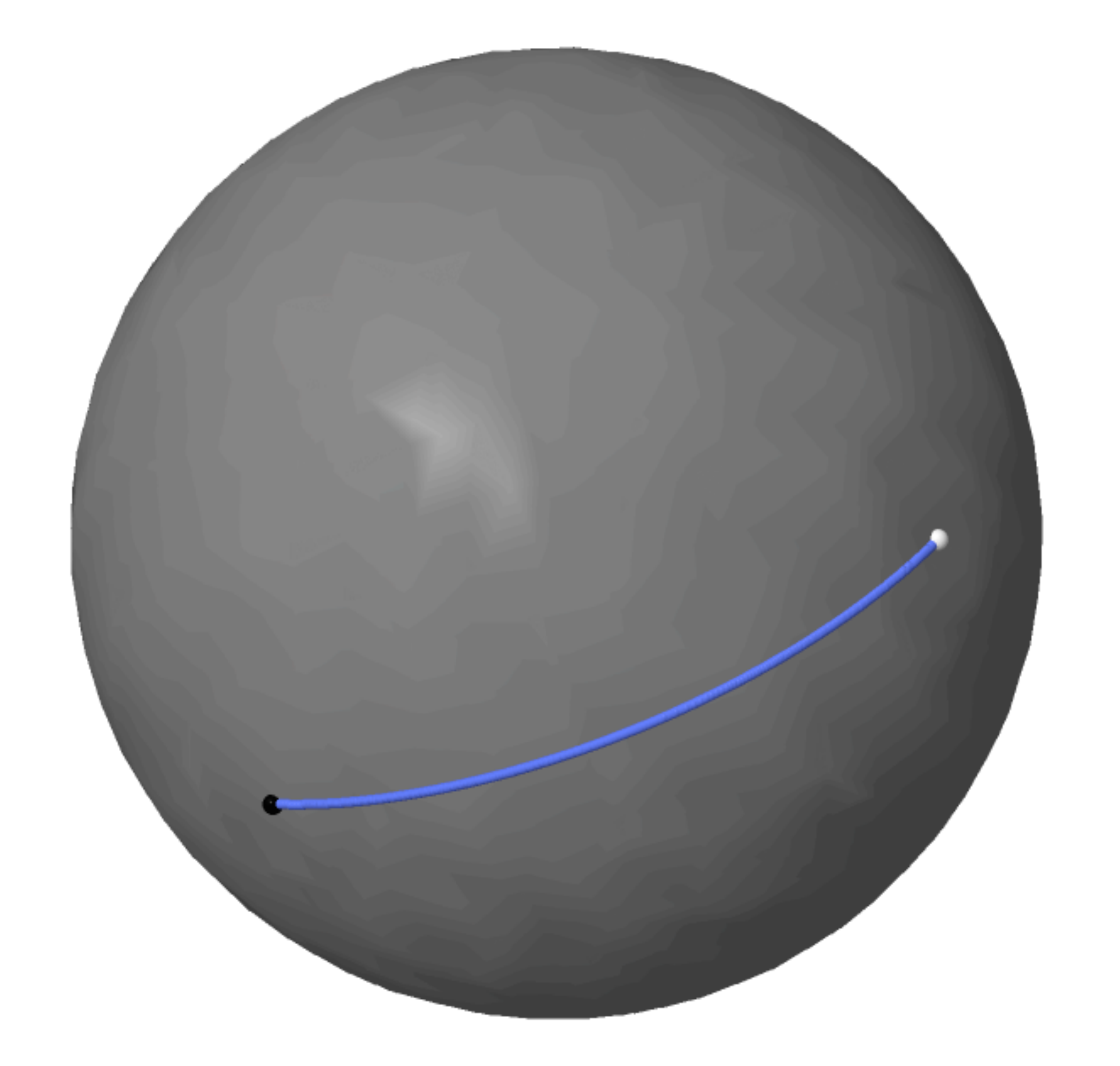}}

This example shows why we used discs rather than polygons: we may very
well have to deal with digons.
\end{ex}

\begin{ex} \label{ex-complexes-by-pictures}
It is possible to convey a great deal of information by pictures
alone, and with this example we explore such shorthands. Consider for
example:

\figurehere{0.5}{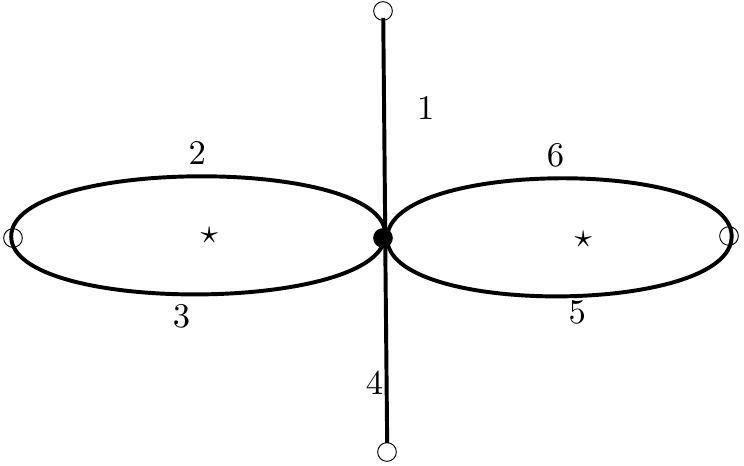}
Here we use integers to label the darts. We can see this picture
as depicting a cell complex with two faces, having boundary $(2, 3)$
and~$(5, 6)$ respectively. Should we choose to do so, there would be
little ambiguity in informing the reader that we mean for there to be
a third face ``on the outside'', hoping that the boundary~$(1, 1, 2,
3, 4, 4, 5, 6)$ (or equivalent) will be understood. The centre of that
face is placed ``at infinity'', that is, we think of the plane as the
sphere~$S^2$ with a point removed via stereographic projection, and
that point is the missing~$\star$. Of course with these three faces,
one has~$\topo{\cell}$ homeomorphic to~$S^2$.

Suppose we were to draw the following picture, and specify that there
is a third face ``at infinity'' (or ``on the outside''):

\figurehere{0.5}{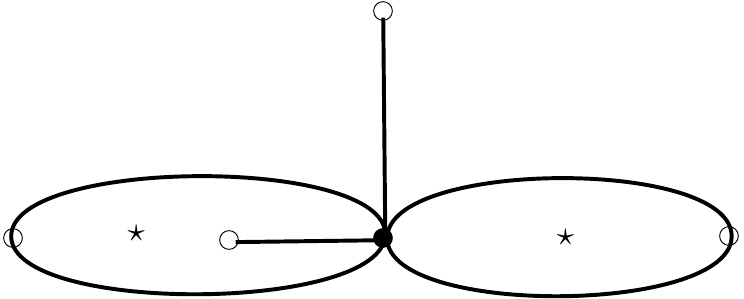}
This is probably enough information for the reader to understand
which cell complex we mean. (It has the same underlying bigraph as the
previous one, but the cell complexes are not isomorphic). The
topological realization, again a sphere, is represented below.

\image{\includegraphics[width=0.5\textwidth]{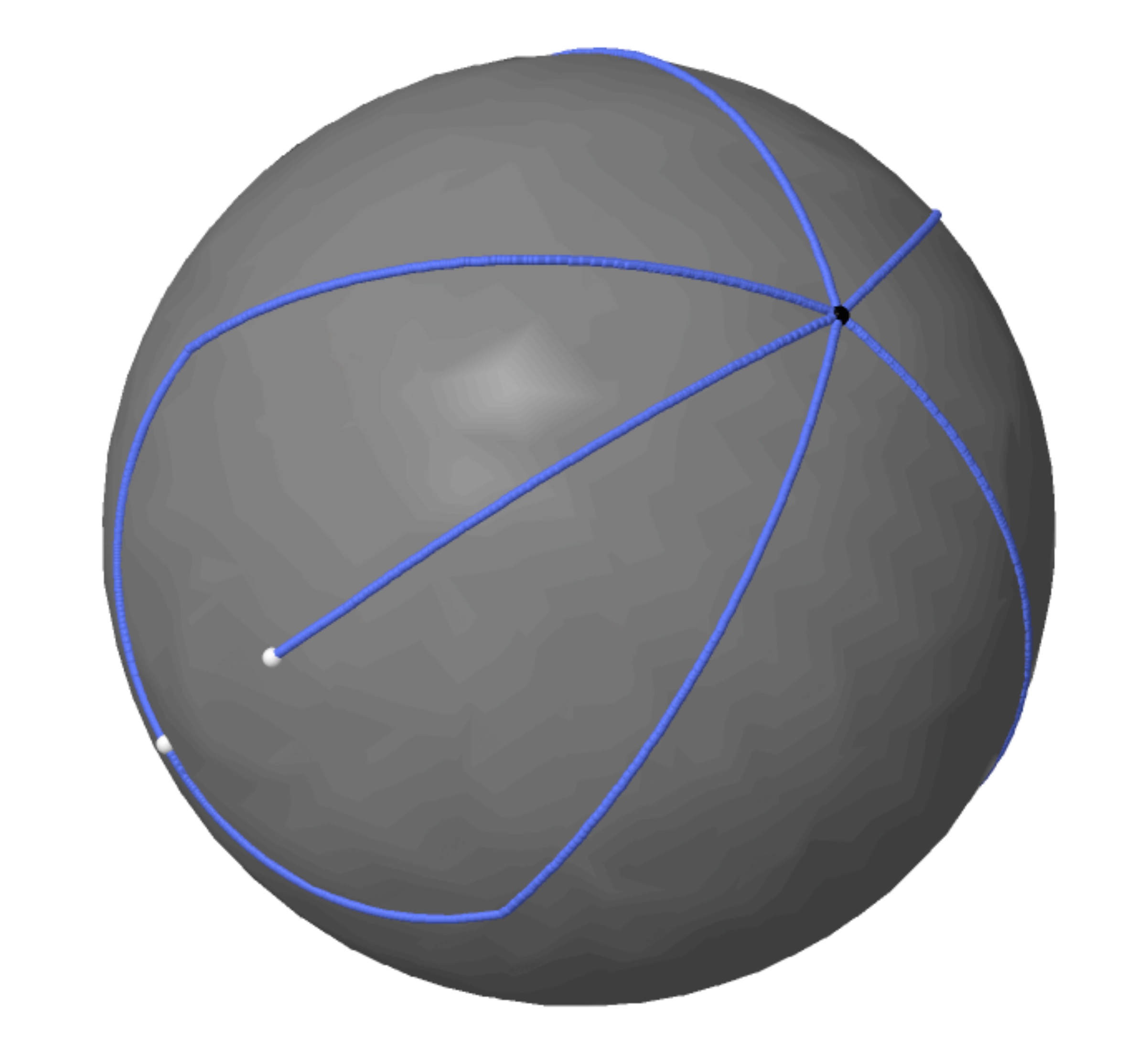}}
\end{ex}

\begin{ex}
It is harder to draw pictures in the following case. Take~$B= \{ b_1,
b_2, b_3 \}$ and~$W= \{ w_1, w_2, w_3 \}$, and add darts so that~$\g$
is ``the complete bipartite graph on~$3+3$ vertices'' : that is, place
a dart between each~$b_i$ and each~$w_j$, for~$1 \le i, j \le
3$. Since there are no multiple darts between any two vertices in this
bigraph, we can designate a dart by its endpoints; we may also
describe a loop by simply giving the list of vertices that it
visits. With this convention, we add four faces:
\begin{align*}
f_1 ~\textnormal{through}~ &  b_1, w_2, b_3, w_3, b_2, w_1 , \\
f_2 ~\textnormal{through}~ &  b_1, w_2, b_2, w_3 , \\
f_3 ~\textnormal{through}~ &  b_2, w_2, b_3, w_1 ,\\
f_4 ~\textnormal{through}~ &  b_1, w_3, b_3, w_1 \, . 
\end{align*}
(Each of these returns to its starting point in the end.) The
topological realization~$\topo{\cell}$ is homeomorphic to the
projective plane~$\r P^2$. We will show this with a picture:

\figurehere{0.5}{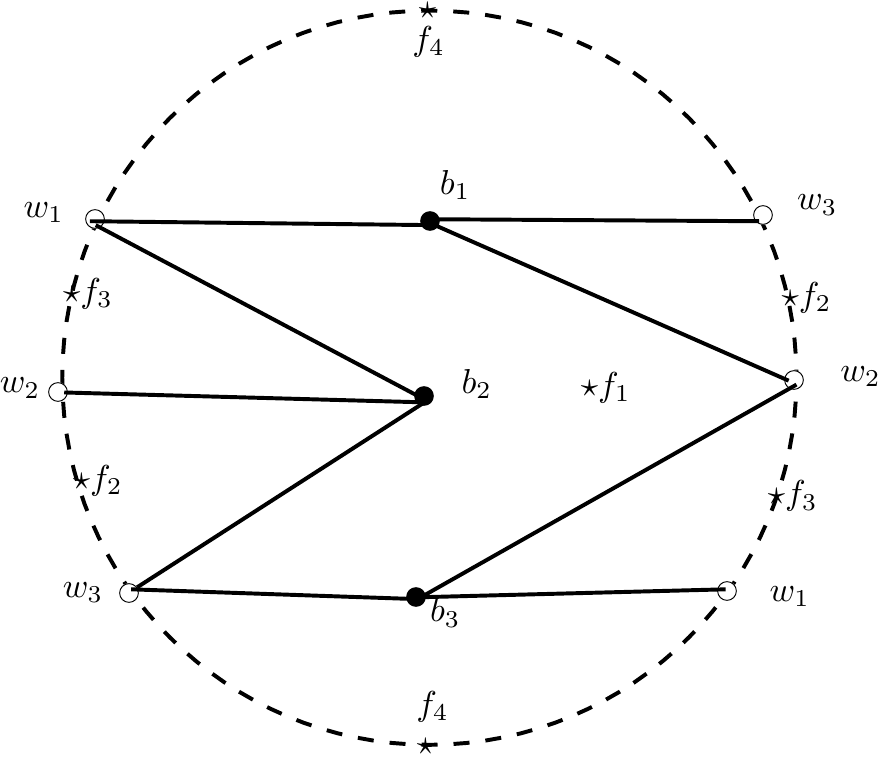}
Here we see~$\r P^2$ as the unit disc~$D$ with~$z$ identified
with~$-z$ whenever~$|z|= 1$; we caution that the dotted arcs,
indicating the boundary of the unit circle, are not darts.

\end{ex}

Here are some very basic properties of the geometric realization.

\begin{prop}
\begin{enumerate}
\item The space~$\topo{\cell}$ is connected if and only if~$\topo{\g}$ is.
\item The space~$\topo{\cell}$ is compact if and only if the complex
  is finite (ie $B$, $W$, $D$ and~$F$ are all finite).
\end{enumerate}
\end{prop}

\begin{proof}
(1) It is quite easy to prove this directly, after showing that each
  path on~$\topo{\cell}$ is homotopic to one lying on~$\topo{\g}$. The
  reader who has recognized that the space $\topo{\cell}$ is, by
  definition, the realization of a CW-complex, whose~$1$-skeleton
  is~$\topo{\g}$, will see the result as a consequence of the cellular
  approximation theorem (\cite{bredon}, Theorem 11.4).

(2) By construction there is a quotient map 
\[ q \colon K = Y \coprod B \coprod W \coprod Z_0 \longrightarrow
\topo{\cell} \, ,   \]
where the notation is as above. Clearly $K$ is compact
if the complex is finite, so~$q(K) = \topo{\cell}$ must be compact,
too, and we have proved that the condition is sufficient.

To see that it is necessary as well, one can argue that the map~$q$ is
proper, or else use elementary arguments as follows. We show that the
faces must be finite in number when~$\topo{\cell}$ is compact, and the
reader will do similarly with the vertices and darts.

For each~$f \in F$, consider the open set~$U_f \subset K$ whose
complement is the union of the closed discs of radius~$\frac{1} {2}$
in all the discs~$D_{f'}$ for~$f' \ne f$ (this complement is closed by
definition of the disjoint union topology). By definition of the
quotient topology $q(U_f)$ is open in~$\topo{\cell}$, and the various
open sets~$q(U_f)$ form a cover of~$\topo{\cell}$ (each~$q(U_f)$ is
obtained by removing a closed disc from each face of~$\topo{\cell}$
but one). By compactness, finitely many of them will cover the space,
and so finitely many of the open sets~$U_f$ will cover~$K$. It follows
that~$F$ is finite.
\end{proof}

\subsection{Morphisms between cell complexes; triangulations}

Let us start with a provisional definition: a {\em naive morphism}
between~$\cell= (\g, F, \partial)$ and~$\cell' = (\g', F', \partial')$
is given by a morphism~$\g \to \g'$ together with a map~$\Phi
\colon F \to F'$ such that $\partial' \Phi(f) = \Delta (\partial f)$
for~$f \in F$; here the map~$\Delta \colon D \to D'$ has been extended
to the set~$L(\g)$ in the obvious way. With this definition, it is
clear that naive morphisms induce continuous maps between the
topological realizations.

However this definition does not allow enough morphisms. Let us
examine this.

\begin{ex}
We return to example \ref{ex-carre}, so we consider the cell
complex~$\cell$ depicted below:

\figurehere{0.5}{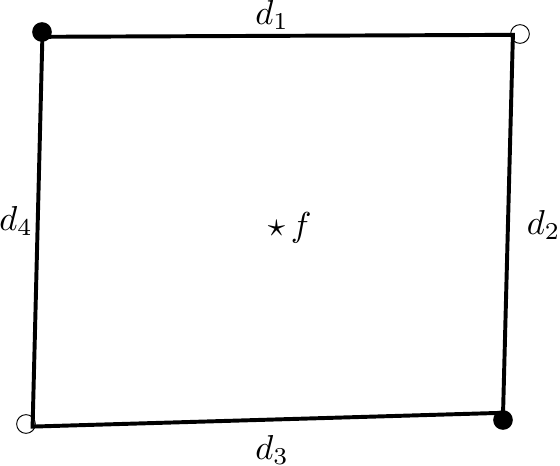}

Here $\partial f = (d_1, d_2, d_3, d_4)$. Now form a complex~$\cell'$
by changing only~$\partial$ to~$\partial'$, with $\partial' f = (d_2,
d_1, d_4, d_3)$. There is indeed a naive isomorphism between~$\cell$
and~$\cell'$, given by ``the reflection in the line joining the white
vertices''.

However, suppose now that we equip~$\cell$ with two faces~$f_1$
and~$f_2$ (leaving the bigraph unchanged) with~$\partial f_1= (d_1,
d_2, d_3, d_4) = \partial f_2$; then~$\topo{\cell}$ is the
sphere~$S^2$. On the other hand consider~$\cell'$ having the same
bigraph, and two faces satisfying~$\partial f'_1 = (d_1, d_2, d_3,
d_4)$ and~$\partial f'_2 = (d_2, d_1, d_3, d_4)$. Then it is readily
checked that there is no naive isomorphism between~$\cell$
and~$\cell'$. 

This is disappointing, as we would like to see these two as
essentially ``the same'' complexes. More generally we would like to
think of the boundaries of the faces in a cell complex as not having
a distinguished (black) starting point, and not having a particular
direction.
\end{ex}

The following better definition will be sufficient in many
situations. A {\em lax morphism} between~$\cell= (\g, F, \partial)$
and~$\cell' = (\g', F', \partial')$ is given by a morphism~$\g \to
\g'$ together with a map~$\Phi \colon F \to F'$ with the
following property. If~$f \in F$ with~$\partial f= (d_1, \ldots,
d_{2n})$, and if~$\partial' \Phi(f) = (d_1', \ldots , d'_{2m}) $, then
\[ \Delta  \left( \{ d_1, \ldots , d_{2n} \} \right) = \{ d'_1, \ldots,
d'_{2m}  \} \, ,  \]
where~$\Delta $ is the map~$D \to D'$. So naive morphisms are lax
morphisms, but not conversely.

\begin{ex}
Resuming the notation of the last example, the identity on~$\g$ and
the bijection~$F \to F'$, $f_1 \mapsto f_1'$, $f_2 \mapsto f_2'$,
together define a lax isomorphism between~$\cell$ and~$\cell'$.
\end{ex}

It is not immediate how lax morphisms can be used to induce continuous
maps. Moreover, the following phenomenon must be observed.

\begin{ex} \label{ex-projective-space}
We build a bigraph~$\g$ with only one black vertex, one white vertex,
and two darts~$d_1$ and~$d_2$ between them; $\topo{\g}$ is a
circle. Turn this into a cell complex~$\cell$ by adding one face~$f$
with~$\partial f=(d_1, d_2, d_1, d_2)$. The topological
realization~$\topo{\cell}$ is obtained by taking a copy of the unit
disc~$\D$, and identifying~$z$ and~$-z$ when~$|z|=1$: in other words,
$\topo{\cell}$ is the real projective plane~$\r P^2$.

Now consider the map~$z \mapsto -z$, from~$\D$ to itself, and factor
it through~$\r P^2$; it gives a self-homeomorphism
of~$\topo{\cell}$. The latter cannot possibly be induced by a lax
morphism, for it is the identity on~$\topo{\g}$: to define a
corresponding lax isomorphism we would have to define the self maps
of~$B, W$ and~$F$ to be the identity. Assuming that we had chosen a
procedure to get a continuous map from a lax morphism, surely the
identity would induce the identity.

However the said self-homeomorphism of~$\r P^2$ is simple enough that
we would like to see it corresponding to an isomorphism of~$\cell$.
\end{ex}

Our troubles seem to arise when repeated darts show up in the boundary
of a single face. We solve the problem by subdividing the faces,
obtaining the canonical {\em triangulation} of our objects.

Let~$\cell$ be a cell complex. We may triangulate the faces
of~$\topo{\cell}$ by adding a point in the interior of each face
(think of the point marked~$\star$ in the pictures), and connecting it
to the vertices on the boundary. More precisely, for each face~$f$,
with~$\partial f = (d_1, \ldots, d_{2n})$, we identify $2n$ subspaces
of~$\topo{\cell}$, each homeomorphic to a triangle, as the images
under the canonical quotient map of the sectors obtained on the unit
disc in the fashion described on the picture below for~$n=3$. We
denote them~$t_i^f$ with~$1 \le i \le 2n$.

\figurehere{0.4}{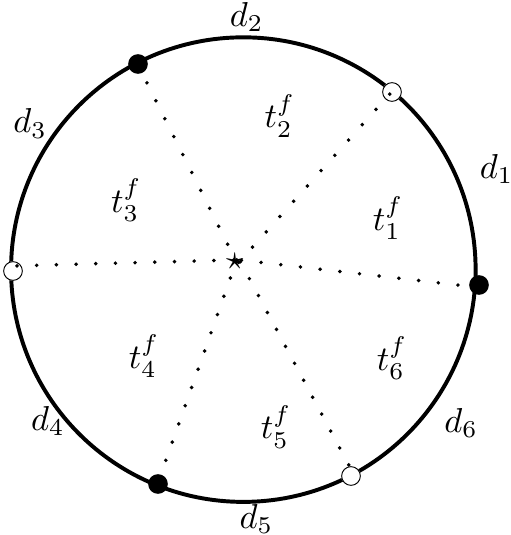}
(As before the labels~$d_i$ indicate the intended gluing, while
the sector bearing the name~$t_i^f$ will map to that subspace under
the quotient map.) The space~$\topo{\cell}$ is thus triangulated, yet
it is not necessarily (the realization of) a simplicial complex, as
distinct triangles may have the same set of vertices, as in
example~\ref{ex-fundamental}. This same example exhibits another
relevant pathology, namely that the disc corresponding to a face might
well map to something which is not homeomorphic to a disc anymore
(viz. the sphere), while the triangles actually cut the
space~$\topo{\cell}$ into ``easy'' pieces. It also has particularly
nice combinatorial properties.

We write~$T$ for the set of all triangles in the complex. We think
of~$T$ as an indexing set, much like~$B$, $W$, $D$ or~$F$. One can
choose to adopt a more combinatorial approach, letting~$t_1^f$,
$\ldots $, $t_{2n}^f$ be (distinct) symbols attached to the face~$f$
whose boundary is~$(d_1, \ldots, d_{2n})$, with~$T$ the set of all
symbols. There is a map~$\DD \colon T \to D$ which associates~$t_i^f$
with~$\DD( t_i^f) = d_i$, there is also a map~$\F \colon T \to F$
with~$\F (t_i^f) = f$. We will gradually use more and more geometric
terms when referring to the triangles, but it is always possible to
translate them into combinatorial relations.


Each~$t \in T$ has vertices which we may call~$\bullet$, $\circ$
and~$\star$ unambiguously. Its sides will be called~$\bullet - \circ$,
$\star - \bullet$ and~$\star - \circ$. Each~$t$ also has a
neighbouring triangle obtained by reflecting in the~$\star - \bullet$
side; call it~$a(t)$. Likewise, we may reflect in the~$\star - \circ$
side and obtain a neighbouring triangle, which we call~$c(t)$. In
other words, $T$ comes equipped with two permutations~$a$ and~$c$, of
order two and having no fixed points. (In particular if~$T$ is finite
it has even cardinality.) The notation~$a, c$ is standard, and there
is a third permutation~$b$ coming up soon. Later we will write~$t^a$
and~$t^c$ instead of~$a(t)$ and~$c(t)$, see
remark~\ref{rmk-convention-permutations}.

\begin{ex}
In example~\ref{ex-fundamental}, there are two triangles, say~$T=\{ 1,
2\}$, and~$a = c = $ the transposition~$(12)$. 
\end{ex}

\begin{ex} 
Let us consider the second complex from
example~\ref{ex-complexes-by-pictures}, that is let us have a look at 

\figurehere{0.5}{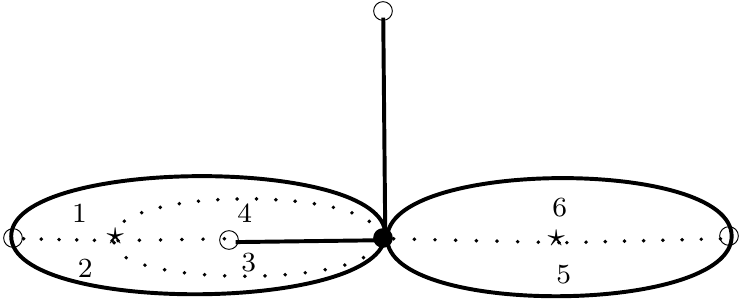}
Let us first assume that there is no ``outside face'', so let the
the triangles be numbered from~$1$ to~$6$. The permutation~$a$ is then 
\[ a = (14)(23)(56) \, ,   \]
while 
\[ c = (12)(34)(56) \, .   \]
If one adds a face at infinity, there are six new triangles, and the
permutations~$a$ and~$c$ change accordingly. We leave this as an exercise.

\end{ex}

We have at long last arrived at the official definition of a {\em
  morphism} between~$\cell= (\g, F, \partial)$ and~$\cell' = (\g', F',
\partial')$. We define this to be given by a morphism~$\g \to \g'$
(thus including a map~$\Delta \colon D \to D'$) and a map~$\Theta
\colon T \to T'$ which
\begin{enumerate}
\item verifies that for each triangle~$t$, one has~$\DD'( \Theta(t)) =
  \Delta (\DD(t))$,
\item is compatible with the permutations~$a$ and~$c$, that is~$\Theta
  (a(t)) = a(\Theta (t) ) $ and~$\Theta (c(t)) = c( \Theta (t) )$.
\end{enumerate} 
It is immediate that morphisms induce continuous maps between the
topological realizations. These continuous maps restrict to
homeomorphisms between the triangles.

Should this definition appear too complicated, we hasten to add:

\begin{lem}
Let~$\cell$ be a cell complex such that, for each face~$f$
with~$\partial f = (d_1, \ldots, d_{2n})$, the darts~$d_1, \ldots,
d_{2n}$ are distinct. Let~$\cell'$ be another cell complex with the
same property. Then any lax morphism between~$\cell$ and~$\cell'$
defines a unique morphism, characterized by the property
that~$\F(\Theta(t)) = \Phi(\F(t)) $ for every triangle~$t$.
\end{lem}
(Recall that lax morphisms have a map~$\Phi$ between the sets of
faces, and morphisms have a map~$\Theta $ between the sets of
triangles.)

Many cell complexes in practice satisfy the property stated in the
lemma, and for these we specify morphisms by giving maps~$B \to B'$,
$W\to W'$, $D \to D'$, and~$F \to F'$.

\begin{proof}
Any triangle~$t$ in~$\cell$ is now entirely determined by the
face~$\F(t)$ and the dart~$\DD(t)$; the same can be said of triangles
in~$\cell'$. So~$\Theta (t)$ must be defined as the only triangle~$t'$
such that~$\F(t')$ and~$\DD(t')$ are appropriate (in symbols $\F(t') =
\Phi(\F(t))$ and $\DD(t') = \Delta (\DD(t))$). The definition of lax
morphisms guarantees the existence of~$t'$. 

That~$\Theta $ is compatible with~$a$ and~$c$ is automatic
here. Indeed $a(t)$ is the only triangle such that~$\F(a(t)) = \F(t)$
and such that~$\DD(a(t))$ has the same black vertex as~$\DD(t)$. An
analogous property is true of both~$\Theta (a(t))$ and~$a(\Theta
(t))$, which must be equal. Likewise for~$c$.
\end{proof}

\begin{ex}
We come back to example~\ref{ex-projective-space}. The face~$f$ is
divided into~$4$ triangles, say~$t_1, t_2, t_3, t_4$. We can define a
self-isomorphism of~$\cell$ by~$\Theta (t_i) = t_{i+2}$ (indices mod
4), and everything else the identity.  The induced continuous
map~$\topo{\cell} \to \topo{\cell}$ is the one we were after (once
some identification of~$\topo{\cell}$ with~$\r P^2$ is made and fixed).
\end{ex}

We are certainly {\em not} claiming that any continuous
map~$\topo{\cell} \to \topo{\cell'}$, or even any homeomorphism, will
be induced by a morphism~$\cell \to \cell'$. For a silly example,
think of the map~$z \mapsto |z| z$ from the unit disc~$\D$ to itself,
which moves points a little closer to the origin; it is easy to
imagine a cell complex~$\cell$ with~$\topo{\cell} \cong \D$ such that
no self-isomorphism can induce that homeomorphism. In fact, whenever a
self-homeomorphism of~$\topo{\cell}$ leaves the triangles stable, then
the best approximation of it which we can produce with an
automorphism of~$\cell$ is the identity. 

However, the equivalence of categories below will show that we have
``enough'' morphisms, in a sense.

\subsection{Surfaces}

Here we adress a natural question: under what conditions on~$\cell$
is~$\topo{\cell}$ a surface (topological manifold of dimension 2), or
a surface-with-boundary?

A condition springing to mind is that each dart should be on the
boundary of precisely two faces (one or two faces for
surfaces-with-boundary). However this will not suffice, as we may well
end up with ``two discs touching at their centres'', that is, a
portion of~$\topo{\cell}$ might look like this:

\figurehere{0.2}{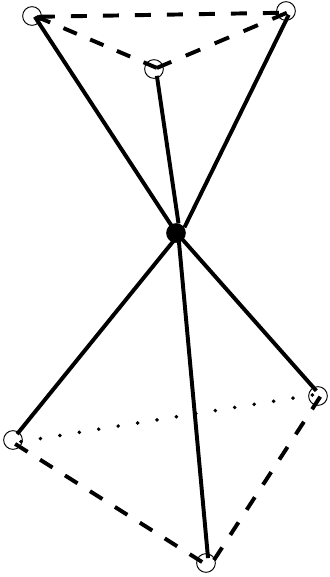}
(On this picture you are meant to see a little bit of six faces,
three at the top and three at the bottom, all touching at the black
vertex; each visible dart is on the boundary of precisely two faces,
yet~$\topo{\cell}$ is not a manifold near the black vertex.)

This is the only pathology that can really occur. To formulate the
condition on~$\cell$, here is some terminology. We say that a dart~$d$
is on the boundary of the face~$f$ if, of course, $d$ shows up in the
tuple~$\partial f$; since~$d$ may appear several times in~$\partial f$,
we define its {\em multiplicity} with respect to~$f$ accordingly. We
say that two darts~$d$ and~$d'$ appear consecutively in~$f$
if~$\partial f$ contains either the sequence~$d, d'$ or~$d', d$. In
this case~$d$ and~$d'$ have an endpoint in common; conversely if they
do have a common point, say a black one, then they appear
consecutively in~$f$ if and only if there are triangles~$t$ and~$t'$
with~$\F(t) = \F(t')= f$ such that~$d= \DD(t)$, $d'= \DD(t')$, and~$t,
t'$ are the image of one another under the permutation~$a$ (the
symmetry in the~$\star - \bullet$ side). Use~$c$ if the common point
is white.


Now let us fix a vertex, say a black one~$b \in B$. It may be
surprising at first that the condition that follows is in terms of
graphs; but it is the quickest way to phrase things. We
  take~$\B^{-1}(b)$, the set of darts whose black vertex is~$b$, as
  the set of vertices of a graph~$\cell_b$, and called the {\em
    connectivity graph} at~$b$. We place an edge between~$d$ and~$d'$
  whenever they appear consecutively in some face~$f$. Note that this
  may create loops in~$\cell_b$ as~$d=d'$ is not ruled out. 


We note that~$\cell_b$ has finitely many vertices. If we assume that
the darts in~$\cell$ are on the boundary of no more than two faces,
counting multiplicities, then it follows that each vertex in~$\cell_b$
is connected to at most two others (corresponding to the images
under~$a$ of the two triangles, at most, which may have the dart as a
side). Thus when~$\cell_b$ is connected, it is either a straight path
or a circle. 

There is a similar discussion involving a graph~$\cell_w$ for a white
vertex~$w \in W$.

Here is an example of complex with the connectivity graphs drawn in
red:

\figurehere{0.8}{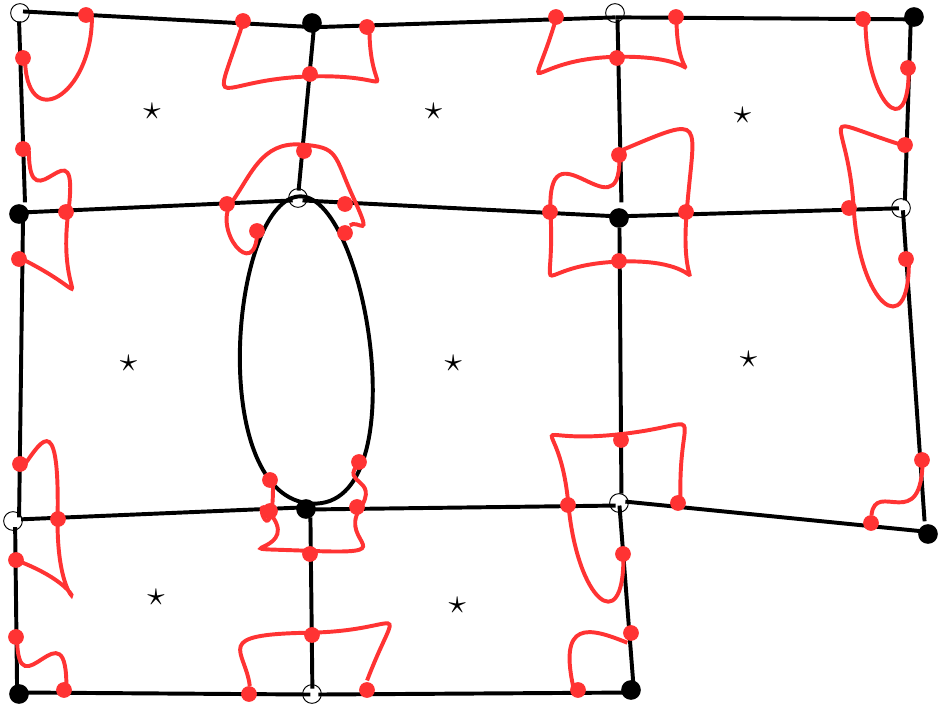}

\begin{prop} \label{prop-dessin-iff-conditions}
Let~$\cell$ be a complex. Then~$\topo{\cell}$ is a topological surface
if and only if the following conditions are met: \begin{enumerate}
\item each vertex has positive degree,
\item each dart is on the boundary of precisely two faces, counting
  multiplicities, 
\item all the connectivity graphs are connected.
\end{enumerate}

Necessary and sufficient conditions for~$\topo{\cell}$ to be a
surface-with-boundary are obtained by replacing (2) with the condition
that each dart is on the boundary of either one or two faces, counting
multiplicities.
\end{prop}

This should be obvious at this point, and is left as an exercise.

We have reached the most important definition in this section. A {\em
  dessin} is a complex~$\cell$ such that~$\topo{\cell}$ is a
surface (possibly with boundary). Whenever~$S$ is a topological
surface, a {\em dessin on~$S$} is a cell complex~$\cell$ together
with a specified homeomorphism~$h \colon \topo{\cell} \longrightarrow
S$. 
Several examples of dessins on the sphere have been
given.

Dessins have been called {\em hypermaps} and {\em dessins d'enfants}
in the literature. When all the white vertices have degree precisely
two, we call a dessin {\em clean}. Clean dessins are sometimes called
{\em maps} in the literature.

\subsection{More permutations}

Let~$\cell$ be a dessin. Each triangle~$t \in T$ determines a dart~$d=
\DD(t)$, and $d$ belongs to one or two triangles (exactly two
when~$\topo{\cell}$ has no boundary). We may thus define a
permutation~$b$ of~$T$ by requiring
\[ b(t) = \left\{ \begin{array}{l}
t ~\textnormal{if no other triangle has}~d ~\textnormal{as a
side} \, ,  \\
t' ~\textnormal{if}~ t' ~\textnormal{has}~d ~\textnormal{as a side
  and}~t' \ne t \, . 
\end{array}\right. \]

\begin{thm} \label{thm-complexes-same-as-permutations}
Let~$T$ be a finite set endowed with three permutations~$a$, $b$, $c$,
each of order two, such that~$a$ and~$c$ have no fixed points. Then
there exists a dessin~$\cell$, unique up to unique isomorphism,
such that~$T$ and~$a, b, c$ can be identified with the set of
triangles of~$\cell$ with the permutations described above.
\end{thm}

Later we will rephrase this as an equivalence of categories (with the
proof below containing all that is necessary).

\begin{rmk} \label{rmk-convention-permutations}
It is time for us to adopt a convention about groups of
permutations. If~$X$ is any set, and~$S(X)$ is the set of permutations
of~$X$, there are (at least) two naturals ways of turning~$S(X)$ into
a group. When~$\sigma, \tau \in S(X)$, we choose to define~$\sigma
\tau $ to be the permutation~$x \mapsto \tau (\sigma
(x))$. Accordingly, we will write~$x^\sigma$ instead of~$\sigma (x)$,
so as to obtain the formula~$x^{\sigma \tau } = (x^\sigma)^\tau$. 

With this convention the group~$S(X)$ acts on~$X$ {\em on the
  right}. This will simplify the discussion later when we bring in
covering spaces (personal preference is also involved here). 
\end{rmk}

\begin{proof}
Let~$G$ be the group of permutations of~$T$ generated by~$a$, $b$,
and~$c$, let~$G_{ab}$ be the subgroup generated by~$a$ and~$b$ alone,
and similarly define~$G_{bc}$, $G_{ac}$, $G_a$, $G_b$ and~$G_c$. Now
put 
\[ B= T/ G_{ab} \, , \qquad W= T/G_{bc} \, , \qquad D = T/G_b \, ,
\qquad F = T / G_{ac}  \, .  \]
The maps~$\B \colon D \to B$ and~$\W \colon D \to W$ are taken to be
the obvious ones, and we already have a bigraph~$\g$. It remains to
define the boundary map~$\partial \colon F \to L(\g)$ in order to define a
cell complex.

So let~$f \in F$, and let~$t \in T$ represent~$f$ (the different
choices we can make for~$t$ will all lead to isomorphic
complexes). Consider the elements~$t$, $t^c$, $t^{ca}$, $t^{cac}$,
$t^{caca}$, $\ldots$, alternating between~$a$ and~$c$.  Since~$T$ is
finite, there can be only finitely many distinct points created by
this process. Using the fact that~$a$ and~$c$ are of order two, and
without fixed points, it is a simple exercise to check that the
following list exhausts the orbit of~$t$ under~$G_{ac}$:
\[ t, t^c, t^{ca}, \ldots,  t^{cac \cdots acacac} \, .   \]
(There is an even number of elements, and the last one ends with
a~$c$.) We then let~$\partial f = (d_1, \ldots, d_{2n})$, where~$d_1$,
$d_2$, $\ldots $ is the~$G_b$-orbit of~$t$, $t^c$, $\ldots $

We have thus defined a cell complex~$\cell$ out of~$T$ together
with~$a$, $b$ and~$c$. It is a matter of checking the definitions to
verify that~$T$ can be identified with the set of triangles
of~$\cell$, in a way that is compatible with all the structure -- in
particular, the map~$T \to T/G_b$ is the map~$\DD$ which to a
triangle~$t$ associates the unique dart which is a side of~$t$, and
from the fact that~$b$ has order two we see that~$\cell$ satisfies
condition (2) of proposition~\ref{prop-dessin-iff-conditions}
(while (1) is obvious).

Let us examine condition (3). Any two darts in~$\cell$ having the same
black endpoint in~$B$ can be represented mod~$G_b$ respectively
by~$t$ and~$t^w$ where~$w$ is a word in~$a$ and~$b$. As we read the letters of~$w$ from left to right and think of the successive darts obtained from~$t$, each occurrence of~$a$ replaces a dart with a {\em consecutive} one, by
definition; occurrences of~$b$ do not change the dart. So~$\cell_b$ is
connected, and~$\cell$ is a dessin.

The uniqueness statement, to which we turn, is almost tautological
given our definition of morphisms. Suppose~$\cell$ and~$\cell'$ are
dessins with sets of triangles written~$T_\cell$
and~$T_{\cell'}$, such that there are equivariant bijections~$\iota
\colon T_\cell \to T$ and~$\iota ' \colon T_{\cell'} \to
T$. Then~$\Theta = (\iota ')^{-1} \circ \iota $ is an equivariant
bijection between~$T_\cell$ and~$T_{\cell'}$. Since~$B$, $W$ and~$D$
can be identified with certain orbits within~$T_{\cell}$, and
similarly with~$B'$, $W'$ and~$D'$, the maps~$B \to B'$, $W \to W'$
and~$D \to D'$ must and can be defined as being induced from~$\Theta
$. Hence there is a unique isomorphism between~$\cell$ and~$\cell'$. 
\end{proof}

We have learned something in the course of this proof:

\begin{coro}[of the proof of
    theorem~\ref{thm-complexes-same-as-permutations}] \label{coro-morphisms-same-equiv-maps}
Let~$\cell$ and~$\cell'$ be dessins. Then a morphism~$\cell \to
\cell'$ defines, and is uniquely defined by, a map~$\Theta \colon T
\to T'$ which is compatible with the permutations~$a, b$ and~$c$.
\end{coro}

\begin{proof}
By definition a morphism furnishes a map~$\Theta \colon T \to T'$
which is compatible with~$a$ and~$c$, and satisfies an extra condition
of compatibility with~$\DD$; however given the definition of~$b$, this
condition is equivalent to the equivariance of~$T$ with respect
to~$b$. 

Conversely if we only have~$\Theta $, equivariant with respect to all
three of~$a$, $b$, $c$, we can complete it to a fully fledged
morphism~$\cell \to \cell'$ as in the last proof, identifying~$B$, $W$
and~$D$ with certain orbits in~$T$.
\end{proof}

The group~$G$ introduced in the proof will be called the {\em full
  cartographic group} of~$\cell$ (below we will define another group
called the cartographic group).

\begin{lem} \label{lem-connected-iff-transitive}
Let~$\cell$ be a compact dessin. Then~$\topo{\cell}$ is connected
if and only if the full cartographic group acts transitively on the
set of triangles.
\end{lem}

\begin{proof}
Let~$T_1$, $T_2$, $\ldots $ be the orbits of~$G$ in~$T$, and let~$X_i
\subset \topo{\cell}$ be the union of the triangles
in~$T_i$. Each~$X_i$ is compact as a finite union of compact
triangles, hence~$X_i$ is closed in~$\topo{\cell}$. Also,
$\topo{\cell}$ is the union of the~$X_i$'s, since a dessin does
not have isolated vertices (condition (1) above).

Thus we merely have to prove that the~$X_i$'s are disjoint. However
when two triangles intersect, they do so along an edge, and then an
element of~$G$ takes one to the other. 
\end{proof}

\subsection{Orientations} \label{subsec-orientations}

\begin{prop}
Let~$\cell$ be a compact, connected dessin. Then the
surface~$\topo{\cell}$ is orientable if and only if it is possible to
assign a colour to each triangle, black or white, in such a way that
two triangles having a side in common are never of the same colour.
\end{prop}

\begin{proof} We give a proof in the case when there is no boundary,
  leaving the general case as an exercise. We use some standard
  results in topology, first and foremost: $\topo{\cell}$ is
  orientable if and only if
\[ H_2(\topo{\cell}, \z) \ne 0 \, .   \]
To compute this group we use cellular homology. More precisely, we
exploit the CW-complex structure on~$\topo{\cell}$ for which the
two-cells are the triangles (of course this space also has a
CW-complex in which the two-cells are the faces, but this is not
relevant here). Recall from an earlier remark that simplicial homology
is not directly applicable.

We need to orient the triangles, and thus declare that the positive
orientation is~$\star - \bullet - \circ$; likewise, we decide to
orient the 1-cells in such a fashion that~$\star - \bullet$, $\bullet
- \circ$ and $\circ - \star$ are oriented from the first named 0-cell
to the second. Writing~$\partial$ for the boundary in cellular
homology, we have then
\[ \partial t = [\star - \bullet]~  +~ [\bullet - \circ] ~+~ [\circ -
  \star] \, , \tag{*}\]
in notation which we hope is suggestive.

So let us assume that there is a 2-chain 
\[ \sigma = \sum_{t \in T} \, n_t t \ne 0\, , \tag{**}   \]
where~$n_t \in \z$, such that~$\partial \sigma = 0$. Suppose~$t$ is
such that~$n_t \ne 0$. From (*), we know the coefficients of the
neighbours of~$t$ in~$\sigma $, namely 
\[ n_{t^a} = n_{t^b} = n_{t^c} =- n_t \, .   \]
Since the full cartographic group acts transitively on~$T$ by the last
lemma, it follows that for each~$t' \in T$, the coefficient~$n_{t'}$
is determined by~$n_t$, and in fact~$n_{t'} = \pm n_t$.

Now let triangles~$t'$ such that~$n_{t'} > 0$ be black, and let the
others be white. We have coloured the triangles as requested. The
converse is no more difficult: given the colours, let~$n_t =1$ if~$t$ is
black and $-1$ otherwise. Then the 2-chain defined by (**) is non-zero
and has zero boundary, so the homology is non-zero.
\end{proof}

When~$\topo{\cell}$ is orientable, we will call an {\em orientation}
of~$\cell$ a colouring as above; there are precisely two orientations
on a connected, orientable dessin. An isomorphism will be said to {\em
  preserve orientations} when it sends black triangles to black
triangles. Note the following:

\begin{lem}
A morphism~$\cell \to \cell'$, where~$\cell$ and~$\cell'$ are
oriented dessins, preserves the orientations if and only
if~$\Theta $ sends black triangles to black triangles, and white
triangles to white triangles.
\end{lem} 

\subsection{More permutations}

Suppose that~$\cell$ is a dessin, and suppose that the
surface~$\topo{\cell}$ is oriented, and has no boundary. Then each
dart is the intersection of precisely two triangles, one black and one
white. The next remark is worth stating as a lemma for emphasis:

\begin{lem}
When~$\cell$ is oriented, without boundary, there is a bijection
between the darts and black triangles.
\end{lem}

Of course there is also a bijection between the darts and the white
triangles, on which we comment below.

Now consider the permutations~$\sigma = ab$, $\alpha =bc$ and~$\phi =
ca$. Each preserves the subset of~$T$ comprised by the black
triangles, so we may see~$\sigma, \alpha $ and~$\phi$ as permutation
of~$D$. It is immediate that they satisfy~$\sigma \alpha \phi = 1$,
the identity permutation.

Let us draw a little picture to get a geometric understanding of these
permutations. We adopt the following convention: when we draw a
portion of an oriented dessin, we represent the black triangles
in such a way that going from~$\star$ to~$\bullet$ to~$\circ$ rotates
us counterclockwise. (If we arrange this for one black triangle, and
the portion of the dessin really is planar, that is embeds into
the plane, then all black triangles will have this property).

\figurehere{0.8}{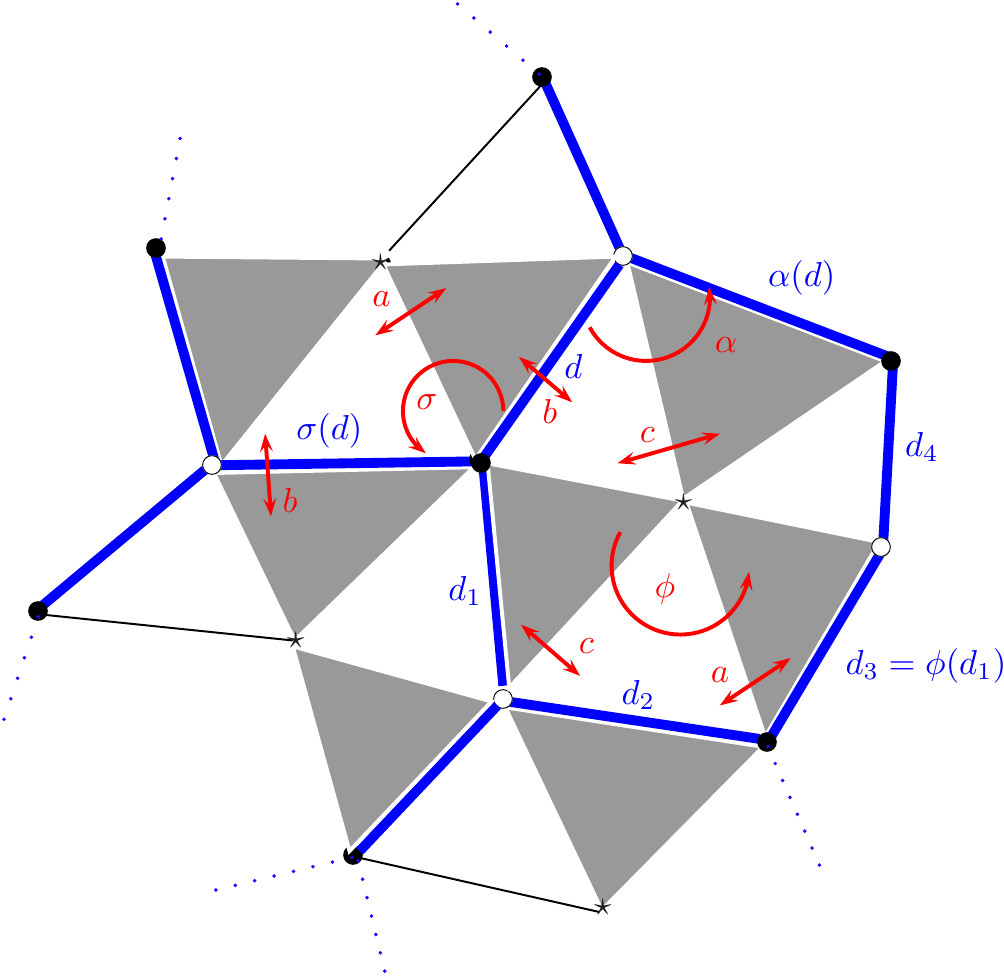}
 (Recall our convention on permutations as per
remark~\ref{rmk-convention-permutations}.)

On this picture, we see that our intuition for~$\sigma $ should be
that it takes a dart to the next one in the rotation around its black
vertex, going counterclockwise. Likewise~$\alpha $ is interpreted as
the rotation around the white vertex of the dart. As for~$\phi$, seen
as a permutation of~$T$, it takes a black triangle to the next one on
the same face, going counterclockwise. This can be made into more than
just an intuition: if~$\partial f = (d_1, \ldots, d_{2n})$, {\em and}
if~$t_i^f$ is black, then $\phi(d_i) = d_{i+2}$. Note that if the
triangle~$t_i^f$ is white, then~$\phi$ takes it to~$t_{i-2}^f$. In particular if one changes the orientation of the
dessin, the rotation~$\phi$ changes direction, as do~$\sigma $
and~$\alpha $.

This is also reflected algebraically in the relation~$b^{-1} \sigma b
= \sigma^{-1}$ (which translates the fact that~$a^2 = 1$): conjugating
by~$b$ amounts to swapping the roles of the black and white triangles
(or to identifying~$D$ with the white triangles instead of the blacks),
and that turns~$\sigma $ into~$\sigma^{-1}$. This relation is
important in the proof of the following.

\begin{thm} \label{thm-oriented-cart-same-as-permutations}
Let~$D$ be a finite set endowed with three permutations~$\sigma$,
$\alpha $, $\phi$ such that~$\sigma \alpha \phi= 1$. Then there exists
a dessin~$\cell$, oriented and without boundary, unique up to
unique orientation-preserving isomorphism, such that~$D$ and $\sigma
$, $\alpha $, $\phi$ can be identified with the set of darts
of~$\cell$ with the permutations described above.
\end{thm}

\begin{proof}
Let~$T = D \times \{ \pm 1 \}$. We extend~$\sigma $ to a
permutation~$\bar \sigma $ on~$T$ by the formula 
\[ \bar \sigma (d, \varepsilon ) = (\sigma ^{\varepsilon }(d),
\varepsilon ) \, ,   \]
and likewise~$ \alpha $ induces~$\bar \alpha $ on~$T$ by 
\[ \bar \alpha (d, \varepsilon ) = (\alpha ^{\varepsilon }(d),
\varepsilon ) \, .   \]
We also define a permutation~$b$ of~$T$ by 
\[ b(d, \varepsilon ) = (d, -\varepsilon ) \, .   \]
Putting~$a= \bar \sigma b$ and~$c = \bar \alpha b$, it is immediate
that~$a$ and~$c$ are of order~$2$ and have no fixed points.

By theorem~\ref{thm-complexes-same-as-permutations}, the set~$T$
together with~$a$, $b$ and $c$ defines a
dessin~$\cell$. Since~$b$ has no fixed points, $\cell$ has no
boundary. Calling the triangles in~$D \times \{ 1 \}$ black, and those
in~$D \times \{ -1 \}$ white, we see that~$\cell$ is naturally
oriented. 

The remaining statements are straightforward to prove.
\end{proof}

\begin{rmk} \label{rmk-euler-char}
We point out that one may prove
theorem~\ref{thm-oriented-cart-same-as-permutations} without appealing
to theorem~\ref{thm-complexes-same-as-permutations} first: one can
identify~$B$, resp~$W$, resp~$F$, with the cycles of~$\sigma $,
resp.\ $\alpha $, resp\ $\phi$, and proceed from there. We leave this
to the reader.

In particular, we may identify the topological surface~$\topo{\cell}$
easily: since it is compact, orientable, and without boundary, it is
determined by its genus or its Euler characteristic. The latter is
\[ \chi(\topo{\cell}) = n_\sigma + n_\alpha  - n + n_\phi \, ,  \]
where~$n$ is the cardinality of~$D$ (the number of darts),
while~$n_\sigma $, resp.\ $n_\alpha $, resp.\ $n_\phi$ is the number
of cycles of~$\sigma $, resp.\ $\alpha $, resp.\ $\phi$.
\end{rmk}

Note that the group of permutations of~$D$ generated by~$\sigma $,
$\alpha $ and~$\phi$ is called the {\em cartographic group}
of~$\cell$, or sometimes the {\em monodromy group}. 

\subsection{Categories}

Next we promote theorems~\ref{thm-complexes-same-as-permutations}
and~\ref{thm-oriented-cart-same-as-permutations} to equivalence of
categories. We write~$\cat$ for the category whose objects are
compact, oriented dessins without boundary, and whose morphisms are
the orientation-preserving maps of cell complexes. Also, $\ucat$ will
be the category whose objects are compact dessins without boundary
(possibly on non-orientable surfaces), and whose morphisms are all
morphisms of cell complexes.


We leave to the reader the task of proving the next theorem based
on theorem~\ref{thm-complexes-same-as-permutations} and
corollary~\ref{coro-morphisms-same-equiv-maps}, as well as
theorem~\ref{thm-oriented-cart-same-as-permutations}. 

\begin{thm} \label{thm-eq-cats-dessins-sets}
Consider the category~$\sets_{a,b,c}$ whose objects are the finite
sets~$T$ equipped with three distinguished permutations~$a$, $b$, $c$,
each of order two and having no fixed points, and whose
arrows are the equivariant maps. Then the assigment~$\cell \to T$
extends to an equivalence of categories between~$\ucat$
and~$\sets_{a, b, c}$.

Likewise, consider the category~$\mathfrak{Sets}_{\sigma , \alpha ,
  \phi}$ whose objects are the finite sets~$D$ equipped with three
distinguished permutations~$\sigma$, $ \alpha $, $\phi$
satisfying~$\sigma \alpha \phi= 1$, and whose arrows are the
equivariant maps. Then the assigment~$\cell \to D$ extends to an
equivalence of categories between~$\cat$ and~$\sets_{\sigma, \alpha,
  \phi}$.
\end{thm}

If one removes the requirement that~$b$ have no fixed point, in the
first part, one obtains a category equivalent to that of compact
dessins possibly with boundary.




\subsection{The isomorphism classes}

It is very easy for us now to describe the set of isomorphism classes
of dessins. There are different approaches in the literature and we
try to give several points of view.

\begin{prop} \label{prop-iso-classes}
\begin{enumerate}
\item A dessin~$\cell$ in~$\cat$ determines, and can be reconstructed
  from, an integer~$n$, a subgroup~$G$ of~$S_n$, and two distinguished
  generators~$\sigma $ and~$\alpha $ for~$G$. Two sets of data~$(n, G,
  \sigma, \alpha )$ and~$(n', G', \sigma ', \alpha ')$ determine
  isomorphic dessins if and only if~$n= n'$ and there is a conjugation
  in~$S_n$ taking~$\sigma $ to~$\sigma '$ and~$\alpha $ to~$\alpha '$
  (and in particular~$G$ to~$G'$).

\item The set of isomorphism classes of connected dessins in~$\cat$ is
  in bijection with the set of conjugacy classes of subgroups of
  finite index in the free group on two generators~$\langle \sigma,
  \alpha \rangle$.

\item Any connected dessin in~$\cat$ determines, and can reconstructed
  from, a finite group~$G$ with two distinguished generators~$\sigma $
  and~$\alpha $, and a subgroup~$H$ such that the intersection of all
  the conjugates of~$H$ in~$G$ is trivial. We obtain isomorphic
  dessins from~$(G, \sigma, \alpha, H )$ and~$(G', \sigma ', \alpha ',
  H')$ if and only if there is an isomorphism~$G \to G'$
  taking~$\sigma $ to~$\sigma'$, $\alpha $ to~$\alpha '$, and~$H$
  to a conjugate of~$H'$. 

\end{enumerate}
\end{prop}

\begin{proof}
At this point this is very easy. (1) is left as an exercise. Here are
some indications with (2): A connected object amounts to a finite
set~$X$ with a transitive, right action of~$\langle \sigma, \alpha
\rangle$, so~$X$ must be isomorphic to~$K \bs \langle \sigma, \alpha
\rangle$, where an isomorphism is obtained by choosing a base-point
in~$X$ (whose stabilizer is~$K$); different choices lead to conjugate
subgroups. (2) follows easily.

We turn to (3). It is clear that a connected object~$X$ is isomorphic
to~$H \bs G$ where~$G$ is the cartographic group and~$H$ is the
stabilizer of some point; elements in the intersection of all
conjugates of~$H$ stabilize all the points of~$X$, and so must be
trivial since~$G$ is by definition a subgroup of~$S(X)$. Conversely
any object of the form~$H \bs G$, with the actions of~$\sigma $
and~$\alpha $ by multiplication on the right, can be seen
in~$\sets_{\sigma, \alpha , \phi}$; it is connected since~$\sigma $
and~$\alpha $ generate~$G$; and its cartographic group must be~$G$
itself given the condition on~$H$. What is more, there is a canonical
map~$f \colon \langle \sigma, \alpha \rangle \to G$ sending~$\sigma $
and~$\alpha $ to the elements with the same name in~$G$, and the
inverse image~$K=f^{-1}(H)$ is the subgroup corresponding to the
dessin as in (2), while the intersection~$N$ of all the conjugates
of~$K$ is the kernel of~$f$. Thus we deduce the rest of (3) from (2).
\end{proof}

In \S\ref{sec-regularity} we shall come back to these questions (see
\S\ref{subsec-regularity-sets-with-perms} in particular). For the
moment let us add that it is common, in the literature, to pay special
attention to certain dessins for which some condition on the order
of~$\sigma $, $\alpha $ and~$\phi$ is prescribed. For example, those
interested in clean dessins very often require~$\alpha^2 = 1$. Assuming that
we are interested in the dessins for which, in addition, the order
of~$\sigma  $ divides a fixed integer~$k$, and that of~$\phi$
divides~$\ell$, then the objects are in bijection with the conjugacy
classes of subgroups of finite index in
\[ T_{k, \ell} =  \langle \sigma, \alpha, \phi : \sigma^k = \alpha^2 =
\phi^\ell = 1, \, \sigma \alpha \phi = 1 
\rangle \, , \]
usually called a {\em triangle group}. (We point out that, in doing
so, we include more than the clean dessins, for~$\alpha $ may have
fixed points.)

The variant in the unoriented case is as follows.

\begin{prop}
Consider the group~$\langle a, b, c : a^2 = b^2 = c^2 = 1 \rangle =
C_2 * C_2 * C_2$, the free product of three copies of the group of
order~$2$. The isomorphism classes of connected objects in~$\ucat$ are
in bijection with the conjugacy classes of subgroups~$H$ of~$C_2 * C_2
* C_2$ having finite index, and with the property that no conjugate
of~$H$ contains any of~$a$, $b$, $c$.
\end{prop}

Note that the last condition rephrases the fact that the actions
of~$a$, $b$ and~$c$ on~$H \bs C_2 * C_2 * C_2$ (on the right) have no
fixed points.

\section{Various categories equivalent to~$\cat$} \label{sec-eq-cats}

We proceed to describe a number of categories which are equivalent to
the category~$\cat$ of dessins -- the word dessin will
henceforth mean compact, oriented dessin without boundary. These
should be familiar to the reader, and there will be little need for
long descriptions of the objects and morphisms.

As for proving the equivalences, it will be a matter of quoting
celebrated results: the equivalence between covering spaces and sets
with an action of the fundamental group, the equivalence between
Riemann surfaces and their fields of meromorphic functions, the
equivalence between algebraic curves and their fields of rational
functions\ldots as well as some elementary Galois theory, which we
have taken from Völklein's book~\cite{helmut}. There is a little work
left for us, but we hope to convince the reader that the theory up to
here is relatively easy -- given the classics! What makes all this
quite deep is the combination of strong theorems in many different
branches of mathematics.

\subsection{Ramified covers}

Let~$S$ and~$R$ be compact topological surfaces. A map~$p \colon S \to
R$ is a {\em ramified cover} if there exists for each~$s \in S$ a
couple of charts, centered around~$s$ and~$p(s)$ respectively, in
which the map~$p$ becomes~$z \mapsto z^e$ for some integer~$e \ge 1$
called the {\em ramification index at~$s$} (this index at~$s$ is
well-defined, for~$p$ cannot look like~$z \mapsto z^{e'}$ for~$e'\ne
e$ in other charts, as can be seen by examining how-many-to-$1$ the
map is).

Examples are provided by complex surfaces: if~$S$ and~$R$ have complex
structures, and if~$p$ is analytic (holomorphic), then it is a basic
result from complex analysis that~$p$ must be a ramified cover in the
above sense (as long as it is not constant on any connected component
of~$S$). However we postpone all complex analysis for a while.

Instead, we can obtain examples (and in fact all examples) by the
following considerations. The set of~$s \in S$ such that the
ramification index~$e$ is~$> 1$ is visibly discrete in~$S$ and closed,
so it is finite by compactness. Its image in~$R$ under~$p$ is called
the {\em ramification set} and written~$R_r$. It follows that the
restriction 
\[ p \colon S \smallsetminus f^{-1}(R_r) \longrightarrow R
\smallsetminus R_r  \]
is a finite covering in the traditional sense. Now, it is a classical
result that one can go the other way around: namely, start with a
compact topological surface~$R$, let~$R_r$ denote a finite subset
of~$R$, and let~$p \colon U \longrightarrow R \smallsetminus R_r$
denote a finite covering map; then one can construct a compact
surface~$S$ together with a ramified cover~$\bar p \colon S \to R$
such that~$U$ identifies with~$\bar p^{-1}(R \smallsetminus R_r)$
and~$p$ identifies with the restriction of~$\bar p$. The ramification
set of~$\bar p$ is then contained in~$R_r$. See \S5 of \cite{helmut}
for all the details in the case~$R= \p$ (the general case is no
different).

Thus when the ramification set is constrained once and for all to be a
subset of a given finite set~$R_r$, ramified covers are in one-one
correspondence with covering maps. To make this more precise, let us
consider two ramified covers~$p \colon S \to R$ and~$p' \colon S' \to
R$ both having a ramification set contained in~$R_r$, and let us
define a morphism between them to be a continuous map~$h \colon S \to
S'$ such that~$p' \circ h = p$. Morphisms, of covering maps above~$R
\smallsetminus R_r$ are defined similarly. We may state:

\begin{thm} \label{thm-ramified-useless}
The category of finite coverings of~$R \smallsetminus R_r$ is
equivalent to the category of ramified covers of~$R$ with ramification
set included in~$R_r$.
\end{thm}

Now let us quote a well-known result from algebraic topology:

\begin{thm} \label{thm-eq-cats-coverings-sets}
Assume that~$R$ is connected, and pick a base point~$* \in
R\smallsetminus R_r$. The category of coverings of~$R\smallsetminus
R_r$ is equivalent to the category of right~$\pi_1(R \smallsetminus
R_r, *)$-sets. The functor giving the equivalence sends~$p\colon U \to
R \smallsetminus R_r$ to the fibre~$p^{-1}(*)$ with the monodromy
action.
\end{thm}


We shall now specialize to~$R = \p = S^2$ and~$R_r = \{ 0, 1, \infty
\}$. With the base point~$* = \frac{1} {2}$ (say), one
has~$\pi_1(\pminus, *) = \langle \sigma, \alpha \rangle$, the free
group on the two distinguished generators~$\sigma $ and~$\alpha $;
these are respectively the homotopy classes of the loops~$t \mapsto
\frac{1} {2} e^{2i \pi t}$ and~$t \mapsto 1 - \frac{1} {2} e^{2 i \pi
  t}$. The category of finite, right~$\pi_1(\pminus, *)$-sets is
precisely the category~$\sets_{\sigma, \alpha , \phi}$ already
mentioned. 

The following result combines
theorem~\ref{thm-eq-cats-dessins-sets} from the previous
section, theorem~\ref{thm-ramified-useless} above, as well as
theorem~\ref{thm-eq-cats-coverings-sets}:

\begin{thm}
The category~$\cat$ of oriented, compact dessins without
boundary is equivalent to the category~$\covs$ of ramified covers of~$\p$
having ramification set included in~$\{ 0, 1, \infty \}$. 
\end{thm}

\subsection{Geometric intuition}

There are shorter paths between dessins and ramified covers of
the sphere, that do not go via permutations. Here we comment on this
approach.

First, let us examine the following portion of an oriented
dessin:

\figurehere{0.3}{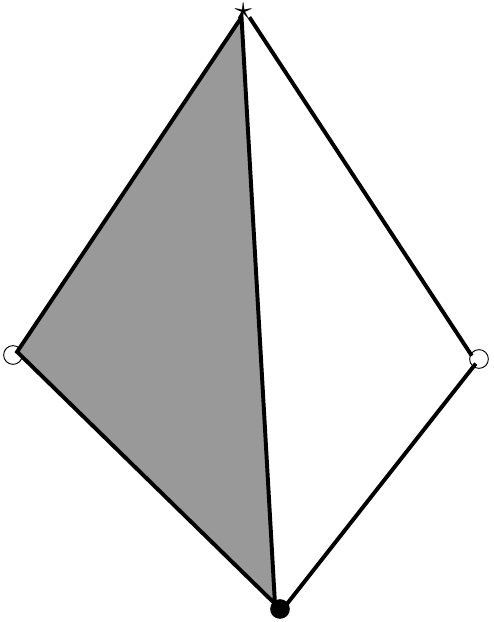}


Consider the identification space obtained from this by gluing the two
white vertices into one, and the four visible edges in pairs
accordingly. The result is a sphere; more precisely, we can
canonically find a homeomorphism with~$S^2$ sending~$\bullet$ to~$0$
and $\circ$ to~$1$, while~$\star$ is sent to~$\infty$. Doing this for
all pairs~$(t, t^a)$, where~$t$ is black, yields a single
map~$\topo{\cell} \to S^2$. The latter is the ramified cover
corresponding to~$\cell$ in the equivalence of categories above.

We will not prove this last claim in detail, nor will we rely on it in
the sequel. On the other hand, we do examine the reverse construction
more closely. In fact let us state:

\begin{prop} \label{prop-dessin-and-ramified-covers}
Let~$\cell$ correspond to~$p \colon S \to \p$ in the above equivalence
of categories. Then~$\topo{\cell} \cong S$, under a homeomorphism
taking~$\topo{\g}$ to the inverse image~$p^{-1}([0, 1])$.
\end{prop}

For the proof it will be convenient to have a modest lemma at our
disposal. It gives conditions under which a ramified cover~$p \colon S
\to R$, which must be locally of the form~$z \mapsto z^e$, can be
shown to be of this form over some given open set. We will write 
\[ \D = \{ z \in \C : |z| \le 1 \}  \]
as before, while 
\[ \DO = \{ z \in \C : |z| < 1 \} \, ,   \]
and 
\[ \DO' = \DO \smallsetminus \{ 0 \} \, .   \]

\begin{lem}
Let~$p \colon S \to R$ be a ramified cover between compact
surfaces. Let~$x \in R_r$, and let~$U$ be an open neighbourhood
of~$x$. We assume that~$U$ is homeomorphic to a disc, and that~$U \cap
R_r = \{ x \}$.

Then each connected component~$V$ of~$p^{-1}(U)$ contains one and only
one point of the fibre~$p^{-1}(x)$. Moreover, each~$V$ is itself
homeomorphic to a disc and there is a commutative diagram 
\[ \begin{CD}
\DO @>{\cong}>> V \\
@V{z \mapsto z^e}VV     @VVpV \\
\DO  @>{\cong}>> U
\end{CD}
  \]
\end{lem}

\begin{proof}
Let us start with the connected components of~$p^{-1}(U \smallsetminus
\{ x \})$. Let us form the pullback square
\[ \begin{CD}
E @>{\cong}>> p^{-1}(U \smallsetminus \{ x \}) \\
@V{\pi}VV     @VVpV \\
\DO'  @>{\cong}>> U \smallsetminus \{ x \}
\end{CD}
  \]
The map~$\pi$ is a covering map. The connected coverings of~$\DO'$ are
known of course: if~$W$ is a connected component of~$E$, then it can
be identified with~$\DO'$ itself, with~$\pi(z) = z^e$.

If~$V$ is as in the statement of the lemma, then it is a surface, so
it remains connected after removing finitely many points. It follows
that $$V \mapsto W = V \smallsetminus p^{-1}(x)$$ is well-defined, and
clearly injective, from the set of connected components of~$p^{-1}(U)$
to the set of connected components of~$p^{-1}(U \smallsetminus \{ x
\})$. 

Let us prove that~$V \mapsto W$ is surjective, so let~$W$ be a
component. Let~$K_n$ be the closure in~$S$ of 
\[ \{ z \in W= \DO' : |z| \le \frac{1} {n} \} \, .  \]
Since~$S$ is compact, there must be a point~$s\in S$ belonging to all
the closed subsets~$K_n$, for all~$n \ge 1$. It follows that~$p(s) =
x$. The point~$s$ must belong to some component~$V$; and by
definition~$s$ is in the closure of~$W$, so~$V \cap W \ne
\emptyset$. Thus the component~$V \smallsetminus p^{-1}(x)$ must
be~$W$.

We have established a bijection between the~$V$'s and the~$W$'s, and
in passing we have proved that each~$V$ contains at least an~$s$ such
that~$p(s) = x$. Let us show that it cannot contain two distinct such
points~$s$ and~$s'$. For this it is convenient to use the following
fact from covering space theory: given a covering~$c \colon X \to Y$
with~$X$ and~$Y$ both path-connected, there is no open subset~$\Omega
$ of~$X$, other than~$X$ itself, such that the restriction~$c \colon
\Omega \to Y$ is a covering of~$Y$. From this, we conclude that
if~$\Omega $ and~$\Omega '$ are open subsets of~$\DO'$, such that the
restriction of~$\pi$ to both of them yields a covering map, over the same
pointed disc~$Y$, then~$\Omega $ and~$\Omega '$ must be both equal
to~$X= \pi^{-1}(Y)$. If now~$s, s' \in V$ satisfy~$p(s) = p(s')= x$,
using the fact that~$p$ is a ramified cover we see that all the
neighbourhoods of~$s$ and~$s'$ must intersect, so~$s = s'$.

So we have a homeomorphism 
\[ h \colon W = \DO' \longrightarrow V \smallsetminus \{ s \}  \]
and we extend it to a map~$\bar h \colon \DO \to V$ by putting~$\bar
h(0) = s$. We see that this extension of~$h$ is again continuous, for
example by using that a neighbourhood of~$s$ in~$V$ mapping onto a
disc around~$x$ must correspond, under the bijection~$h$, to a disc
around~$0$, by the above ``fact''. This shows also that~$\bar h$ is an
open map, so it is a homeomorphism.
\end{proof}

\begin{proof}[Proof of proposition~\ref{prop-dessin-and-ramified-covers}]

Let us start with~$p \colon S \longrightarrow \p$, a ramified cover
with ramification in~$\{ 0, 1, \infty \}$, and let us build {\em some}
dessin~$\cell$. We will then prove that it is the dessin
corresponding to~$p$ in our equivalence of categories, so this proof
will provide a more explicit construction.

So let~$B = p^{-1}( 0 )$, $W= p^{-1}( 1 )$. There is no ramification
along~$(0, 1)$, and this space is simply-connected, so~$p^{-1}((0,
1))$ is a disjoint union of copies of~$(0, 1)$; we let~$D$ denote the
set of connected components of~$p^{-1}((0, 1))$.

For each~$b \in B$ we can find a neighbourhood~$U$ of~$b$ and a
neighbourhood~$V$ of~$0 \in \p$, both carrying charts onto discs,
within which $p$ looks like the map~$z \mapsto z^e$. Pick~$\varepsilon
$ such that~$[0, \varepsilon ) \subset V$; then the open set~$U$
  with~$p^{-1}([0, \varepsilon )) \cap U$ drawn on it looks like a
    disc with straight line segments connecting the centre to
    the~$e$-th roots of unity. Taking~$\varepsilon $ small enough for
    all~$b \in B$ at once, $p^{-1}([0, \varepsilon ))$ falls into
      connected components looking like stars and in bijection
      with~$B$. As a result, each~$d \in D$ determines a unique~$b \in
      B$, corresponding to the unique component that it
      intersects. This is~$\B(d)$; define~$\W(d)$ similarly.

We have defined a bigraph~$\g$, and it is clear that~$\topo{\g}$ can be identified with the inverse image~$p^{-1}([0, 1])$. We turn it into a cell complex now. Let~$F = p^{-1}( \infty )$. We apply the previous lemma to~$\p \smallsetminus [0, 1]$, which is an open subset in~$\p$ homeomorphic to a disc and containing only one ramification point, namely~$\infty$. By the lemma, we know that~$p^{-1}(\p \smallsetminus [0, 1])$ is a disjoint union of open discs, each containing just one element of~$F$. We need to be a little more precise in order to define~$\partial f$.


We consider the map~$h \colon \D \to \p$ constructed in two steps as
follows. First, let~$\D \to \D/ \!\! \sim$ be the quotient map that
identifies~$z$ and~$\bar z$ if and only if~$|z|= 1$; then, choose a
homeomorphism~$\D / \!\!  \sim \, \to \p$, satisfying~$1 \mapsto 0$,
$-1 \mapsto 1$, $0 \mapsto \infty$, and sending both circular arcs
from~$1$ to~$-1$ in~$\D$ to~$[0, 1]$. We think of~$h$ as the map~$\D
\to \topo{\cell}$ in example~\ref{ex-fundamental}. In~$\D$, we think
of~$1$ as a black vertex, of~$-1$ as a white vertex, of the circular
arcs just mentioned as darts, and of the two half-discs separated by
the real axis as black and white triangles.

\figurehere{0.25}{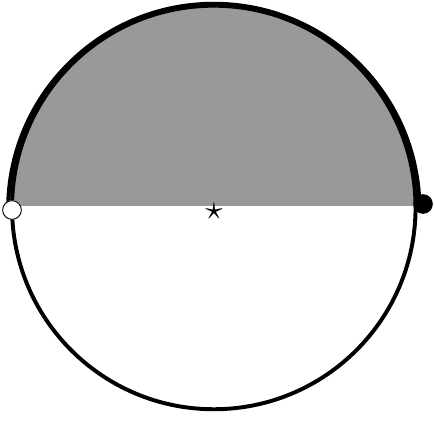}


Let~$\D^1 = \D \smallsetminus \{ 1, -1, 0 \}$ and in fact define~$\D^n
= \D \smallsetminus \{ \omega : \omega^{2n} = 1 \} \cup \{ 0
\}$. We emphasize that~$\D^n$ contains numbers of modulus~$1$. There
is a covering map~$\D^n \to \D^1$ given by~$z \mapsto
z^n$. Since~$\D^1$ retracts onto a circle, its fundamental group
is~$\z$, and we see that any connected covering of finite degree~$n$
must actually be of this form.

Now let~$S' \to \pminus$ be the covering defined by~$p$. Let us
construct a pull-back square
\[ \begin{CD}
E @>\theta >> S' \\
@VqVV    @VV{p}V \\
\D^1 @>h>> \pminus
\end{CD}
  \]
Here~$E \to \D^1$ is a finite covering map, so each connected
component of~$E$ can be identified with~$\D^n$ for some~$n$, while the
map~$q$ becomes~$z \mapsto z^n$. These components are in bijection
with~$F$, so we write~$\D^n_f$ for~$f \in F$.

If~$\omega$ is a~$2n$-th root of unity, the circular arc~$(\omega^i,
\omega ^{i+1}) \subset \D^n_f$ is mapped onto a dart by the
map~$\theta \colon E \to S'$. This defines, for each face~$f$, a
sequence of darts which is~$\partial f$. This completes our
construction of a cell complex from a ramified cover of~$\p$. Note
that~$\theta \colon \D^n_f \to S'$ can be extended to a map~$\D \to S$,
clearly, and it follows easily that~$\topo{\cell}$ is homeomorphic
to~$S$ itself, or in other words that~$\cell$ is a dessin on~$S$.

It remains to prove that~$\cell$ is the dessin corresponding to
the ramified cover~$p$ in the equivalence of categories at hand. For
this we compare the induced actions. To~$\cell$ are attached two
permutations~$\sigma $ and~$\alpha $ of the set~$D$ of darts. Note
that~$D$ is here in bijection with the fibre~$p^{-1}(\frac{1} {2})$,
and taking~$\frac{1} {2}$ as base point we have the monodromy action
of~$\pi_1(\pminus) = \langle \sigma ', \alpha' \rangle$, defining the
permutations~$\sigma '$ and~$\alpha'$. We must prove that~$\sigma =
\sigma '$ and~$\alpha = \alpha '$. Here~$\sigma '$ and~$\alpha '$ are
the classes of the loops defined above (where we used the
notation~$\sigma $ and~$\alpha $ in anticipation). 

We will now use the fact (of which we say more after the proof)
that~$S$ can be endowed with a unique smooth structure and
orientation, such that~$p\colon S \to \p$ is smooth and
orientation-preserving. We use this first to obtain, for each dart, a
smooth parametrization~$\gamma \colon [0, 1] \to S$ such that~$p \circ
\gamma $ is the identity of~$[0, 1]$. Each dart belongs to two
triangles, and it now makes sense to talk about the triangle {\em on
  the left} of the dart as we travel along~$\gamma $. Colour it
black. We will prove that this is a colouring of the type considered
in~\S\ref{subsec-orientations}.

Pick~$b \in B$, and a centered chart~$\DO \to U$ onto a
neighbourhood~$U$ of~$b$, such that the map~$p$ when pulled-back
to~$\DO$ is~$z \mapsto z^e$. The monodromy action of~$\pi_1(\DO')$ on
the cover~$\DO' \to \DO'$ given by~$z \mapsto z^e$ is generated by the
counterclockwise rotation of angle~$\frac{2 \pi} {e}$. Now it is
possible for us to insist that the chart~$\DO \to U$ be
orientation-preserving, so ``counterclockwise'' can be safely
interpreted on~$S$ as well as~$\DO$. Let us draw a picture of~$U$
with~$p^{-1}([0, 1 )) \cap U$ on it, together with the
  triangles, for~$e= 4$.

\figurehere{0.3}{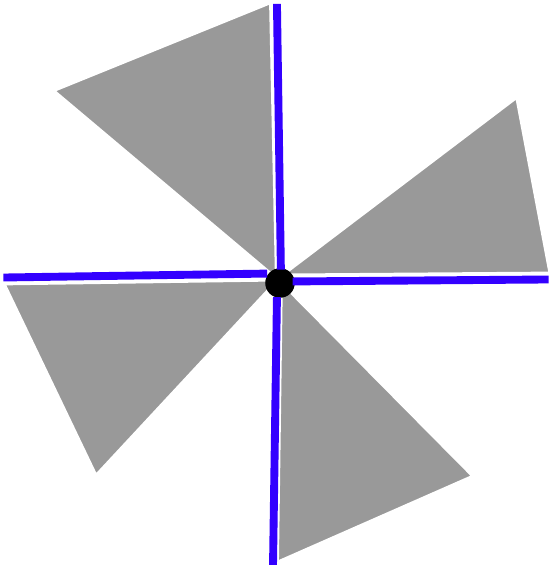}

The complement of the star-like subset of~$U$ given by~$p^{-1}([0,
  1 ))$ falls into connected components, each contained in a
  face; so two darts obtained by a rotation of angle~$\frac{2 \pi}
  {e}$ are on the boundary of the same face, and must be
  consecutive. The symmetry~$a$, that is the symmetry in the~$\star -
  \bullet$ side, is now clearly seen to exchange a black triangle with
  a white one. What is more, calling~$b$ as usual the symmetry in the
  darts, the permutation~$\sigma = ab$ sends a black triangle to its
  image under the rotation already mentioned. This is also the effect
  of the monodromy action, and~$\sigma = \sigma '$. 

Reasoning in the same fashion with white vertices, we see that~$c$,
the symmetry in the~$\star - \circ$ side, also exchanges triangles of
different colours. So the colouring indeed has the property that
neighbouring triangles are never of the same colour. That~$\alpha =
\alpha '$ is observed similarly. This concludes the proof. 
\end{proof}

\begin{ex}[Duality] \label{ex-duality}
The geometric intuition gained with this proposition and its proof may
clarify some arguments. Let~$\cell$ be a dessin, whose sets of
triangles and darts will be written~$T$ and~$D$, so that~$\cell$
defines the object~$(D, \sigma , \alpha, \phi )$ in~$\sets_{\sigma ,
  \alpha , \phi}$. Now let~$p \colon S \longrightarrow \p$ correspond
to~$\cell$. What is the dessin corresponding to~$1/p$ ? And what
is the object in~$\sets_{\sigma, \alpha, \phi}$ ?

Let us use the notation~$\cell'$, $T'$ and~$D'$. We can think
of~$\cell$ and~$\cell'$ as being drawn on the same surface~$S$. Zeroes
of~$1/p$ are poles of~$p$ and {\em vice-versa}, so black vertices are
exchanged with face centres, while the white vertices remain in
place. In fact, the most convenient property to observe is
that~$\cell$ and~$\cell'$ have exactly the same triangles, as
subspaces of~$S$, and we identify~$T = T'$. The~$\star - \circ$ sides
are promoted to darts.

The symmetries of~$T$ which we have called~$a, b$ and~$c$ become,
for~$\cell'$, the symmetries~$a'= a$, $b' = c$ and~$c'=b$ (simply look
at the definitions and exchange~$\star$ and~$\bullet$ throughout). It
follows that~$\sigma = ab$ becomes~$\sigma ' = a'b' = ac = \phi^{-1}$
and similarly one obtains~$\alpha ' = \alpha ^{-1}$ and~$\phi' =
\sigma^{-1}$. 

One must be careful, however. The object in~$\sets_{\sigma , \alpha,
  \phi}$ defined by~$1/p$, which we are after, is hidden behind one
more twist. The {\em black} triangles in~$T$ for~$\cell$ are those
mapping to the upper half plane under~$p$; the white triangles
for~$\cell$ are the black ones for~$\cell'$ as a result. Identifying
darts and black triangles, we see~$T$ as the disjoint union of~$D$
and~$D'$. While it is the case that~$\cell'$ corresponds to~$(D',
\phi^{-1}, \alpha^{-1}, \sigma^{-1})$ in~$\sets_{\sigma, \alpha,
  \phi}$, this notation is confusing since we tend to think
of~$\phi^{-1}$ as a map defined on either~$T$ or~$D$, when in fact it
is the induced map on~$D'$ which is considered here (in fact we should
write something like~$\phi^{-1} |_{D'}$). It is clearer to use for
example the map~$b' \colon D \to D'$ and transport the permutations
to~$D$, which is simply a conjugation. As already observed, this
``change of orientation'' amounts to taking inverses for~$\sigma '$
and~$\alpha '$.

The conclusion is that {\em replacing~$p$ by~$1/p$ takes the
  object~$(D, \sigma, \alpha, \phi)$ to the object~$(D, \phi, \alpha,
  \alpha^{-1} \sigma \alpha )$}.
\end{ex}

\begin{ex}[Change of colours] \label{ex-change-colours}
As an exercise, the reader will complete the following
outline. If~$\cell$ is represented by~$p \colon S \to \p$, with
corresponding object~$(D, \sigma, \alpha, \phi)$, then~$1 - p \colon S
\to \p$ corresponds to~$(D, \alpha, \sigma, \alpha \phi \alpha ^{-1}
)$. Indeed, $\cell$ and~$\cell'$ have the same triangles, as subsets
of~$S$, and the black triangles for~$\cell$ are precisely the white
ones for~$\cell'$ and {\em vice-versa}; the vertices of~$\cell'$ are
those of~$\cell$ with the colours exchanged, while the face centres
remain in place. (Informally~$\cell'$ is just that: the same
as~$\cell$ with the colours exchanged.) So~$c'= a$, $b'=b$ and~$a'=c$,
and~$\sigma ' = c \alpha c^{-1}$, $\alpha ' = b \sigma b^{-1}$, as
maps of~$T$. As maps of~$D$, using the bijection~$b \colon D \to D'$
to transport the maps induced on~$D'$, we end up with the permutations
announced.
\end{ex}

\subsection{Complex structures}

When~$p \colon S \to R$ is a ramified cover, and~$R$ is equipped with
a complex structure, there is a unique complex structure on~$S$
such that~$p$ is complex analytic (\cite{douady}, 6.1.10). Any
morphism between~$S$ and~$S'$, over~$R$, is then itself complex
analytic. Conversely if~$S$ and~$R$ both have complex structures, an
analytic map~$S \to R$ is a ramified cover as soon as it is not
constant on any connected component of~$S$.

We may state yet another equivalence of categories. Recall that an
analytic map~$S \to \p$ is called a meromorphic function on~$S$.

\begin{thm}
The category~$\cat$ is equivalent to the category~$\belyi$ of compact
Riemann surfaces with a meromorphic function whose ramification set is
contained in~$\{ 0, 1, \infty \}$.
\end{thm}

(The arrows considered are the maps above~$\p$.) A
pair~$(S, p)$ with~$p \colon S \to \p$ meromorphic, not ramified
outside of~$\{ 0, 1, \infty \}$, is often called a {\em Belyi pair},
while~$p$ is called a {\em Belyi map}.

\begin{ex} \label{ex-belyi-fractions}
Let us illustrate the results up to now with dessins on the sphere, so
let~$\cell$ be such that~$\topo{\cell}$ is homeomorphic to~$S^2$. By
the above, $\cell$ corresponds to a Riemann surface~$S$ equipped with
a Belyi map~$p \colon S \to \p$.

By proposition~\ref{prop-dessin-and-ramified-covers}, $S$ is
itself topologically a sphere. The uniformization theorem states that
there is a complex isomorphism~$\theta \colon \p \longrightarrow S$,
so we may as well replace~$S$ with~$\p$ equipped with~$F= p \circ \theta
$. Then~$(\p, F )$ is a Belyi pair isomorphic to~$(S, p)$.

Now $ F \colon \p \to \p$, which is complex analytic and not constant,
must be given by a rational fraction, as is classical. {\em The
  bigraph~$\g$ can be realized as the inverse image~$F^{-1}([0,
    1])$ where~$F\colon \p \longrightarrow \p$ is a rational
  fraction.} 

Let us take this opportunity to explain the terminology {\em dessins
  d'enfants} (children's drawings), and stress again some remarkable
features. By drawing a simple picture, we may as in
example~\ref{ex-complexes-by-pictures} give enough information to
describe a cell complex~$\cell$. Very often it is evident
that~$\topo{\cell}$ is a sphere, as we have seen in this example. What
the theory predicts is that we can find a rational fraction~$F$ such
that the drawing may be recovered as~$F^{-1}([0, 1])$. This works with
pretty much any planar, connected drawing that you can think of, and
gives these drawings a rigidified shape.

To be more precise, the fraction~$F$ is unique up to an isomorphism
of~$\p$, that is, up to precomposing with a Moebius
transformation. This allows for rotation and stretching, but still
some features will remain unchanged. For example the darts around a
given vertex will all have the same angle~$\frac{2\pi} {e}$ between
them, since~$F$ looks like~$z \mapsto z^e$ in conformal charts.

\end{ex}

\subsection{Fields of meromorphic functions}

When~$S$ is a compact, connected Riemann surface, one can consider all
the meromorphic functions on~$S$, comprising a field~$\m
(S)$. When~$S$ is not assumed connected, the meromorphic functions
form an {\em étale algebra}, still written~$\m (S)$: in this paper an
étale algebra is simply a direct sum of fields, here corresponding to
the connected components of~$S$. In what follows we shall almost
always have to deal with an {\em étale algebra over~$K$} where~$K$ is
some field, by which we mean an étale algebra which is also
a~$K$-algebra, and which is finite-dimensional over~$K$. (In the
literature étale algebras have to satisfy a separability condition,
but we work in characteristic~$0$ throughout the paper.)

If now~$p \colon S \to R$ is a ramified cover between compact
surfaces, we may speak of its degree, as the degree of the
corresponding covering~$p^{-1}(R \smallsetminus R_r) \to R
\smallsetminus R_r$. The following is given in \S6.2.4
in~\cite{douady}.

\begin{thm} \label{thm-riemann-surfaces-same-as-fields}
Fix a compact, connected Riemann surface~$R$. The category of compact
Riemann surfaces~$S$ with a ramified cover~$S \to R$ is
anti-equivalent to the category of étale algebras over~$\m (R)$. The
equivalence is given by~$S \mapsto \m (S)$, and the degree of~$S \to
R$ is equal to the dimension of~$\m (S)$ as a vector space over~$\m
(R)$.
\end{thm} 

(Here and elsewhere, ``anti-equivalent'' means ``equivalent to the
opposite category''.)


Taking~$R = \p$, we get a glimpse of yet another category that could
be equivalent to~$\cat$. However to pursue this, we need to translate
the condition about the ramification into a statement about étale
algebras (lest we should end up with a half-baked category, consisting
of algebras such that the corresponding surface has a certain
topological property; that would not be satisfactory). For this we
reword §2.2.1 of~\cite{helmut}.

Recall that~$\m (\p) = \C (x)$, where~$x$ is the identity of~$\p$. So
let us start with any field~$k$ at all, and consider a finite, Galois
extention~$L$ of~$k(x)$. We shall say that~$L/k(x)$ is {\em not
  ramified at~$0$} when it embeds into the extension~$k((x))/k(x)$,
where as usual~$k((x))$ is the field of formal power series in~$x$. In
this paper we will not enter into the subtleties of the
field~$k((x))$, nor will we discuss the reasons why this definition
makes sense. We chiefly want to mention that there is a simple
algebraic statement corresponding to the topological notion of
ramification, quoting the results we need.

Now take any~$s \in k$. From~$L$ we construct~$L_s = L \otimes_{k(x)}
k(x)$, where we see~$k(x)$ as an algebra over~$k(x)$ {\em via} the
map~$k(x) \to k(x)$ which sends~$x$ to~$x+s$; concretely if we pick a
primitive element~$y$ for~$L/k(x)$, so that~$L \cong k(x)[y] /(P)$,
then~$L_s$ is~$k(x)[y] / (P_s)$ where~$P_s$ is the result of
applying~$x \mapsto x+s$ to the coefficients of~$P$. When~$L_s/k(x)$
is not ramified at~$0$, we say that~$L/k(x)$ {\em is not ramified
  at~$s$}.

Finally, using the map~$k(x) \to k(x)$ which sends~$x$ to~$x^{-1}$, we
get an extension~$L_\infty/k(x)$, proceeding as above. When the latter
is not ramified at~$0$, we say that~$L/k(x)$ is {\em not ramified
  at~$\infty$}.

When the conditions above are not satisfied, for~$s \in k \cup \{
\infty \}$, we will of course say that~$L$ does ramify at~$s$ (or is
ramified at~$s$). That the topological and algebraic definitions of
ramification actually agree is the essence of the next lemma.

\begin{lem}
Let~$p \colon S \to \p$ be a ramified cover, with~$S$ connected, and
assume that the corresponding extension~$\m (S)/ \C(x)$ is
Galois. Then for any~$s \in \p$, the ramification set~$\p_r$
contains~$s$ if and only if~$\m (S)/\C (x)$ ramifies at~$s$ in the
algebraic sense.

In particular, the ramification set in contained in~$\{ 0, 1, \infty
\}$ if and only if the extension~$\m (S)/\C (x)$ does not ramify
at~$s$ whenever~$s \not \in \{ 0, 1, \infty \}$.
\end{lem}

This is the addendum to theorem 5.9 in~\cite{helmut}. Now we need to
get rid of the extra hypothesis that~$\m (S)/ \C(x)$ be Galois (a case
not considered in~\cite{helmut}, strictly speaking). Algebraically, we
say that an extension~$L/k(x)$ does not ramify at~$s$ when its Galois
closure~$\tilde L / k(x)$ does not. To see that, with this definition,
the last lemma generalizes to all ramified covers, we need to prove the
following. 

\begin{lem}
Let~$p \colon S \to \p$ be a ramified cover, where~$S$ is
connected. Let~$\tilde p \colon \tilde S \to \p$ be the ramified cover
such that~$\m (\tilde S) / \C(x)$ is the Galois closure of~$\m (S)/
\C(x)$. Then the ramification sets for~$S$ and~$\tilde S$ are equal. 
\end{lem}

\begin{proof}
We have~$\C(x) \subset \m(S) \subset \m(\tilde S)$, so we also have a
factorization of~$\tilde p$ as~$\tilde S \to S \to \p$. From this it
is clear that, if~$\tilde p$ is not ramified at~$s \in \p$, then
neither is~$p$. 

The crux of the proof of the reverse inclusion is the fact that
covering maps have Galois closures, usually called regular
covers. The following argument anticipates the material of the next
section, though it should be understandable now. 

Let~$\p_r$ be the ramification set for~$p$, and let~$U = p^{-1}(\p
\smallsetminus \p_r)$, so that~$U \to \p \smallsetminus \p_r$ is a
finite covering map. Now let~$\tilde U \to \p \smallsetminus \p_r$ be
the corresponding regular covering map. Here ``regular'' can be taken
to mean that this cover has as many automorphisms as its degree;
and~$\tilde U$ is minimal with respect to this property, among the
covers factoring through~$U$. The existence of~$\tilde U$ is standard
in covering space theory, and should become very clear in the next
section. Note that, if~$U$ corresponds to the subgroup~$H$
of~$\pi_1(\p \smallsetminus \p_r)$, then~$\tilde U$ corresponds to the intersection of
all the conjugates of~$H$.

As above, we can construct a Riemann surface~$S'$ from~$\tilde U$, and
the latter does not ramify outside of~$\p_r$. To prove the lemma, it
is sufficient to show that~$S'$ can be identified with~$\tilde S$.

However from basic Galois theory we see that~$\m(S') / \C(x)$ must be
Galois since it possesses as many automorphisms as its degree, and by
minimality it must be the Galois closure of~$\m (S) / \C(x)$. So~$S'$
and~$\tilde S$ are isomorphic covers of~$\p$.
\end{proof}

Finally, an étale algebra over~$k(x)$ will be said not to ramify
at~$s$ when it is a direct sum of field extensions, none of which
ramifies at~$s$. This clearly corresponds to the topological situation
when~$k= \C$, and we have established the following.

\begin{thm} \label{thm-C-complex-fields}
The category~$\cat$ is anti-equivalent to the category~$\et$ of
finite, étale algebras over~$\C(x)$ that are not ramified outside
of~$\{ 0, 1, \infty \}$, in the algebraic sense.
\end{thm}

\subsection{Extensions of~$\qb (x)$}

Let~$L/ \C(x)$ be a finite, Galois extension, and let~$n= [L :
  \C(x)]$. We shall say that it is {\em defined over~$\qb$} when there
is a subfield~$L_{rat}$ of~$L$, containing~$\qb (x)$ and Galois over
it, such that~$[L_{rat} : \qb(x)] = n$. This is equivalent to
requiring the existence of~$L_{rat}$ containing~$\bar \q(x)$ and
Galois over it such that~$L \cong L_{rat} \otimes_{\qb} \C$. That
these two conditions are equivalent follows (essentially) from (a) of
lemma 3.1 in~\cite{helmut}: more precisely this states that, when the
condition on dimensions holds, there is a primitive element~$y$
for~$L/\C(x)$ whose minimal polynomial has coefficients in~$\qb (x)$,
and~$y$ is also a primitive element for~$L_{rat}/\bar\q (x)$.

 Item (d) of the same lemma reads:

\begin{lem}
When~$L$ is defined over~$\qb$, the subfield~$L_{rat}$ is unique.
\end{lem}

This relies on the fact that~$\qb$ is algebraically closed, and
would not be true with~$\qb$ and~$\C$ replaced by arbitrary
fields. 

There is also an existence statement, which is theorem 7.9
in~\cite{helmut}:

\begin{thm} \label{thm-defined-over-qb}
If~$L/\C (x)$ is a finite, Galois extension which does not ramify
at~$s \in \C$ unless~$s \in \qb \cup \{ \infty \}$, then it is
defined over~$\qb$.
\end{thm} 

We need to say a word about extensions which are not assumed to be
Galois over~$\C(x)$. For this we now quote (b) of the same lemma 3.1
in~\cite{helmut}:

\begin{lem} \label{lem-Lrat-preserves-galois}
When~$L/\C(x)$ is finite, Galois, and defined over~$\bar\q$, there is
an isomorphism~$Gal(L/\C (x)) \cong Gal(L_{rat} / \qb (x))$
induced by restriction.
\end{lem} 

So from the Galois correspondence, we see that fields between~$\C(x)$
and~$L$, Galois or not over~$\C(x)$, are in bijection with fields
between~$\bar\q (x)$ and~$L_{rat}$. If~$K/\C(x)$ is any finite
extension, not ramified outside of~$\{ 0, 1, \infty \}$, we see by the
above that its Galois closure~$L/\C(x)$ is defined over~$\qb$, and
thus there is a unique field~$K_{rat}$, between~$\qb(x)$
and~$L_{rat}$, such that~$K \cong K_{rat} \otimes_{\qb} \C$.

Putting together the material in this section, we get:

\begin{thm} \label{thm-C-rational-fields}
The category~$\cat$ is anti-equivalent to the category~$\etq$ of
finite, étale extensions of~$\qb(x)$ that are not ramified outside
of~$\{ 0, 1, \infty \}$, in the algebraic sense.
\end{thm}

The functor giving the equivalence with the previous category is the
tensor product~$- \otimes_{\qb} \C$. Theorem~\ref{thm-defined-over-qb} shows  that it is essentially surjective; proving that it is fully faithful is an argument similar to the proof of lemma~\ref{lem-Lrat-preserves-galois} above. 

\subsection{Algebraic curves} \label{subsec-curves}

Strictly speaking, the following material is not needed to understand
the rest of the paper, and to reach our goal of describing the action
of~$\gal$ on dessins. Moreover, we expect the majority of our
readers to fit one of two profiles: those who know about algebraic
curves and have immediately translated the above statements about fields
into statements about curves; and those who do not know about
algebraic curves and do not wish to know. Nevertheless, in the sequel
we shall occasionally (though rarely) find it easier to make a point
in the language of curves. 

Let~$K$ be an algebraically closed field, which in the sequel will
always be either~$\C$ or~$\qb$. A {\em curve}~$C$ over~$K$ will be,
for us, an algebraic, smooth, complete curve over~$K$. We do not
assume curves to be irreducible, though smoothness implies that a
curve is a disjoint union of irreducible curves.

We shall not recall the definition of the above terms, nor the
definition of morphisms between curves. We also require the reader to
be (a little) familiar with the {\em functor of points} of a
curve~$C$, which is a functor from~$K$-algebras to sets that we
write~$L \mapsto C(L)$. There is a bijection between the set of
morphisms~$C \to C'$ between two curves on the one hand, and the set
of natural transformations between their functors of points on the
other hand; in particular if~$C$ and~$C'$ have isomorphic functors of
points, they must be isomorphic. For example, the first projective
space~$\p$ is a curve for which~$\p(L)$ is the set of lines in~$L^2$
when~$L$ is a field. (This holds for any base field~$K$; note that we
have already used the notation~$\p$ for~$\p(\C)$, the Riemann
sphere. We also use below the notation~$\mathbb{P}^n(L)$ for the set
of lines in~$L^{n+1}$, as is perfectly standard (though~$\mathbb{P}^n$
is certainly not a curve for~$n \ge 2$)).

 In concrete terms, given a connected curve~$C$ it is always possible
 to find an integer~$n$ and homogeneous polynomials~$P_i(z_0, \ldots ,
 z_n)$ (for~$1 \le i \le m$) with the following property: for each
 field~$L$ containing~$K$, we can describe~$C(L)$ as the subset of
 those points~$[z_0 : \cdots : z_n]$ in the projective
 space~$\mathbb{P}^n(L)$ satisfying
\[ P_i(z_0, \ldots , z_n) = 0 \qquad (1 \le i \le m)\, . \tag{*}  \]
Thus one may (and should) think of curves as subsets of~$\mathbb{P}^n$
defined by homogeneous polynomial equations. When~$K$ is algebraically
closed, as is the case for us, one can in fact show that~$C$ is
entirely determined by the {\em single} subset~$C(K)$ {\em together}
with its embedding in~$\mathbb{P}^n(K)$. 

We illustrate this with the so-called {\em rational functions} on~$C$,
which by definition are the morphisms~$C\to \p$ with the exclusion of
the ``constant morphism which is identically~$\infty$''. When~$C(K)$
is presented as above as a subset of~$\mathbb{P}^n(K)$, these
functions can alternatively be described in terms of maps of
sets~$C(K) \to K \cup \{ \infty \}$ of the following particular form:
take~$P$ and~$Q$, two homogeneous polynomials in~$n+1$ variables, of
the same degree, assume that~$Q$ does not vanish identically
on~$C(K)$, assume that~$P$ and~$Q$ do not vanish simultaneously
on~$C(K)$, and consider the map on~$C(K)$ sending~$z$ to~$P(z) /Q(z)$
if~$Q(z) \ne 0$, and to~$\infty$ otherwise. (In other words~$z$ is
sent to~$[P(z) : Q(z)]$ in~$\p(K) = K \cup \{ \infty \}$.)

The rational functions on the connected curve~$C$ comprise a field~$\m (C)$ (an étale algebra when~$C$ is not connected). We use the same letter as we did for meromorphic functions, which is justified by the following arguments. Assume that~$K= \C$. Then our hypotheses guarantee that~$S = C(\C)$ is naturally a Riemann surface. In fact if we choose polynomial equations as above, then~$S$ appears as a complex submanifold of~$\mathbb{P}^n(\C)$. It follows that the rational functions on~$C$, from their description as functions on~$S$, are meromorphic. However, a non-trivial but classical result asserts the converse : all meromorphic functions on~$S$ are in fact rational functions (\cite{harris}, chap.\ 1, \S3). Thus~$\m (S) = \m(C)$. When~$K= \qb$, it still makes sense to talk about the Riemann surface~$S= C(\C)$, and then~$\m(S) = \m(C) \otimes_\qb \C$. For example~$\m(\p) = K(x)$, when we see~$\p$ as a curve over any field~$K$.

The following theorem is classical.

\begin{thm} \label{thm-curves-same-as-fields}
The category of connected curves over~$K$, in which constant morphisms
are excluded, is anti-equivalent to the category of fields of
transcendence degree~$1$ over~$K$, the equivalence being given by~$C
\mapsto \m (C)$.
\end{thm}

From this we have immediately a new category equivalent to~$\cat$, by
restricting attention to the fields showing up in
theorem~\ref{thm-C-complex-fields} or
theorem~\ref{thm-C-rational-fields}.  Let us define a morphism~$C \to
\p$ to be ramified at~$s \in K \cup \{ \infty \}$ if and only if the
corresponding extension of fields~$\m (C) / K(x)$ ramifies at~$s$;
this may sound like cheating, but expressing properties of a morphism
in terms of the effect on the fields of rational functions seems to be
in the spirit of algebraic geometry. It is then clear that:

\begin{thm}
The category~$\cat$ is equivalent to the category of curves~$C$,
defined over~$\C$, equipped with a morphism~$C \to \p$ which does not ramify
outside of~$\{ 0, 1, \infty \}$. Here the morphisms taken into account
are those over~$\p$.

Likewise, ~$\cat$ is equivalent to the category of curves defined
over~$\qb$ with a map~$C \to \p$ having the same ramification
property.
\end{thm}

(Note that we have used the same notation~$\p$ for an object which is
sometimes seen as a curve over~$\C$, sometimes as a curve over~$\qb$,
sometimes as a Riemann surface.)

As a side remark, we note that these equivalences of categories imply in particular the well-known fact that ``Riemann surfaces are algebraic''. For if we start with~$S$, a Riemann surface, and consider the field~$\m (S)$, then by theorem~\ref{thm-curves-same-as-fields} there must be a curve~$C$ over~$\C$ such that~$\m (C) = \m(S)$ (where on the left hand side~$\m$ means ``rational functions'', and on the right hand side it means ``meromorphic functions''). However, we have seen that~$\m(C) = \m( C(\C))$ (with the same convention), and the fact that~$\m(S)$ and~$\m(C(\C))$ can be identified implies that~$S$ and~$C(\C)$ are isomorphic (theorem~\ref{thm-riemann-surfaces-same-as-fields}). Briefly, any Riemann surface~$S$ can be cut out by polynomial equations in projective space.

Likewise, the above theorems show that if~$S$ has a Belyi map, then
there is a curve {\em over~$\qb$} such that~$S$ is isomorphic
to~$C(\C)$. This is usually expressed by saying that~$S$ is ``defined
over~$\qb$'', or is an ``arithmetic surface''. The converse is
discussed in the next section.

\subsection{Belyi's theorem}

When considering a dessin~$\cell$, we define a curve~$C$
over~$\qb$. Is it the case that {\em all} curves over~$\qb$ are
obtained in this way? Given~$C$, it is of course enough to find a
Belyi map, that is a morphism~$C \to \p$ with ramification in~$\{ 0,
1, \infty \}$: the above equivalences then guarantee that~$C$
corresponds to some~$\cell$. In turn, Belyi has proved precisely this
existence statement:

\begin{thm}[Belyi]
Any curve~$C$ over~$\qb$ possesses a Belyi map.
\end{thm} 

The proof given by Belyi in~\cite{belyi}, and reproduced in many
places, is very elegant and elementary. It starts with any morphism~$F
\colon C \to \p$, and modifies it ingeniously to obtain another one
with appropriate ramification.

\section{Regularity} \label{sec-regularity}

From now on, it will be convenient to use the word {\em dessin} to
refer to an object in any of the equivalent categories at our disposal
(especially when we want to think of it simultaneously as a cell
complex and a field, for example). 

In this section we study regular dessins. These could have been called
``Galois'' instead of ``regular'', since the interpretation in the
realm of field extensions is precisely the Galois condition, but we
want to avoid the confusion with the Galois group~$\gal$ which will
become a major player in the sequel.

\subsection{Definition of regularity}

An object in~$\cat$ has a {\em degree} given by the number of darts. In the other categories equivalent to~$\cat$, this translates in various ways. In~$\sets_{\sigma , \alpha , \phi}$, it is the cardinality of the set having the three permutations on it. In the categories of étale algebras over~$\C(x)$ or~$\qb (x)$, it is the dimension of the algebra as a vector space over~$\C(x)$ or~$\qb (x)$ respectively. In the category of finite coverings of~$\p \smallsetminus \{ 0, 1, \infty \}$, it is the cardinality of any fibre.

There is also a notion of {\em connectedness} in these categories. A
dessin~$\cell$ is connected when~$\topo{\cell}$ is connected,
which happens precisely when the corresponding étale algebras are
actually fields, or when the cartographic group acts transitively (cf
lemma~\ref{lem-connected-iff-transitive}).

In this section we shall focus on the automorphism groups of connected
dessins. We are free to conduct the arguments in any category, and
most of the time we prefer~$\sets_{\sigma, \alpha, \phi}$. However,
note the following at once.

\begin{lem}
The automorphism group of a connected dessin is a finite group, of
order no greater than the degree.
\end{lem}

\begin{proof}
This is obvious in~$\etq$: in fact for any finite-dimensional
extension of fields~$L/K$, basic Galois theory tells us that the
automorphism group of the extension has order no greater
than~$[L:K]$. 

A proof in~$\sets_{\sigma, \alpha, \phi}$ will be immediate from
lemma~\ref{lem-aut-sets-explicit} below.
\end{proof}

A dessin will be called {\em regular} when it is connected and the
order of its automorphism group equals its degree.

In terms of field extensions for example, then~$L/\C(x)$
is regular if and only if it is Galois (in the elementary sense, ie
normal and separable). In terms of a covering~$U \to \p \smallsetminus
\{ 0, 1, \infty \}$, with~$U$ is connected, then it is regular if
and only if it is isomorphic to the cover~$U \to U/G$ where~$G$ is the
automorphism group (this agrees with the use of the term ``regular''
in covering space theory, of course).

\begin{rmk}
The reader needs to pay special attention to the following
convention. When~$X$ is a dessin and~$h, k \in Aut(X)$, we write~$hk$
for the composition of~$k$ followed by~$h$; that is~$hk(x) = h(k(x))$,
at least when we are willing to make sense of~$x \in X$ (for example
in~$\cat$ this will mean that~$x$ is in fact a triangle). In other
words, we are letting~$Aut(X)$ act on~$X$ {\em on the left}. While
this will be very familiar to topologists, for whom it is common to
see the ``group of deck transformations'' of a covering map act on the
left and the ``monodromy group'' act on the right, other readers may
be puzzled to see that we have treated the category of sets
differently when we took the convention described in
remark~\ref{rmk-convention-permutations}.

To justify this, let us spoil the surprise of the next paragraphs, and
announce the main result at once: in~$\sets_{\sigma, \alpha, \phi}$, a
regular dessin is precisely a group~$G$ with two distinguished
generators~$\sigma $ and~$\alpha $; the monodromy group is~$G$ itself,
acting on the right by translations, while the automorphism group is
again~$G$ itself, acting on the left by translations.

If we had taken different conventions, we would have ended up with one
of these actions involving inverses, in a way which is definitely
unnatural. 
\end{rmk}

\subsection{Sets with permutations} \label{subsec-regularity-sets-with-perms}

We explore the definition of regularity in the context
of~$\sets_{\sigma, \alpha , \phi}$, where it is very easy to
express. 

Let~$X$ be a set of cardinality~$n$, with three permutations~$\sigma ,
\alpha, \phi$ satisfying~$\sigma \alpha \phi = 1$. Let~$G$ denote the
cartographic group; recall that by definition, it is generated
by~$\sigma$ and~$\alpha $ as a subgroup of~$S(X) \cong S_n$, acting
on~$X$ on the right. We assume that~$G$ acts transitively (so the
corresponding dessin is connected).

We choose a base-point~$*\in X$. The map~$g \mapsto *^g$ identifies~$H
\bs G$ with~$X$, where~$H$ is the stabilizer of~$*$. This is an
isomorphism in~$\sets_{\sigma , \alpha , \phi}$, with~$G$ acting on~$H
\bs G$ by right translations. As we shall insist below that the choice
of base-point is somewhat significant, we shall keep the notation~$X$
and not always work directly with~$H \bs G$.

Since the morphisms in~$\sets_{\sigma, \alpha, \phi}$ are special maps
of sets, we can relate~$Aut(X)$ and~$S(X)$, where the automorphism
group is taken in~$\sets_{\sigma, \alpha, \phi}$, and~$S(X)$ as always
is the group of all permutations of~$X$. More precisely, any~$h \in
Aut(X)$ can be seen as an element of~$S(X)$, still written~$h$, and
there is a homomorphism~$Aut(X) \to S(X)$ given by~$h \mapsto h^{-1}$;
our left-right conventions force us to take inverses to get a
homomorphism. (In other words, $Aut(X)$ is naturally a subgroup
of~$S(X)^{op}$, the group~$S(X)$ with the opposite composition law.)
As announced, the conventions will eventually lead to a result without
inverses.

\begin{lem} \label{lem-aut-sets-explicit}
Let~$X, G, H$ be as above. We have the following two descriptions
of~$Aut(X)$.
\begin{enumerate}
\item Let~$N(H)$ be the normalizer of~$H$ in~$G$. Then for each~$g \in
  N(H)$, the map~$H \bs G \to H \bs G$ given by~$[x] \mapsto [gx]$ is
  in~$Aut(H \bs G)$. This construction induces an isomorphism~$Aut(X)
  \cong N(H)/H$.
\item The map~$Aut(X) \to S(X)$ is an isomorphism onto the centralizer
  of~$G$ in~$S(X)$.
\end{enumerate}
\end{lem}

\begin{proof}
(1) The notation~$[x]$ is for the class of~$x$ in~$H \bs G$, of
  course. To see that~$[gx]$ is well-defined, let~$h \in H$, then~$ghx
  = g h g^{-1} g x$ so~$[ghx] = [gx]$. The map clearly commutes with
  the right action of~$G$, and so is an automorphism, with inverse
  given by~$[x] \mapsto [g^{-1} x]$.

Conversely, any automorphism~$h$ is determined by~$h([1])$, which we
call~$[g]$, and we must have~$h([x]) = h([1]^x) = [g]^x = [gx]$ for
any~$x$; the fact that~$h$ is well-defined implies that~$g \in
N(H)$. So there is a surjective map~$N(H) \to Aut(H \bs G)$ whose
kernel is clearly~$H$.

(2) An automorphism of~$X$, by its very definition, is a
self-bijection of~$X$ commuting with the action of~$G$; so this second
point is obvious.
\end{proof}

We also note the following.

\begin{lem}
$Aut(X)$ acts freely on~$X$.
\end{lem}

\begin{proof}
If~$h(x) = x$ for some~$x \in X$, then~$h(x^g) = h(x)^g = x^g$
so~$x^g$ is also fixed by~$h$, for any~$g \in G$. By assumption~$G$
acts transitively, hence the lemma.
\end{proof}

\begin{prop} \label{prop-equiv-conditions-regularity}
The following are equivalent.
\begin{enumerate}
\item $Aut(X)$ acts transitively on~$X$.
\item $G$ acts freely on~$X$.
\item $H$ is normal in~$G$.
\item $H$ is trivial.
\item $G$ and~$Aut(X)$ are isomorphic.
\item $G$ and~$Aut(X)$ are both of order~$n$.
\item $X$ is regular.
\end{enumerate}
\end{prop}

\begin{proof}
That (1) implies (2) is almost the argument we used for the last
lemma, only with the roles of~$Aut(X)$ and~$G$ interchanged. Condition
(2) implies (4) by definition and hence (3); when we have (3) we have
$ N(H) /H = G/H$, and the description of the action of~$N(H) / H$
on~$H \bs G$ makes it clear that (1) holds.

Condition (4) implies~$N(H)/H \cong G$, so we have (5); we also have
(6) since~$X$ (whose cardinality is~$n$) can be identified with~$G$
acting on itself on the right. Conversely if we have (6), given that
the cardinality of~$X$ is~$n = |G| / |H|$ we deduce (4).

Finally (7), by definition, means that~$Aut(X)$ has order~$n$, so it
is implied by (6). Conversely, since this group acts freely on~$X$,
having cardinality~$n$, it is clear that (7) implies that the action
is also transitive, which is (1).
\end{proof}

\begin{coro}[of the proof] \label{coro-regular-objects}
Let~$X$ be a regular object in~$\sets_{\sigma, \alpha, \phi}$ with
cartographic group~$G$. Then~$X$ can be identified with~$G$ itself
with its action on itself on the right by translations. The
automorphism group~$Aut(X)$ can also be identified with~$G$, acting
on~$X=G$ on the left by translations. 

Conversely any finite group~$G$ with two distinguished
generators~$\sigma $ and~$\alpha $ defines a regular object in this
way. 
\end{coro}

\begin{proof}
There remains the (very easy) converse to prove. If we start with~$G$,
a finite group generated by~$\sigma $ and~$\alpha $, we can let it act
on itself on the right by translations, thus defining an object
in~$\sets_{\sigma, \alpha , \phi}$. The cartographic group is easily
seen to be isomorphic to~$G$ (in fact this is the traditional Cayley
embedding of~$G$ into the symmetric group~$S(G)$). The action of the
cartographic group is, as a result, free and transitive, so the object
is regular.
\end{proof}


However, some care must be taken. The identifications above are not
canonical, but depend on the choice of base-point. Also, the actions
of~$g \in G$ on~$X$, given by right and left multiplications, are very
different-looking maps of the set~$X$. We want to make these points
crystal-clear.  The letter~$d$ below is used for ``dart''.

\begin{prop} \label{prop-various-isos-carto-auto}
Suppose that~$X$ is regular. Then for each~$d \in X$ there is an
isomorphism
\[ \iota_d \colon G \longrightarrow Aut(X) \, .   \]
The automorphism~$\iota_d(g)$ is the unique one taking~$d$
to~$d^g $. 

Changing~$d$ to~$d'$ amounts to conjugating, in~$Aut(X)$, by the
unique automorphism taking~$d$ to~$d'$.
\end{prop}

\begin{proof}
This is merely a reformulation of the discussion above, and we only
need to check some details. We take~$*=d$ as base-point. The
map~$\iota_d$ is clearly well-defined, and we check that it is a
homomorphism: $\iota_d(gh)(d) = d^{gh} = (d^g)^h = \iota_d(g)(d)^h =
\iota_d( g)(d^h) = \iota_d (g) \iota_d (h) (d)$, so the
automorphisms~$\iota_d(gh)$ and~$\iota_d(g) \iota_d(h)$ agree at~$d$,
hence everywhere by transitivity of the action of~$G$. 
\end{proof}

\begin{ex}
Consider the dessin on the sphere given by the tetrahedron, as
follows:

\figurehere{0.3}{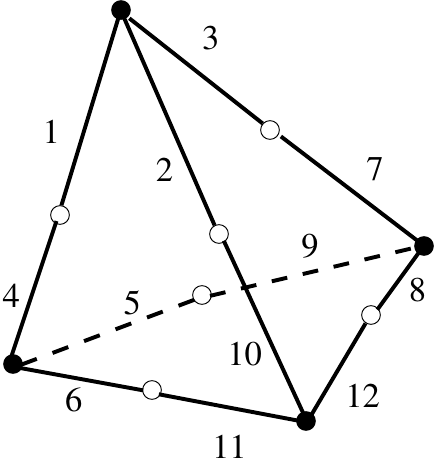}
Here we have numbered the darts, for convenience (the faces, on the
other hand, are implicit). There are many ways to see that this is a
regular dessin. For example, one may find enough rotations to
take any one dart to any other one, and apply criterion (1) of
proposition~\ref{prop-equiv-conditions-regularity}. Or, we could write
the permutations
\[ \sigma = (123)(456)(789)(10, 11, 12) \, , \qquad \alpha = (14)(2,
10)(37)(59)(6, 11)(8, 12) \, ,   \]
and compute the order of the group generated by~$\sigma $ and~$\alpha
$, which is~$12$ (a computer does that for you immediately). Then
appeal to criterion (6) of the same proposition. Finally, one could
also determine the automorphism group of this dessin, and find
that it has order 12. This is the very definition of regularity.

Take~$d= 1$ as base point, and write~$\iota $ for~$\iota_1$. What
is~$\iota (\sigma )$? This is the automorphism taking~$1$ to~$2$,
which is the rotation around the black vertex adjacent to~$1$
and~$2$. The permutation of the darts induced by~$\iota (\sigma )$ is 
\[ (123)(4, 10, 7) (6, 12, 9) (11, 8, 5) \, .   \]
We see that~$\sigma $ and~$\iota (\sigma )$ are not to be
confused. Likewise, $\iota (\alpha )$ is the rotation taking~$1$
to~$4$, and the induced permutation is 
\[ (14)(8, 12) (2, 5)(3, 6)(10, 9)(11, 7) \, .  \]

\end{ex}

\subsection{The distinguished triples}

From proposition~\ref{prop-various-isos-carto-auto}, we see that each
choice of dart in a regular dessin~$\cell$ defines three elements
of~$Aut(\cell)$, namely~$\tilde \sigma = \iota_d( \sigma )$, $\tilde
\alpha = \iota_d(\alpha )$, and~$\tilde \phi = \iota_d(\phi)$. These
are generators of~$Aut(\cell)$, and they satisfy~$\tilde \sigma \tilde
\alpha \tilde \phi = 1$. Changing~$d$ to another dart conjugates all
three generators {\em simultaneously}. Any such triple, obtained for a
choice of~$d$, will be called a {\em distinguished triple}
for~$\cell$. 

\begin{lem}
If~$d$ and~$d'$ are darts with a common black vertex, then~$\iota_d(
\sigma ) = \iota _{d'}(\sigma )$. Similarly if they have a common
white vertex then~$\iota_d( \alpha ) = \iota_{d'} (\alpha )$. Finally
if the black triangles corresponding to~$d$ and~$d'$ respectively lie
in the same face, then~$\iota_d( \phi) = \iota_{d'}(\phi)$.
\end{lem}

\begin{proof}
We treat the first case, for which~$d' = d^{\sigma^k} $ for
some~$k$. Write~$\tilde \sigma = \iota_d(\sigma )$. Since~$d^{\sigma
  ^k} = \tilde \sigma^k(d)$, we see that~$\iota_{d'}(\sigma ) =
\tilde \sigma ^k \tilde \sigma \tilde \sigma ^{-k} = \tilde \sigma $.
\end{proof}

Thus the notation~$\tilde \sigma $ makes senses unambiguously when it
is understood that the possible base-darts are incident to a given
black vertex. Similarly for the other types of points. We can now
fully understand the fixed points of automorphisms:

\begin{prop} \label{prop-fixed-points}
Let~$h \in Aut(\cell)$, where~$\cell$ is regular. Suppose that the
induced homeomorphism~$\topo{\cell} \to \topo{\cell}$ has a fixed
point. Suppose also that~$h$ is not the identity. Then the fixed point
is a vertex or the centre of a face; moreover there exists an
integer~$k$ such that, for any choice of dart~$d$ incident with the
fixed point, we can write~$h=\tilde \sigma ^k$, $\tilde \alpha ^k$
or~$\tilde \phi ^k$, according to the type of fixed point, $\bullet$,
$\circ $ or~$\star $.

In particular, the subgroup of~$Aut(\cell)$ comprised of the
automorphisms fixing a given point of type~$\bullet$ is cyclic,
generated by~$\tilde \sigma = \iota_d(\sigma ) $ where we have chosen
any dart incident with the fixed point. Likewise for the other types
of fixed point.
\end{prop}

(In this statement we have abused the language slightly, by saying that
a dart is ``incident'' to the centre of a face if the corresponding
black triangle belongs to that face.)

\begin{proof}
Let~$t$ be a triangle containing the fixed point. Note that~$h(t) \ne
t$: otherwise by regularity we would have~$h=$ identity. We have~$t
\cap h(t) \ne \emptyset$ though, and as the triangle~$h(t)$ is of the
same colour as~$t$, unlike its neighbours, we conclude that~$t \cap
h(t)$ is a single vertex of~$t$, and the latter is our fixed point.

Say it is a black vertex. Let~$d$ be the dart on~$t$. Then~$h(d)$ is a
dart with the same black vertex as~$d$, so~$h(d) = d^{\sigma ^k}$ for
some integer~$k$. In other words~$h = \iota_d(\sigma^k)$.
\end{proof}

Thus we have a canonical generator for each of these subgroups. Here
we point out, and this will matter in the sequel, that the
generator~$\tilde \sigma $ agrees with what Völklein calls the
``distinguished generator'' in Proposition 4.23 of~\cite{helmut}. This
follows from unwinding all the definitions.

The following result is used very often in the literature on regular
``maps''. 

\begin{prop} \label{prop-regular-dessins-are-groups-with-generators}
Let~$\cell$ be a dessin, with cartographic group~$G$, and the
distinguished elements~$\sigma , \alpha, \phi \in G$. Similarly,
let~$\cell', G', \sigma ', \alpha ', \phi'$ be of the same
kind. Assume that~$\cell$ and~$\cell'$ are both regular. Then the
following conditions are equivalent: \begin{enumerate}
\item $\cell$ and~$\cell'$ are isomorphic,
\item there is an isomorphism~$G \to G'$ taking~$\sigma $ to~$\sigma
  '$, $\alpha$ to $ \alpha '$ and~$\phi$ to $\phi'$,
\item there is an isomorphism~$Aut(\cell) \to Aut(\cell')$ taking a
  distinguished triple to a distinguished triple.
\end{enumerate}
\end{prop}

\begin{proof}
That (1) implies (2) is obvious, and holds without any regularity
assumption. Since there are isomorphisms~$G \cong Aut(\cell)$
and~$G' \cong Aut(\cell')$ taking the distinguished permutations in the
cartographic group to a distinguished triple (though none of this is
canonical), we see that (2) implies (3). 

Finally, if we work in~$\sets_{\sigma , \alpha , \phi}$, we can
identify~$\cell$ with the group~$Aut(\cell)$ endowed with the three
elements~$\tilde \sigma , \tilde \alpha , \tilde \phi$ acting by
multiplication on the right, where we have picked some distinguished
triple~$\tilde \sigma, \tilde \alpha , \tilde \phi$. Thus (3) clearly
implies (1).
\end{proof}

The equivalence of (1) and (3), together with
corollary~\ref{coro-regular-objects}, reduces the classification of
regular dessins to that of finite groups with two distinguished
generators (or three distinguished generators whose product
is~$1$). We state this separately as an echo to
proposition~\ref{prop-iso-classes}. Recall that dessins are implicitly
compact, oriented and without boundary here.

\begin{prop} \label{prop-iso-classes-regular}
\begin{enumerate}
\item A regular dessin determines, and can be reconstructed from, a
  finite group~$G$ with two distinguished generators~$\sigma $
  and~$\alpha $. We obtain isomorphic dessins from~$(G, \sigma, \alpha
  )$ and~$(G', \sigma ', \alpha ')$ if and only if there is an
  isomorphism~$G \to G'$ taking~$\sigma $ to~$\sigma '$ and~$\alpha $
  to~$\alpha '$.

\item The set of isomorphism classes of regular dessins is in
  bijection with the normal subgroups of the free group on two
  generators. More precisely, if a connected dessin corresponds to the
  conjugacy class of the subgroup~$K$ as in
  proposition~\ref{prop-iso-classes}, then it is regular if and only
  if~$K$ is normal.
\end{enumerate}
\end{prop} 

\begin{proof}
We have already established (1). As for the first statement in (2), we
only need to remark that the groups mentioned in (1) are precisely the
groups of the form~$G = \langle \sigma, \alpha \rangle / N$ for some
normal subgroup~$N$ in the free group~$F_2 =\langle \sigma, \alpha
\rangle$, and that an isomorphism of the type specified in (1)
between~$G= F_2/N$ and~$G' = F_2 / N'$ exists if and only if~$N = N'$.

We turn to the last statement. If a connected dessin corresponds
to~$K$, then it is isomorphic to~$X= K \bs \langle \sigma, \alpha
\rangle$ in~$\sets_{\sigma, \alpha, \phi}$. The action of~$\langle
\sigma, \alpha \rangle$ on~$X$ yields a homomorphism~$f \colon \langle
\sigma, \alpha \rangle \to S(X)$ whose image is the cartographic
group~$G$, and whose kernel is the intersection~$N$ of all the conjugates
of~$K$, so~$G \cong \langle \sigma, \alpha \rangle / N$. Let~$H$ be
the stabilizer in~$G$ of a point in~$X$. Then~$f^{-1}(H)$ is the
stabilizer of that same point in~$\langle \sigma, \alpha \rangle$, so
it is a conjugate of~$K$. Now, $X$ is regular if and only if~$H$ is
trivial, which happens precisely when~$f^{-1}(H) = N$, which in turn
occurs precisely when~$K$ is normal. 
\end{proof}

\subsection{Regular closure \& Galois correspondence}

In the discussion that follows, we restrict our attention to {\em
  connected} dessins. 

When~$\cell$ and~$\cell'$ are two dessins, we call~$\cell'$ an {\em
  intermediate dessin} of~$\cell$ when there exists a morphism~$\cell
\to \cell'$. To appreciate the term ``intermediate'', it is best to
move to categories other than~$\cat$. In~$\covs$, if~$\cell$
corresponds to~$p \colon S \to \p$ and~$\cell'$ corresponds to~$p'
\colon S' \to \p$, then~$\cell'$ is an intermediate dessin of~$\cell$
when there is a factorization of~$p$ as 
\[ p \colon S \stackrel{f}{\longrightarrow} S' \stackrel{p'}{\longrightarrow} \p \,
,  \]
for some map~$f$; so~$\topo{\cell'} = S'$ is intermediate
between~$\topo{\cell} = S$ and~$\p$, if you will. In~$\etq$, the
towers~$\qb(x) \subset L' \subset L$ provide examples where~$L'/
\qb(x)$ is an intermediate dessin of~$L/\qb(x)$, and all examples are
isomorphic to one of this kind. 

Of course the word ``intermediate'' is borrowed from field/Galois
theory, where the ideas for the next paragraphs come from. Let us
point out one more characterization.

\begin{lem}
Let~$\cell$ and~$\cell'$ correspond to the conjugacy classes of the
subgroups~$H$ and~$H'$ of~$\langle \sigma, \alpha \rangle$
respectively, as in proposition~\ref{prop-iso-classes}. Then~$\cell'$
is an intermediate dessin of~$\cell$ if and only if some conjugate
of~$H'$ contains~$H$.
\end{lem}

So~$H'$ is intermediate between~$H$ and the free group~$\langle
\sigma, \alpha \rangle$.

\begin{proof}
The object in~$\sets_{\sigma, \alpha, \phi}$ corresponding to~$H$ (and
also to~$\cell$) is~$X=H \bs \langle \sigma, \alpha \rangle$, and
likewise for~$H'$ we can take~$X' = H' \bs \langle \sigma, \alpha
\rangle$; there is a map~$X \to X'$ if and only if the stabilizer of
some point in~$X$ is contained in the stabilizer of some point
in~$X'$, hence the lemma.
\end{proof}

\begin{lem}
Let~$\cell$ be a connected dessin. There exists a regular
dessin~$\tilde \cell$ such that~$\cell$ is an intermediate dessin
of~$\tilde \cell$. Moreover, we can arrange for~$\tilde \cell$ to be
minimal in the following sense: if~$\cell$ is an intermediate dessin
of any regular dessin~$\cell'$, then~$\tilde \cell$ is itself an
intermediate dessin of~$\cell'$. Such a minimal~$\tilde \cell$ is
unique up to isomorphism.

Finally, the cartographic group of~$\cell$ is isomorphic
to~$Aut(\tilde \cell)$.
\end{lem}

We call~$\tilde \cell$ the {\em regular closure} of~$\cell$.

\begin{proof}
Leaving the last statement aside, in~$\etq$, this is a basic result
from Galois theory. Alternatively, we can rely on
proposition~\ref{prop-iso-classes} and the previous lemma: if~$\cell$
corresponds to the conjugacy class of~$H$, then clearly the object
corresponding to~$N$, the intersection of all conjugates of~$H$, suits
our purpose. As for the last statement, that the cartographic group
of~$H \bs \langle \sigma , \alpha \rangle$ is isomorphic to~$\langle
\sigma , \alpha \rangle / N$ was already observed during the proof of
proposition~\ref{prop-iso-classes-regular} (and is obvious anyway).
\end{proof}



The fundamental theorem of Galois theory applied in~$\etq$, or some
elementary considerations with the subgroups of~$\langle \sigma,
\alpha \rangle$, imply:

\begin{prop}
Let~$\cell$ be a regular dessin. There is a bijection between the set
of isomorphism classes of intermediate dessins of~$\cell$ on the one
hand, and the conjugacy classes of subgroups of~$Aut(\cell)$ on the
other hand. Normal subgroups corresponds to regular, intermediate
dessins.
\end{prop}

The concepts of this section are, as usual, very easily illustrated
within~$\sets_{\sigma, \alpha, \phi}$. A connected object is of the
form~$H \bs G$, as we have seen, where~$G$ has two distinguished
generators~$\sigma $ and~$\alpha $. The regular closure is the
object~$G$, with its right action on itself, seen in~$\sets_{\sigma ,
  \alpha , \phi}$. Of course there is the natural map~$G \to H \bs
G$. Conversely any~$X$ with a surjective, equivariant map~$G \to X$
(that is, any connected, intermediate object of~$G$) must be of the
form~$H \bs G$, clearly. From this we see that whenever~$\cell$ is
regular, its intermediate dessins might called its {\em quotient}
dessins instead.





\section{The action of~$\gal$} \label{sec-galois-action}

In this section we show how each element~$\lambda \in \gal$ defines a
self-equivalence of~$\cat$, or any of the other categories equivalent
to it. Writing~$\act{\lambda }{\cell}$ for the object obtained by
applying this functor to the dessin~$\cell$, we show that there is an
isomorphism between~$\act{\lambda \mu }{\cell}$ and~$\act{\lambda
}{(\act{\mu }{\cell})}$, so~$\gal$ acts on the set of isomorphism
classes of dessins. 

The definition of the action is in fact given in~$\etq$, where it is
most natural. The difficulty in understanding it in~$\cat$ has much to
do with the zig-zag of equivalences that one has to go through. For
example, the functor from Riemann surfaces to fields is
straightforward, and given by the ``field of meromorphic functions''
construction, but the inverse functor is more mysterious.

We study carefully the genus~$0$ case, and include a detailed
description of a procedure to find a Belyi map associated to a planar
dessin -- which is, so far, an indispensable step to study the
action. We say just enough about the genus 1 case to establish that
the action is faithful.

We then proceed to study the features which are common to~$\cell$
and~$\act{\lambda }{\cell}$, for example the fact that the
surfaces~$\topo{\cell}$ and~$\topo{\act{\lambda }{\cell}}$ are
homeomorphic (so that the action modifies dessins on a given
topological surface). Ultimately one would hope to know enough of
these ``invariant'' features to predict the orbit of a given dessin
under~$\gal$ without having to compute Belyi maps, but this remains an
open problem.

\subsection{The action}

Let~$\lambda \colon \qb \to \qb$ be an element
of~$\gal$. We extend it to a map~$\qb(x) \to \qb (x)$
which fixes~$x$, and use the same letter~$\lambda $ to denote it. In
this situation the tensor product operation
\[ - \otimes_{\lambda } \qb (x)  \]
defines a functor from~$\etq$ to itself. In more details,
if~$L/\qb(x)$ is an étale algebra, one considers 
\[ \act{\lambda }{L} = L \otimes_\lambda \qb (x) \, .    \]
The notation suggests that we see~$\qb (x)$ as a module over itself
{\em via} the map~$\lambda $. We turn~$\act{\lambda }{L}$ into an
algebra over~$\qb(x)$ using the map~$t \mapsto 1 \otimes t$.

To describe this in more concrete terms, as well as verify
that~$\act{\lambda }{L}$ is an étale algebra over~$\qb(x)$
whenever~$L$ is, it is enough to consider field extensions, since the
operation clearly commutes with direct sums. So if~$L \cong
\qb(x)[y]/(P)$ is a field extension of~$\qb(x)$, with~$P \in
\qb(x)[y]$ an irreducible polynomial, then~$\act{\lambda }{L} \cong
\qb(x)[y]/(\act{\lambda }P )$, where~$\act{\lambda }P $ is what you
get when the (extented) map~$\lambda $ is applied to the coefficients
of~$P$. Clearly~$\act{\lambda }{P}$ is again irreducible (if it could
be factored as a product, the same could be said of~$P$ by
applying~$\lambda^{-1}$). Therefore~$\act{\lambda }{L}$ is again a
field extension of~$\qb(x)$, and coming back to the general case, we
do conclude that~$\act{\lambda }{L}$ is an étale algebra whenever~$L$
is. What is more, the ramification condition satisfied by the objects of~$\etq$ is obviously preserved.

Let~$\mu \in \gal$. Note that~$y \otimes s \otimes t \mapsto
y \otimes \mu (s) t$ yields an isomorphism 
\[ \act{\mu }{\left( \act{\lambda }{L}    \right)} = L
\otimes_\lambda \qb (x) \otimes_\mu \qb (x) \longrightarrow L
\otimes_{\mu \lambda } \qb (x) = \act{\mu \lambda }{L} \, .  \]
As a result, the group~$\gal$ acts (on the left) on the
set of isomorphism classes of objects in~$\etq$, or in any category
equivalent to it. We state this separately in~$\cat$.

\begin{thm}
The absolute Galois group~$\gal$ acts on the set of
isomorphism classes of compact, oriented dessins without
boundaries.
\end{thm}

\subsection{Examples in genus~$0$; practical computations}

We expand now on example~\ref{ex-belyi-fractions}. Let~$\cell$ be a
dessin on the sphere. We have seen that we can find a rational
fraction~$F$ such that~$F \colon \p \to \p$ is the ramified cover
corresponding to~$\cell$. 

In terms of fields of meromorphic functions, we have the
injection~$\C(x) \to \C(z)$ mapping~$x$ to~$F(z)$; here~$x$ and~$z$
both denote the identity of~$\p$, but we use different letters in
order to distinguish between the source and target of~$F$. The
extension of fields corresponding to~$\cell$, as per
theorem~\ref{thm-C-complex-fields}, is~$\C(z) / \C( F(z) )$. We will
write~$x= F(z)$ for simplicity, thus seeing the injection above as an
inclusion. If~$F = P/Q$, note that~$P(z) - x Q(z) = 0$, illustrating
that~$z$ is algebraic over~$\C(x)$.

Suppose that we had managed to find an~$F$ as above whose coefficients
are in~$\qb$. Then~$z$ is algebraic over~$\qb(x)$, and in this
case~$\C(z)_{rat}$ can be taken to be~$\qb (z)$.  We have
identified the extension~$\qb(z) / \qb (x)$ corresponding
to~$\cell$ as in theorem~\ref{thm-C-rational-fields}.

Now that theorem and the discussion preceding it do not, as stated,
claim that~$F$ can always be found with coefficients in~$\qb$: we
merely now that some primitive element~$y$ can be found with minimal
polynomial having its coefficients in~$\qb$. The stronger
statement is equivalent to~$\C(z)_{rat}$ being purely transcendental
over~$\qb$, as can be seen easily. Many readers will no doubt
be aware of abstract reasons why this must in fact always be the case;
we will now propose an elementary proof which, quite importantly, also
indicates how to find~$F$ explicitly in practice. The Galois action
will be brought in as we go along.

Let us first discuss the number of candidates for~$F$. Any two
rational fractions corresponding to~$\cell$ must differ by an
isomorphism in the category of Belyi pairs; that is, any such fraction
is of the form~$F(\phi(z))$ where~$F$ is one fixed solution and~$\phi
\colon \p \to \p$ is some isomorphism. Of course~$\phi$ must be a
Moebius transformation, $\phi(z) = (az+b)/(cz+d)$.  
%
%
Let us call a Belyi map~$F \colon \p \to \p$ {\em normalized} when
$F(0)=0$, $F(1)=1$ and~$F(\infty) = \infty$.

\begin{lem} \label{lem-finitely-many-normalized}
Let~$\cell$ be a dessin on the sphere. There are finitely many
normalized fractions corresponding to~$\cell$.
\end{lem}

\begin{proof}
The group of Moebius transformations acts simply transitively on
triples of points, so we can arrange for there to be at least one
normalized Belyi fraction, say~$F$, corresponding to~$\cell$. Other
candidates will be of the form~$F \circ \phi$ where~$\phi$ is a
Moebius transformation, so~$\phi(0)$ must be a root of~$F$
and~$\phi(1)$ must be a root of~$F-1$, while~$\phi(\infty)$ must be a
pole of~$F$. Since~$\phi$ is determined by these three values, there
are only finitely many possibilities.
\end{proof}

We shall eventually prove that any normalized fraction has its
coefficients in~$\qb$. 

Our strategy for finding a fraction~$F \colon \p \to \p$ which is a
Belyi map is to pay attention to the associated fraction 
\[ A = \frac{F'} {F(F-1)} \, .   \]

\begin{prop} \label{prop-properties-A}
Let~$F$ be a Belyi fraction such that~$F(\infty) = \infty$, and
let~$A$ be as above. Then the following holds.
\begin{enumerate}
\item The partial fraction decomposition of~$A$ is of the form 
\[ A = \sum_i \frac{m_i} {z - w_i}  ~ - ~ \sum_i \frac{n_i} {z - b_i}
\, ,   \]
where the~$n_i$'s and the~$m_i$'s are positive integers, the~$b_i$'s
are the roots of~$F$, and the~$w_i$'s are the roots of~$F-1$. In
fact~$n_i$ is the degree of the black vertex~$b_i$, and~$m_i$ is the
degree of the white vertex~$w_i$.

\item One can recover~$F$ from~$A$ as: 
\[ \frac{1} {F} = 1 - \frac{ \prod_i (z-w_i)^{m_i}} {\prod_i (z - b_i)^{n_i}} \, .   \]

\item The fraction~$A$ can be written in reduced form 
\[ A = \lambda \, \frac{ \prod_i (z - f_i)^{r_i-1}} { \prod_i(z - b_i) \, \prod_i
(z - w_i)} \, ,   \]
where the~$f_i$'s are the poles of~$F$ (other than~$\infty$),
and~$r_i$ is the multiplicity of~$f_i$ as a pole of~$F$. In fact~$r_i$
is the number of black triangles inside the face corresponding
to~$f_i$. 
\end{enumerate}

Conversely, let~$A$ be any rational fraction of the form given in (3),
with the numbers~$f_i$, $b_i$, $w_i$ distinct. Assume that~$A$ has a
partial fraction decomposition of the form given in (1); define~$F$ by
(2); and finally assume that the~$f_i$'s are poles of~$F$. Then~$F$ is
a Belyi map, $A = F'/( F(F-1))$, and we are in the previous situation.
\end{prop}

We submit a proof below. For the moment, let us see how we can use
this proposition to establish the results announced above. So
assume~$\cell$ is a given dessin on the sphere, and we are
looking for a corresponding normalized Belyi map~$F \colon \p \to
\p$. We look for the fraction~$A$ instead, and our ``unknowns'' are
the~$f_i$'s, the~$b_i$'s, the~$w_i$'s, and~$\lambda $, cf (3). Of
course we now the numbers~$r_i$ from counting the black triangles
on~$\cell$, just as we now the number of black vertices, white vertices,
and faces, giving the number of~$b_i$'s, $w_i$'s, and~$f_i$'s (keeping
in mind the pole at~$\infty$ already accounted for). 

Now comparing (3) and (1) we must have
\[ \lambda \, \frac{ \prod_i (z - f_i)^{r_i-1}} { \prod_i(z - b_i) \, \prod_i
(z - w_i)} = \sum_i \frac{m_i} {z -
  w_i}  ~ - ~ \sum_i \frac{n_i} {z - b_i} \tag{*} \]
where the integers~$n_i$ and~$m_i$ are all known, since they are the
degrees of the black and white vertices respectively, and again these
can be read from~$\cell$.

Further, the~$f_i$'s must be poles of~$F$, which is related to~$A$ by
(2). Thus we must have
\[  \prod_i (f_j-w_i)^{m_i} = \prod_i (f_j - b_i)^{n_i} \, , \tag{**} \]
for all~$j$. We also want~$F$ to be normalized so we pick
indices~$i_0$ and~$j_0$ and throw in the equations 
\[ b_{i_0} = 0 \, , \qquad w_{j_0} = 1 \, . \tag{***}  \]
Finally we want our unknowns to be distinct. The usual trick to
express this as an equality rather than an inequality is to take an
extra unknown~$\eta$ and to require 
\[ \eta (b_1 - b_2)(f_1 - f_2) \cdots =1\, , \tag{****}  \]
where in the dots we have hidden all the required differences. 

\begin{lem}
The system of polynomials equations given by (*), (**), (***) and
(****) has finitely many solutions in~$\C$. These solutions are all
in~$\qb$.
\end{lem}

\begin{proof}
By the proposition, each solution defines a normalized Belyi map, and
thus a dessin on the sphere. Define an equivalence relation on
the set of solutions, by declaring two solutions to be equivalent when
the corresponding dessins are isomorphic. By
lemma~\ref{lem-finitely-many-normalized}, there are finitely many
solutions in an equivalence class. However there must be finitely many
classes as well, since for each~$n$ there can be only a finite number
of dessins on~$n$ darts, clearly, and for all the solutions we
have~$n=\sum_i n_i$ darts. 

It is a classical fact from either algebraic geometry, or the theory
of Gr\"obner bases, that a system of polynomial equations with
coefficients in a field~$K$, having finitely many solutions in an
algebraically closed field containing~$K$, has in fact all its
solutions in the algebraic closure of~$K$. Here the equations have
coefficients in~$\q$.
\end{proof}

We may state, as a summary of the discussion:

\begin{prop}
A dessin~$\cell$ on the sphere defines, and is defined by, a rational
fraction~$F$ with coefficients in~$\qb$ which is also a Belyi map. The
dessin~$\act{\lambda }{\cell}$ corresponds to the
fraction~$\act{\lambda }{\cell}$ obtained by applying~$\lambda $ to
the coefficients of~$F$.
\end{prop}

\begin{ex} \label{ex-galois-action}
Suppose~$\cell$ is the following dessin on the sphere:

\figurehere{0.85}{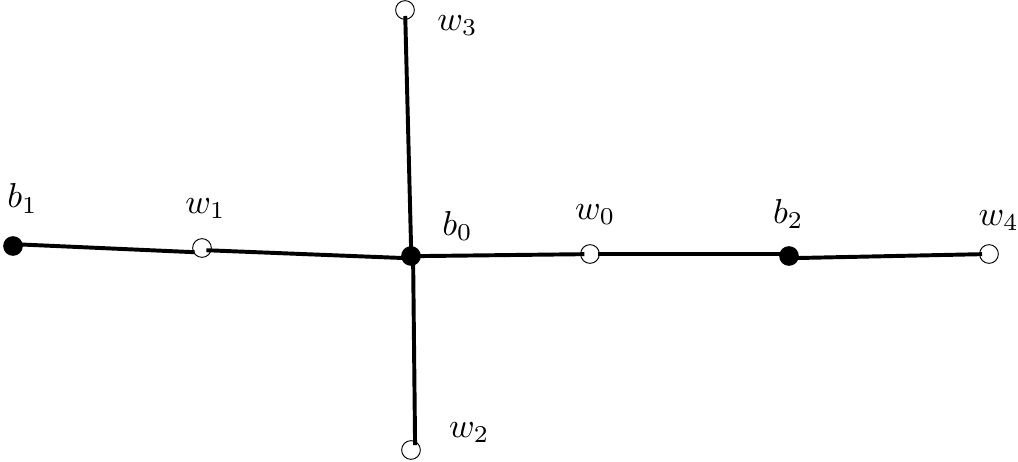}
Let us find a fraction~$F$ corresponding to~$\cell$ by the method just
described. Note that, whenever the dessin is really a planar
tree, one can greatly improve the efficiency of the computations, as
will be explained below, but we want to illustrate the general case. 

We point out that the letters~$b_i$ and~$w_i$ above are used to label
the sets~$B$ and~$W$, and the same letters will be used in the
equations which we are about to write down. A tricky aspect is that,
in the equations, there is really nothing to distinguish between, say,
$w_2$, $w_3$, and~$w_4$; and we expect more solutions to our system of
equations than the one we want. We shall see that some solutions will
actually give a different dessin. 

Here there is just one face, so~$F$ will have just the one pole
at~$\infty$; in other words~$F$ will be a polynomial. As for~$A$, it
is of the form 
\[ A= \frac{\lambda } {(z-b_0) (z-b_1) (z-b_2) (z-w_0)  (z-w_1)  (z-w_2)
      (z-w_2)  (z-w_4)} \, .   \]
The first equations are obtained by comparing this with the expression 
\[ A= -\frac{4} {z - b_0} - \frac{1} {z-b_1} - \frac{2} {z-b_2} +
\frac{2} {z-w_0} + \frac{2} {z-w_1} + \frac{1} {z - w_2} + \frac{1} {z -
  w_3} + \frac{1} {z - w_4} \, .   \]
There are no~$f_i$'s so no extra condition, apart from the one
expressing that the unknowns are distinct: 
\[ \eta (b_0 - b_1) \cdots (b_2 - w_3) \cdots = 1 \, ,   \]
where we do not write down the 28 terms. Finally, for~$F$ to be
normalized, we add 
\[ b_0 = 0 \, , \qquad w_0 = 1 \, .   \]

At this point we know that there must be a finite set of
solutions. This is confirmed by entering all the polynomial equations
into a computer, which produces exactly 8 solutions (using Groebner
bases). For each solution, we can also ask the computer to plot (an
approximation to) the set~$F^{-1}([0, 1])$.

\begin{tabular}{llll}
{\footnotesize 1} \includegraphics[width=0.25\textwidth]{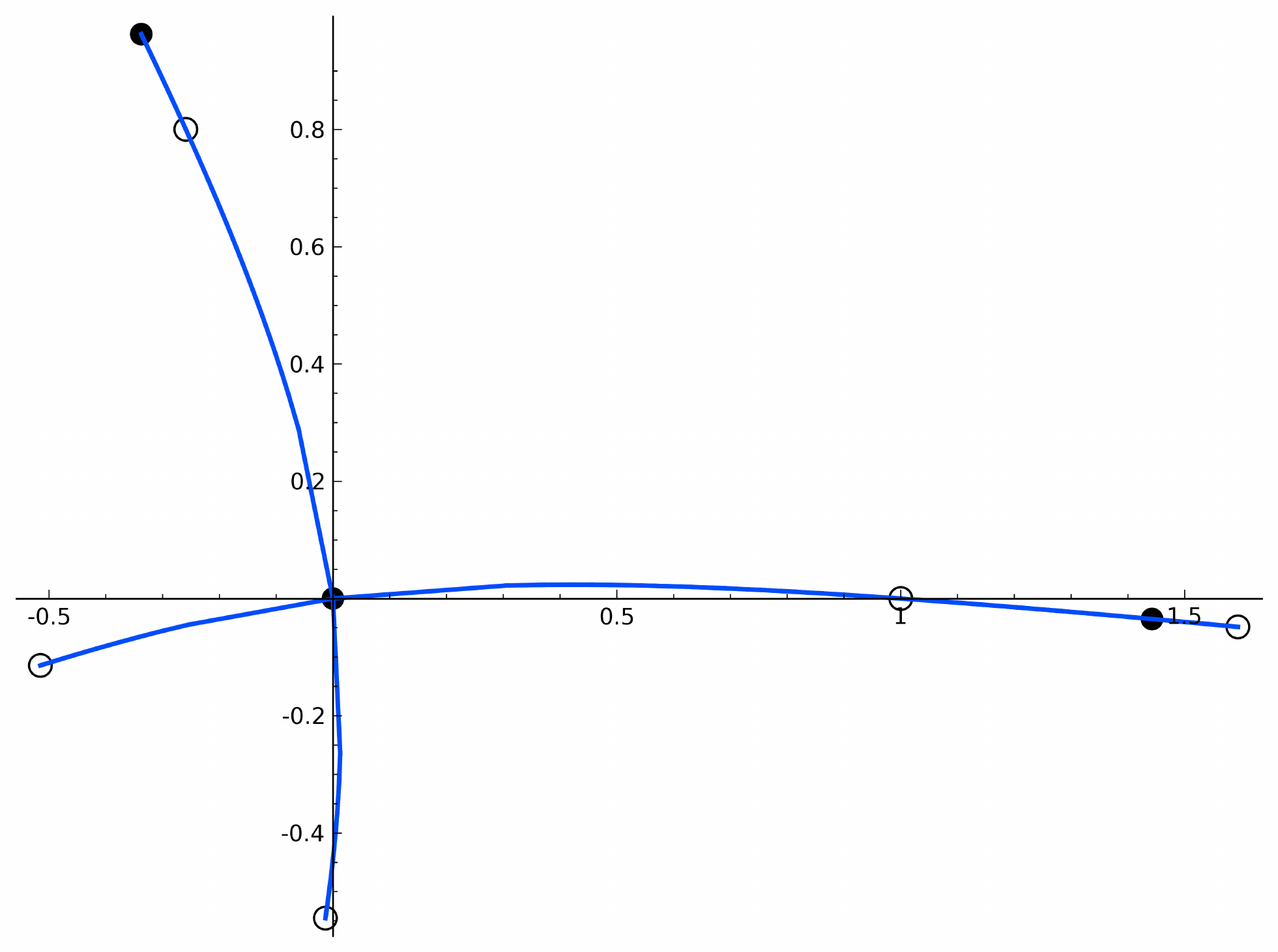} &
{\footnotesize 2} \includegraphics[width=0.25\textwidth]{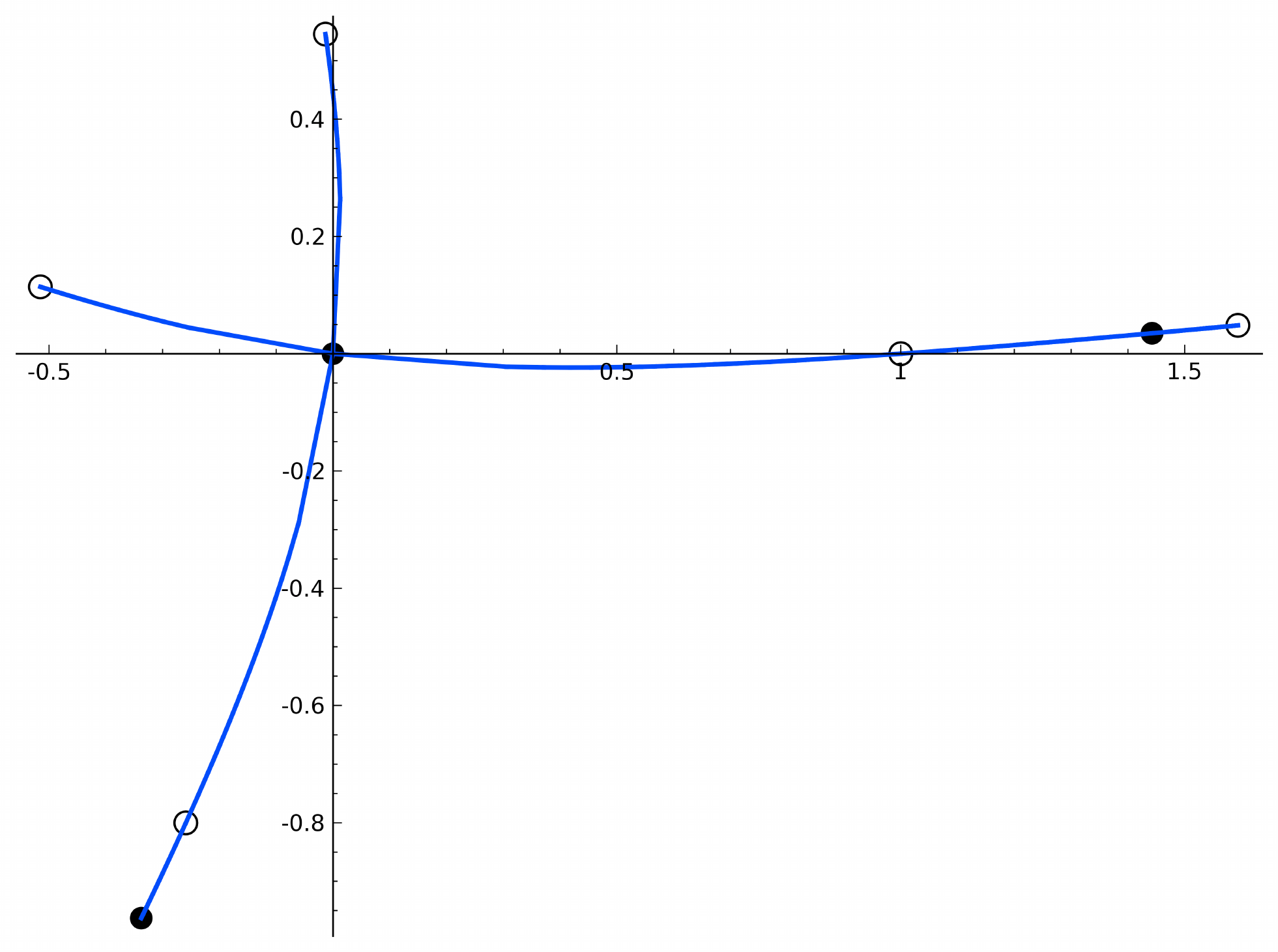} &
{\footnotesize 3} \includegraphics[width=0.25\textwidth]{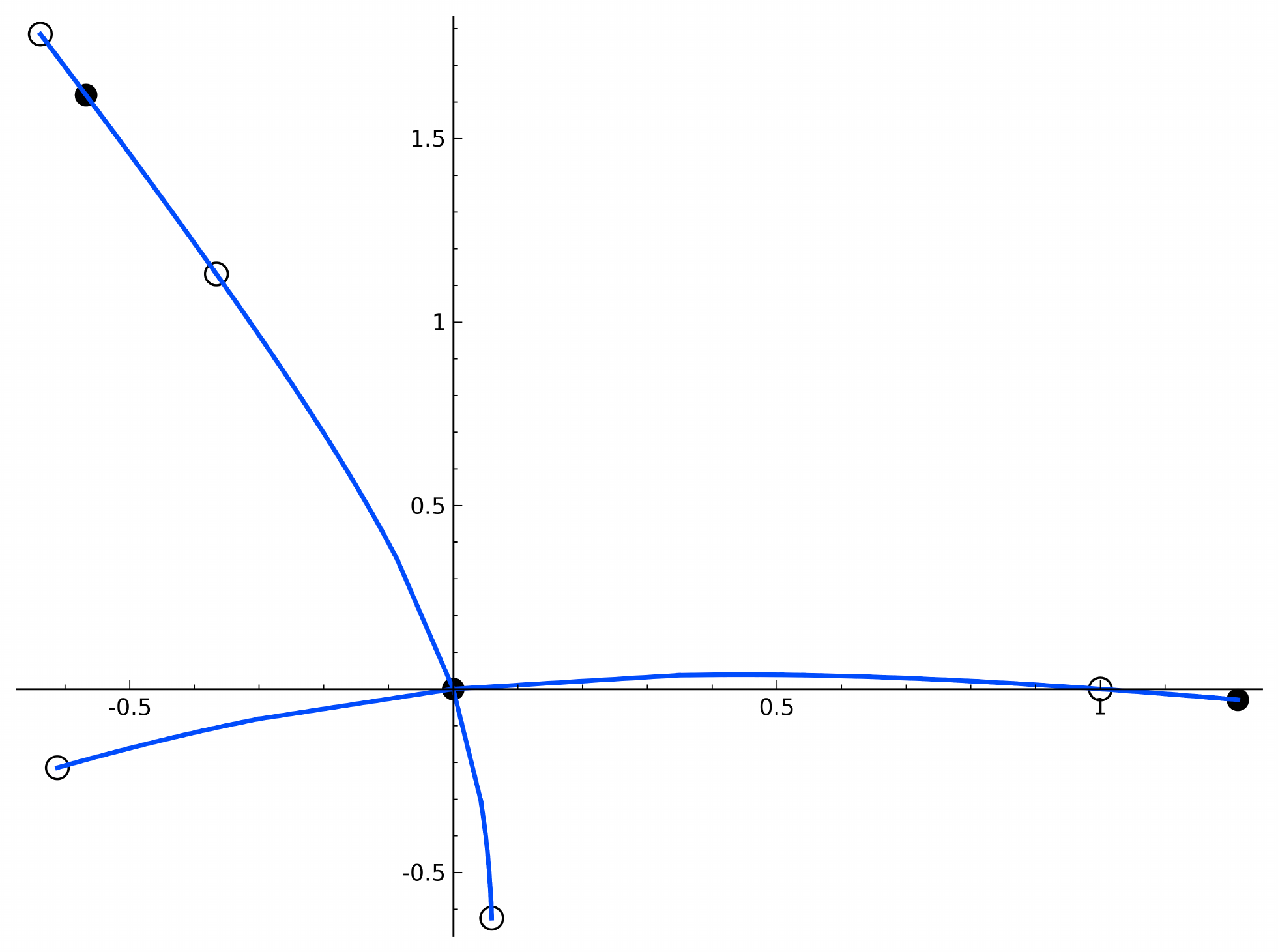} &
{\footnotesize 4} \includegraphics[width=0.25\textwidth]{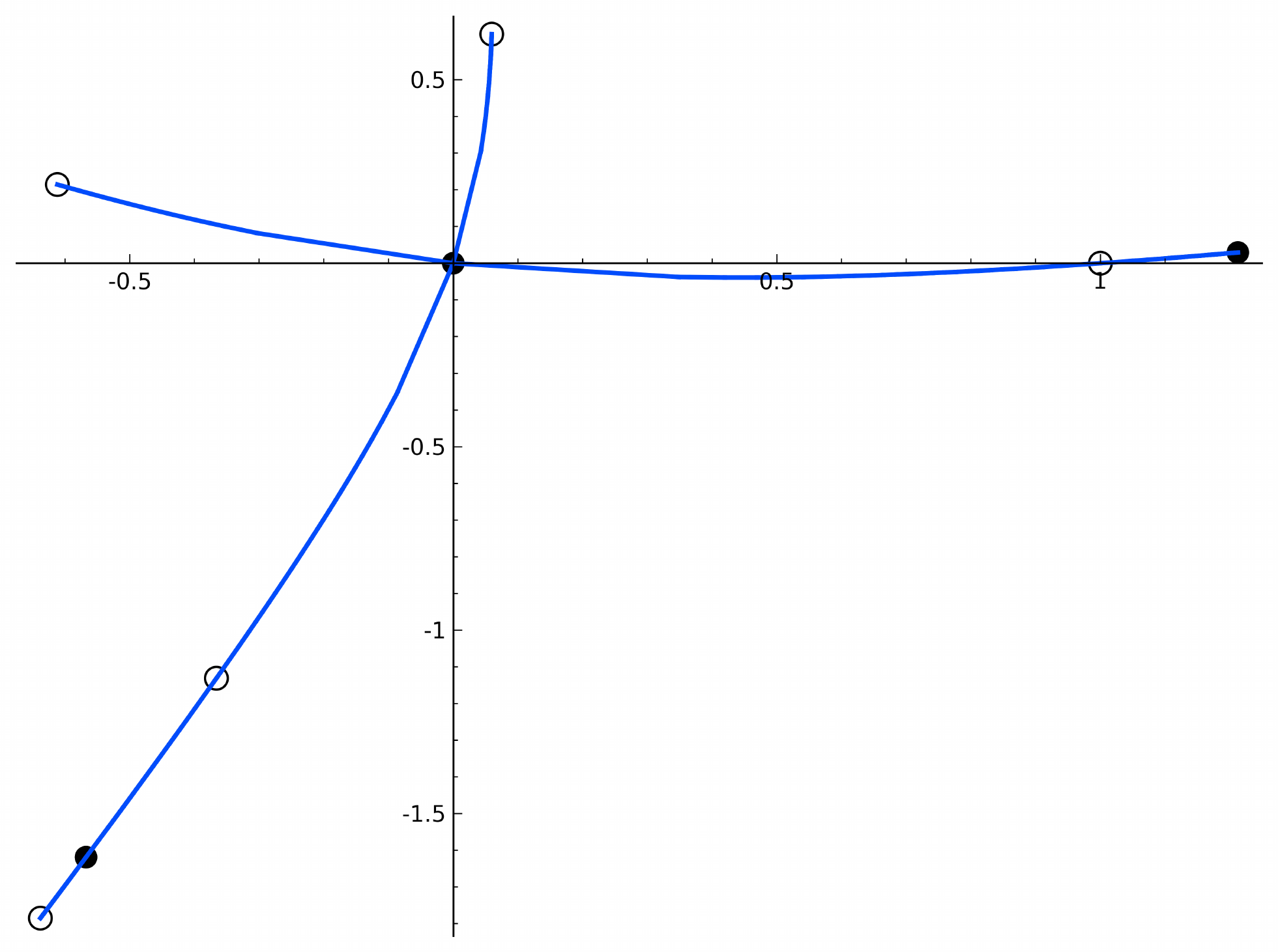} \\
{\footnotesize 5} \includegraphics[width=0.25\textwidth]{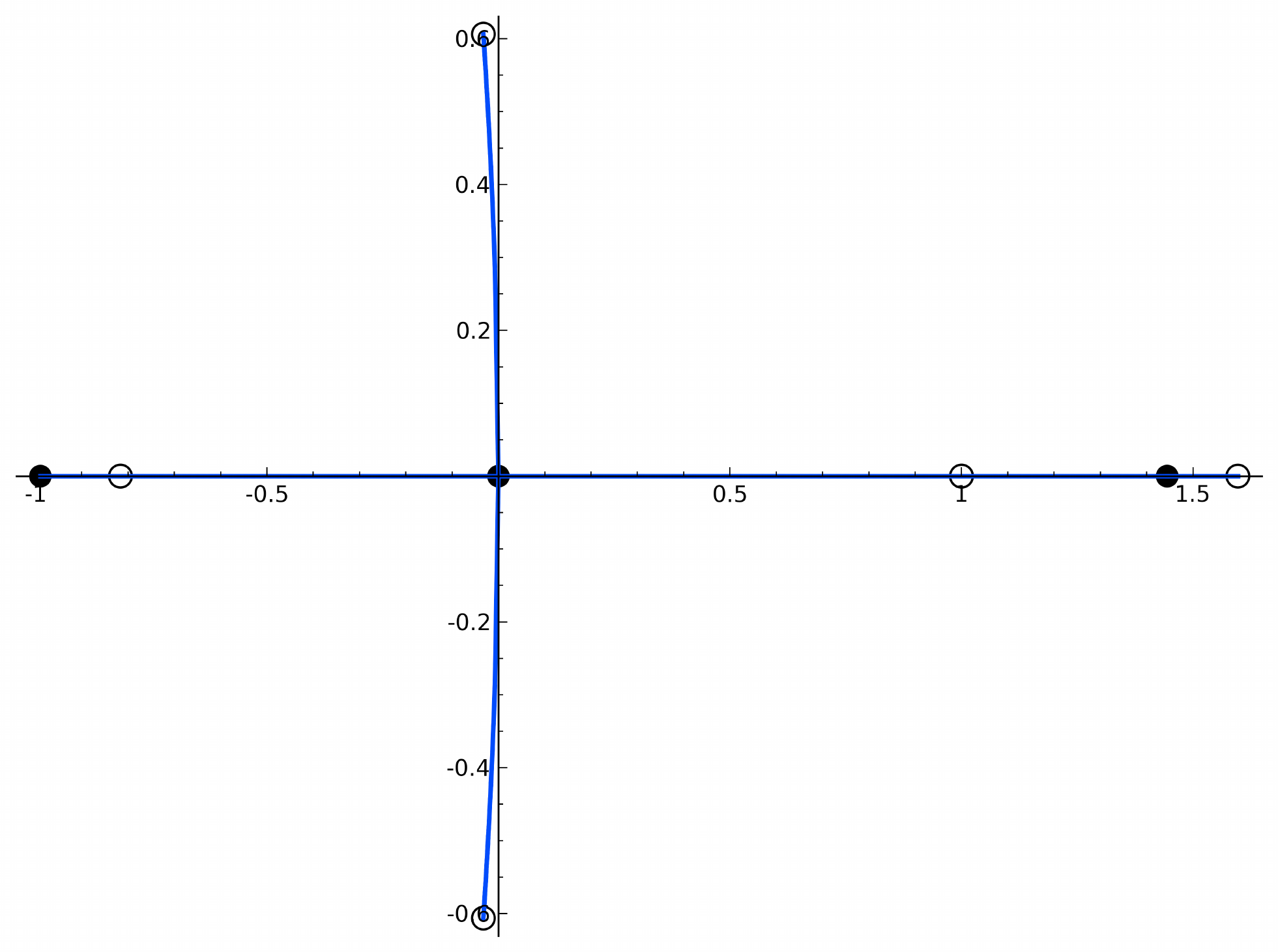} &
{\footnotesize 6} \includegraphics[width=0.25\textwidth]{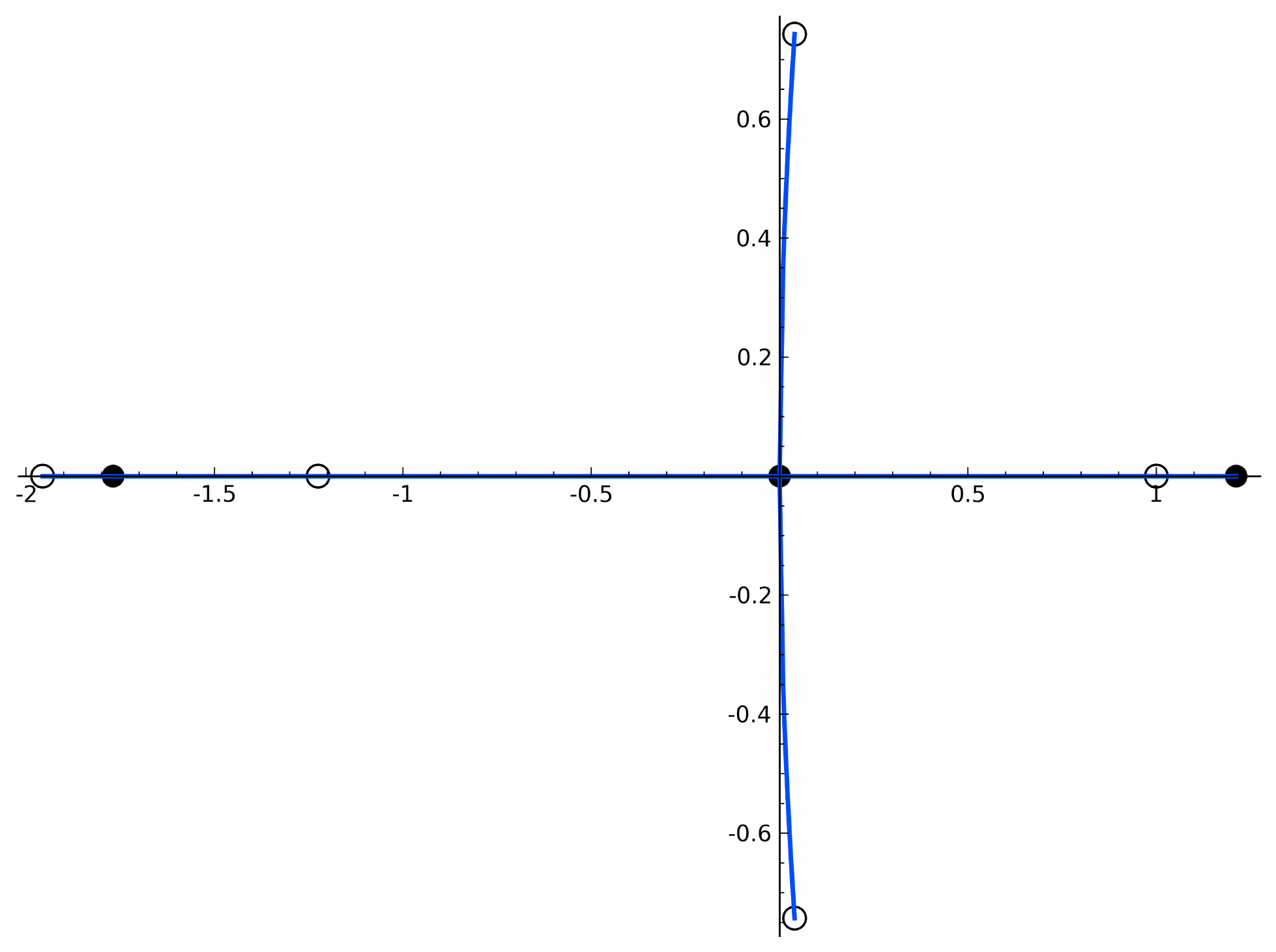} &
{\footnotesize 7} \includegraphics[width=0.25\textwidth]{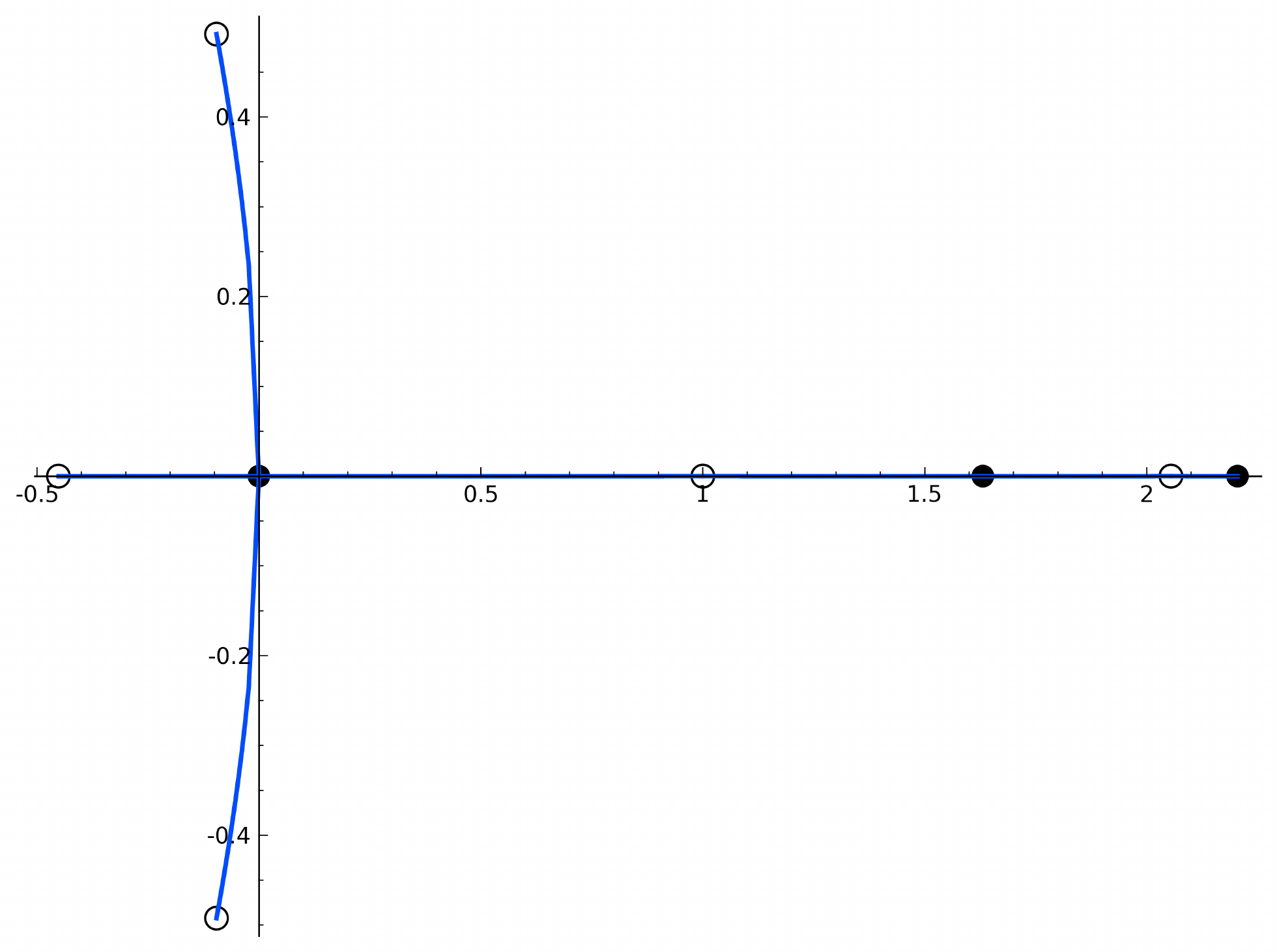} &
{\footnotesize 8} \includegraphics[width=0.25\textwidth]{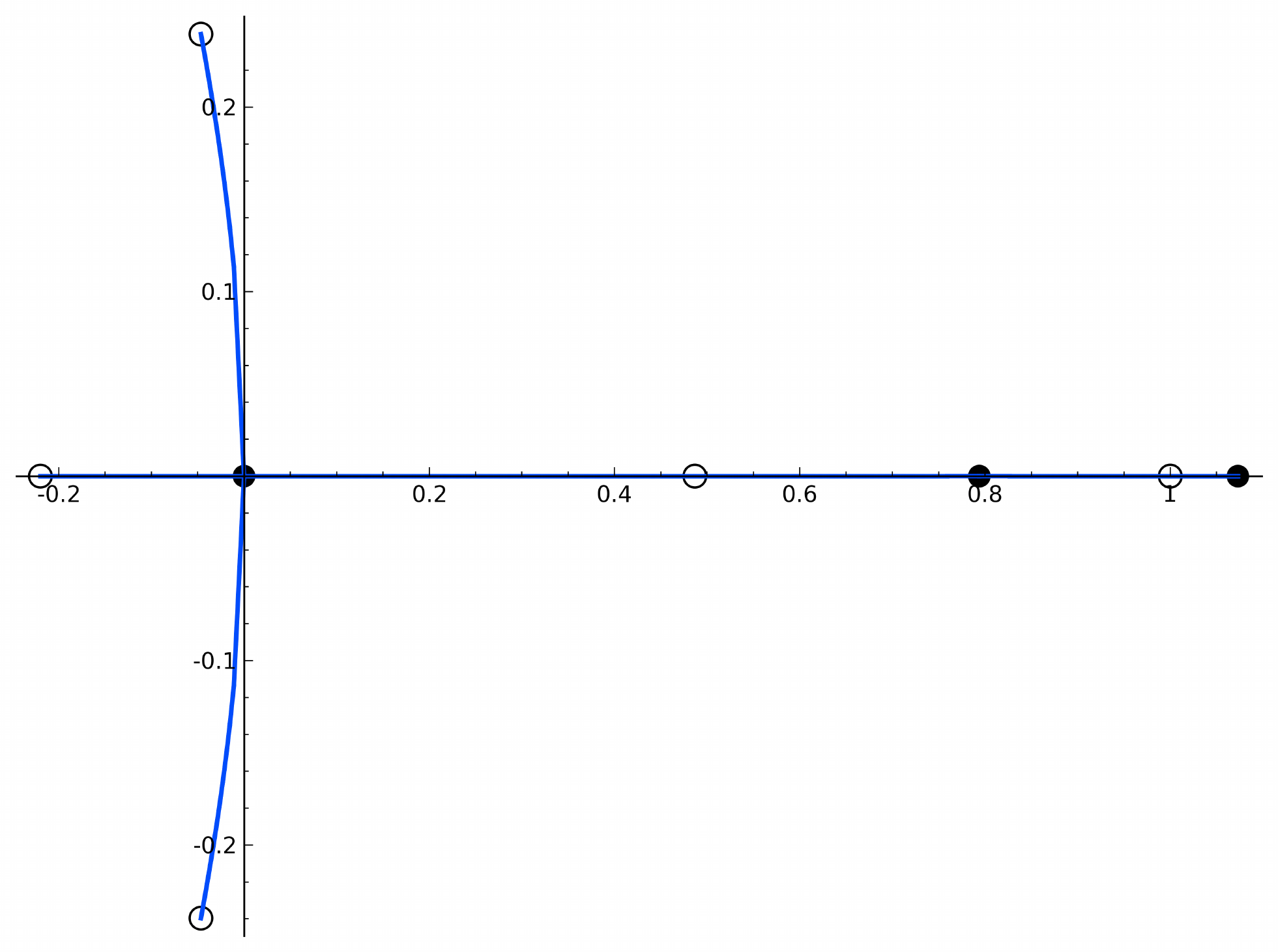} 
\end{tabular}

It seems that 5 and 6 look like our original dessin~$\cell$,
while the other six are certainly not isomorphic to~$\cell$ (even the
underlying bigraphs are not isomorphic to that of~$\cell$). Let us
have a closer look at 5 and 6:

\begin{tabular}{llll}
\includegraphics[width=0.4\textwidth]{sol0-big.pdf} & \includegraphics[width=0.4\textwidth]{sol5-big.pdf}
\end{tabular}

We see precisely what is going on: we have imposed the condition~$w_0=
1$, but in the equations there was nothing to distinguish the two
white vertices of degree two, and they can really both play the role
of~$w_0$. These two solutions give isomorphic dessins, though:
one diagram is obtained from the other by applying a rotation of
angle~$\pi$, that is~$z \mapsto -z$, and the two fractions are of the
form~$F(z)$ and~$F(-z)$ respectively. This could be confirmed by
calculations, though we will spare the tedious verifications.

The other solutions all come in pairs, for the same reason. Let us
have a closer look at 1, 3, 5, 7:

\begin{tabular}{ll}
{\footnotesize 1} \includegraphics[width=0.4\textwidth]{sol1-big.pdf} &
{\footnotesize 3} \includegraphics[width=0.4\textwidth]{sol4-big.pdf} \\
{\footnotesize 5} \includegraphics[width=0.4\textwidth]{sol0-big.pdf} & 
{\footnotesize 7} \includegraphics[width=0.4\textwidth]{sol6-big.pdf}
\end{tabular}

Here 1 and 3 present the same bicolored tree; 1, 5 and 7 are
non-isomorphic bicolored trees. However 1 and 3 are not isomorphic
dessins -- or rather, they are not isomorphic as oriented
dessins, as an isomorphism between the two would have to
change the orientation.

Let~$\sigma $, $\alpha $ and~$\phi$ be the three permutations
corresponding to~$\cell$. Now suppose we were to look for a
dessin~$\cell'$ with permutations~$\sigma '$, $\alpha '$
and~$\phi'$ such that~$\sigma' $ is conjugated to~$\sigma $
within~$S_7$ (there are 7 darts here), and likewise for~$\alpha '$
and~$\alpha $, and~$\phi '$ and~$\phi$. Then we would write down the
same equations, which only relied on the cycle types of the
permutations. Thus~$\cell'$ would show up among the solutions, and
conversely. So we have an interpretation of this family of four
dessins.


Let us have a look at the Galois action. Here is the number~$b_1$ in
the cases 1, 3, 5, 7:
\[ \frac{1}{32} \, {\left(-2 i \, \sqrt{5 i \, \sqrt{7} - 7} \sqrt{7}
  + 3 i \, \sqrt{2} \sqrt{7} + 7 \, \sqrt{2}\right)} \sqrt{2} \, ,   \]
\[ \frac{1}{32} \, {\left(2 \, \sqrt{5 i \, \sqrt{7} + 7} \sqrt{7} - 3
  i \, \sqrt{2} \sqrt{7} + 7 \, \sqrt{2}\right)} \sqrt{2} \, ,   \]
\[ -\frac{1}{72} \, {\left(\sqrt{8 \, \sqrt{3} \sqrt{7} + 63} \sqrt{3} \sqrt{7} - 21 \, \sqrt{3} + 12 \, \sqrt{7}\right)} \sqrt{3} \, ,   \]
\[ \frac{1}{72} \, {\left(\sqrt{-8 \, \sqrt{3} \sqrt{7} + 63} \sqrt{3}
  \sqrt{7} + 21 \, \sqrt{3} + 12 \, \sqrt{7}\right)} \sqrt{3} \, .   \]
%

One can check that the minimal polynomial for~$b_1$ in case 1 has
degree 4, and that the four distinct values for~$b_1$ in cases 1, 2,
3, 4 all have the same minimal polynomial (these are questions easily
answered by a computer). Thus they are the four roots of this
polynomial, which are in the same~$\gal$-orbit. On the other hand, in
cases 5, 6, 7, 8 the values for~$b_1$ have another minimal polynomial
(and they have the same one), so~$\gal$ cannot take solution 1 to any
of the solutions 5, 6, 7, 8. In the end we see that the four solutions
1, 2, 3, 4 are in the same Galois orbit, in particular 1 and 3 are in
the same orbit. A similar argument shows that 5 and 7 also belong to
the same orbit. However these orbits are different.

Understanding the action of the absolute Galois group of~$\q$ on
(isomorphism classes of) dessins will be a major theme in the
rest of this paper.
\end{ex}

\begin{rmk}
Let us comment of efficiency issues. A seemingly anecdotal trick,
whose influence on the computation is surprising, consists in grouping
the vertices of the same colour and the same degree. In the last
example, we would ``group together'' $w_2$, $w_3$ and~$w_4$, and write 
\[ (z-w_2)(z-w_3)(z-w_4) = z^3 + u z^2 + v z + s \, .   \]
All subsequent computations are done with the unknowns~$u$, $v$
and~$s$ instead of~$w_2$, $w_3$ and~$w_4$, thus {\em reducing the
  degree of the equations}. 

More significant is the alternative approach at our disposal when the
dessin is a planar tree. Then~$F$ is a polynomial (if we arrange
for the only pole to be~$\infty$), and~$F'$ {\em divides} $F(F-1)$,
so~$F(F-1)= P F' $, where everything in sight is a polynomial. 

Coming back to the last example, we would write 
\[ F = c z^4 (z-b_1) (z-b_2)^2  \]
(incorporating~$b_0 = 0$) and 
\[ F - 1 = c (z-1)^2 (z-w_1)^2 (z^3 + uz^2 + vz + s) \, ,   \]
the unknowns being now~$c, b_1, b_2, w_1, u, v$ and~$s$. In the very
particular case at hand, there is already a finite number of solutions
to the polynomial equations resulting from the comparison of the
expressions for~$F$ and~$F-1$. In general though, the very easy next
step is to compute the remainder in the long division of~$F(F-1)$
by~$F'$, say in~$\q(c, b_1, b_2, w_1, u, v, s)[z]$. Since~$F$ and~$F'$
both have~$c$ as the leading coefficient, it is clear that the result
will have coefficients in~$\q[c, b_1, b_2, w_1, u, v, s]$. These
coefficients must be zero, and these are the equations to consider. 

Proceeding in this way is, based on a handful of examples, several
orders of magnitude faster than with the general method.
\end{rmk}

We conclude with a proof of proposition~\ref{prop-properties-A}. 

\begin{proof}
Let~$F$ be as in the proposition, let~$A= F'/F(F-1)$, and let us write
the partial fraction decomposition of~$A$ over~$\C$: 
\[ A  = \sum_{\alpha , r, k} \frac{\alpha } {(z - r)^k} \, .   \]
Now we integrate; we do this formally, though it can be made rigorous
by restricting~$z$ to lie in a certain interval of real numbers. Note
that essentially we are solving the differential equation~$F(F-1) =
A^{-1} F'$. On the one hand:
\[ \int \frac{F'(z) dz} {F(z)(F(z)-1)} = \int \frac{dF} {F(F-1)} =
\int \left(\frac{-1} {F} + \frac{1} {F - 1}\right) dF =
\log(\frac{F-1} {F})  \, ,   \]
up to a constant. On the other hand this must be equal to 
\[ \sum_{\alpha, r, k > 1} \frac{\alpha }
   {(1-k)(z-r)^{k-1}} + \sum_{\alpha , r} \alpha \log(z-r) \, , \]
up to a constant. Thus the exponential of this last expression is a
rational fraction, from which it follows that the first sum above must
be zero. In other words, $k=1$ in all the nonzero terms of the partial
fraction decomposition of~$A$. Moreover, for the same reason
all~$\alpha $'s must be integers. In the end
\[ A = \sum_{\alpha, r} \frac{\alpha } {z-r} \, ,  \]
and 
\[ \frac{F-1} {F} = c \prod_{\alpha, r} (z-r)^\alpha \, .   \]
We rewrite this 
\[ \frac{1} {F} =  1 - c  \prod_{\alpha, r} (z - r)^\alpha \, .   \]
Examination of this expression establishes (1) and (2) simultaneously. Indeed $F(\infty) = \infty$ implies~$c=1$ (and~$\sum \alpha = 0$). Likewise, the roots of~$F$ are the numbers~$r$'s such that~$\alpha < 0$, and the roots of~$F-1$ are the~$r$'s such that~$\alpha > 0$. The multiplicities are interpreted as degrees of vertices, as already discussed (we see that~$\sum \alpha = 0$ amounts to~$\sum m_i = \sum n_i$, and as a matter a fact these two sums are equal to the number~$n$ of darts, each dart joining a back vertex and a white one). Let us now use the notation~$b_i$, $w_i$, $n_i$ and~$m_i$.

We have shown that 
\[ A = \lambda \frac{B} {\prod_i (z-b_i) (z-w_i)}  \]
where~$B$ is a monic polynomial. It remains, in order to prove (3), to
find the roots of~$B$ together with their multiplicities, knowing
that~$B$ does not vanish at any~$b_i$ or any~$w_i$. 

For this write~$F = P/Q$ with~$P, Q$ coprime polynomials, so that 
\[ A = \frac{P'Q + P Q'} {P(P-Q)} \, .   \]
If~$f_i$ is a root of~$Q$, with multiplicity~$r_i$, then it is a root
of~$P'Q + PQ'$ with multiplicity~$r_i-1$. Also, it is not a root
of~$P(P-Q)$, so in the end~$f_i$ is a root of~$B$ of
multiplicity~$r_i-1$. 

Finally, from the expression~$A = F'/F(F-1)$ we know that the roots
of~$A$ are to be found among the roots of~$F'$ and the poles
of~$F(F-1)$, that is the roots of~$Q$. So a root of~$A$ which is not a
root of~$Q$ would have to be a root of~$F'$. Now we use the fact
that~$F$ is a Belyi map: a root of~$F'$ is taken by~$F$ to~$0$ or~$1$,
so it is among the~$b_i$'s and the~$w_i$'s. These are not roots
of~$B$, as observed, so we have proved (3).

Now we turn to the converse, so we let~$A$ have the form in (3), we
suppose that (1) holds and define~$F$ by (2). From the arguments above
it is clear that~$A = F' / F(F-1)$.

Is~$F$ a Belyi map? For~$z_0$ satisfying~$F'(z_0) = 0$, we need to
examine whether the value~$F(z_0)$ is among~$0, 1, \infty$. Suppose
$F(z_0)$ is neither~$0$ nor~$1$. Then it is not a root of~$F(F-1)$, so
it is a root of~$A$. If we throw in the assumption that the roots
of~$A$ are poles of~$F$, it follows that~$F(z_0) = \infty$.
\end{proof}

\subsection{Examples in genus~$1$; faithfulness of the action}

Let us briefly discuss the Galois action in the language of curves, as
in~\S\ref{subsec-curves}. A dessin defines a curve~$C$, which can
be taken to be defined by homogeneous polynomial equations~$P_i = 0$
in projective space, where~$P_i$ has coefficients in~$\qb$. Also~$C$
comes equiped with a map~$F \colon C \to \p$, or equivalently~$F \in
\m(C)$, and~$F$ can be written as a quotient~$F= P/Q$ where~$P$
and~$Q$ are homogeneous polynomials of the same degree, again with
coefficients in~$\qb$. Conversely such a curve, assuming that~$F$ does
not ramify except possibly at~$0$, $1$ or~$\infty$, defines a
dessin.

It is then easy to show (though we shall not do it here)
that~$\act{\lambda }{\cell}$ corresponds to the curve~$\act{\lambda
}{C}$ obtained by applying~$\lambda $ to the coefficients of
each~$P_i$; it comes with a Belyi map, namely~$\act{\lambda }{F}$,
which we again obtain by applying~$\lambda $ to the coefficients
of~$F$. (Note in particular that~$\act{\lambda }{C}$, as a curve
without mention of a Belyi map, is obtained from~$\lambda $ and~$C$
alone, and~$F$ does not enter the picture.)

We illustrate this with dessins in degree~$1$. An {\em elliptic
  curve} is a curve~$C$ given in~$\mathbb{P}^2$ by a ``Weierstrass
equation'', that is, one of the form
\[ y^2z - x^3 - a x z^2 - b z^3 = 0 \, .   \]
Assuming we work over~$\qb$ or~$\C$, the surface~$C(\C)$ is then a
torus. One can show conversely that whenever~$C(\C)$ has genus 1, the
curve is an elliptic curve. 

The equation is of course not uniquely determined by the
curve. However one can prove that 
\[  j = 1728 (4a)^3 /  16 (4a^3 + 27 b^2)  \]
depends only on~$C$ up to isomorphism. (The notation is standard, with
1728 emphasized.) What is more, over an algebraically closed field we
have a converse: the number~$j$ determines~$C$ up to
isomorphism. Further, each number~$j \in K$ actually corresponds to an
elliptic curve over~$K$. These are all classical results, see for example~\cite{silverman}. 

Now we see that, in obvious notation, $j(\act{\lambda }{C}) =
\act{\lambda }{j(C)}$, with the following consequence. Given~$\lambda
\in \gal$ which is not the identity, there is certainly a number~$j
\in \qb$ such that~$\act{\lambda }{j} \ne j$. Considering the (unique)
curve~$C$ such that~$j(C) = j$, we can use Belyi's theorem to make
sure that it possesses a Belyi map~$F$ (it really does not matter
which, for our purposes), producing at least one
dessin~$\cell$. It follows that~$\act{\lambda }{\cell}$ is not
isomorphic to~$\cell$, and we see that the action of~$\gal$ on
dessins is {\em faithful}. 

As it happens, one can show that the action is faithful even when
restricted to genus~$0$, and even to plane trees. What is more, the
argument is easy and elementary, see the paper by
Schneps~\cite{leila}, who ascribes the result to Lenstra.  

We note for the record: 

\begin{thm} \label{thm-action-on-dessins-faithful}
The action of~$\gal$ on dessins is faithful. In fact, the action on
plane trees is faithful, as is the action on dessins of genus~$1$. 
\end{thm}

In this statement it is implicit that the image of a plane tree under the Galois action is another plane tree. Theorem~\ref{thm-galois-invariants} below proves this, and more. 

\subsection{Invariants}

We would like to find common features to the dessins~$\cell$
and~$\act{\lambda }{\cell}$, assumed connected for simplicity. First
and foremost, if~$L/\qb(x)$ corresponds to~$\cell$, one must observe
that there is the following commutative diagram:
\[ \begin{CD}
\qb(x)   @>\lambda >>     \qb(x) \\
  @VVV                          @VVV    \\
L = L \otimes_{id} \qb(x)    @>{y \otimes s \mapsto y \otimes
  \lambda (s)}>>    L \otimes_\lambda \qb (x) = \act{\lambda }{L}
\, . 
\end{CD}
  \]
Here both horizontal arrows are isomorphisms of fields (but the bottom
one is not an isomorphism of~$\qb(x)$-extensions, of course). It
follows that there is an isomorphism
\[ \lambda ^* \colon \Gal(L/\qb(x)) \longrightarrow \Gal(\act{\lambda
}{L} / \qb(x)  ) \, ,   \]
obtained by conjugating by the bottom isomorphism (this is the
approach taken in~\cite{helmut}). Alternatively, the existence of a
homomorphism~$\lambda^*$ between these groups is guaranteed by the
functoriality of the Galois action; while the fact that~$\lambda^*$ is
a bijection is established by noting that its inverse is~$(\lambda
^{-1})^*$. The two definitions of~$\lambda^*$ agree, as is readily
seen. 

For the record, we note:

\begin{lem}
If~$\cell$ is regular, so is~$\act{\lambda }{\cell}$. 
\end{lem}

\begin{proof}
It is clear that~$\cell$ and~$\act{\lambda }{\cell}$ have the same
degree, and their automorphism groups are isomorphic
under~$\lambda^*$, so the lemma is obvious.
\end{proof}

Using curves, we can guess a property of~$\lambda ^*$ which is
essential (a rigorous argument will be given next). Let~$C$ be a curve
in projective space corresponding to~$\cell$. It is a consequence of
the material in~\S\ref{sec-eq-cats} that~$C(\C)$ is homeomorphic
to~$\topo{\cell}$. The automorphism~$\tilde \sigma $ must then
correspond to a self-map~$C \to C$, and the latter must fix a black
vertex by proposition~\ref{prop-fixed-points}. This black vertex has
its coordinates (in projective space) lying in~$\qb$.

Now, this map~$\tilde \sigma \colon C \to C$ is a map of curves
over~$\qb$, and so is given, at least locally, by rational fractions
with coefficients in~$\qb$. Applying~$\lambda^*$ amounts to
applying~$\lambda $ to these coefficients. Thus we get a
map~$\lambda^*( \tilde \sigma ) \colon \act{\lambda }{C} \to
\act{\lambda }{C}$, and clearly it also has a fixed point. By
proposition~\ref{prop-fixed-points} again, we see that~$\lambda^*(
\tilde \sigma )$ must be a power of the distinguished
generator~$\tilde \sigma_\lambda $ (in suggestive notation). Likewise
for~$\lambda^*(\tilde \alpha )$ and~$\lambda^*(\tilde \phi)$.

With a little faith, one may hope that the map~$C \to C$, having a
fixed point, looks like~$z \mapsto \zeta z$ in local coordinates,
where~$\zeta$ is some root of unity. If so, the power of~$\tilde
\sigma _\lambda $ could be found by examining the effect of~$\lambda $
on roots of unity, and we may hope that it is the same power
for~$\tilde \sigma$, $\tilde \alpha $ and~$\tilde \phi$.

Exactly this is true. The result even has an easy and elementary
proof, that goes {\em via} fields.

\begin{prop}[Branch cycle argument] \label{prop-branch-cycle-argument}
Assume that~$\cell$ is regular, and let~$\tilde \sigma , \tilde \alpha
$ and~$\tilde \phi$ be a distinguished triple for~$\Gal(L/\qb(x))
\cong Aut(\cell)$. Let~$n$ be the degree of~$\cell$, let~$\zeta_n =
e^{\frac{2 i \pi} {n}}$, and let~$m$ be such that 
\[ \lambda^{-1} (\zeta_n) = \zeta_n^m \, .   \]
Finally, let~$\tilde \sigma_{\lambda}, \tilde \alpha_{\lambda}$
and~$\tilde \phi_{\lambda}$ be a distinguished triple
for~$\Gal(\act{\lambda }{L}/ \qb(x))$.

Then~$\lambda ^*( \tilde \sigma^m )$ is conjugated to~$\tilde
\sigma_{\lambda}$, while~$\lambda ^*( \tilde \alpha^m )$ is conjugated
to~$\tilde \alpha _{\lambda}$ and $\lambda ^*( \tilde \phi^m )$ is
conjugated to~$\tilde \phi_{\lambda}$.
\end{prop}

\begin{proof}
This is lemma 2.8 in~\cite{helmut}, where it is called ``Fried's
branch cycle argument''. The following comments may be helpful. In
{\em loc.\ cit.}, this is stated using the ``conjugacy classes
associated with~$0, 1, \infty$''; in the addendum to theorem 5.9,
these are identified with the ``topological conjugacy classes
associated with~$0, 1, \infty$''; and we have already observed (after
proposition~\ref{prop-fixed-points}) that they are the conjugacy
classes of~$\tilde \sigma , \tilde \alpha , \tilde \phi$. 
\end{proof}

We should pause to compare this with
proposition~\ref{prop-regular-dessins-are-groups-with-generators},
which states that a regular dessin, up to isomorphism, is nothing
other than a finite group~$G$ with two distinguished
generators~$\sigma, \alpha$ (and~$\phi= (\sigma \alpha )^{-1}$ is
often introduced to clarify some formulae). Let us see the
map~$\lambda ^*$ as an identification (that is, we pretend that it is
the identity). Then the action of~$\lambda $ on~$(G, \sigma, \alpha )$
produces the same group, with two new generators, which are of the
form~$g \sigma^m g^{-1}$ and~$h \alpha^m h^{-1}$; moreover, if we call
these~$\sigma_\lambda $ and~$\alpha_\lambda $ respectively,
then~$\phi_\lambda = (\sigma_\lambda \alpha_\lambda )^{-1}$ is
conjugated to~$\phi^m$.

Of course, not all random choices of~$g, h, m$ will conversely produce
new generators for~$G$ by the above formulae. And not all recipes for
producing new generators out of old will come from the action of
a~$\lambda \in \gal$. Also note that, if~$g = h$ and~$m=1$,
that is if we simply conjugate the original generators, we get an
object isomorphic to the original dessin -- more generally when
there is an automorphism of~$G$ taking~$\sigma $ to~$\sigma_\lambda $
and~$\alpha $ to~$\alpha_\lambda $, then~$\act{\lambda }{\cell} \cong
\cell$.  

One further remark. In~$\sets_{\sigma, \alpha, \phi}$, the regular
dessin~$\cell$ is modeled by the set~$X = Aut(\cell)$ with the
distinguished triple~$\tilde \sigma, \tilde \alpha, \tilde \phi$
acting by right multiplication; similarly for~$\act{\lambda }
{\cell}$. Now, if we simply look at~$X$, and its
counterpart~$\act{\lambda }{X}$, in the category of
sets-with-an-action-of-a-group, that is if we forget the specific
generators at our disposal, then~$X$ and~$\act{\lambda }{X}$ become
impossible to tell apart, by the discussion above. 

We expand on this idea in the next theorem, where we make no
assumption of regularity.

\begin{thm} \label{thm-galois-invariants}
Let~$\cell$ be a compact, connected, oriented dessin without
boundary, and let~$\lambda \in \gal$.
\begin{enumerate}
\item $\cell$ and~$\act{\lambda }{\cell}$ have the same degree~$n$.
\item It is possible to number the darts of~$\cell$ and~$\act{\lambda
}{\cell}$ in such a way that these two dessins have precisely
  the same cartographic group~$G \subset S_n$. 

\item Let~$m$ be such that~$\lambda^{-1} (\zeta_N) = \zeta_N^m$,
  where~$N$ is the order of~$G$ and~$\zeta_N = e^{\frac{2 i \pi}
    {N}}$. Then within~$G$, the generator~$\sigma_\lambda $ is
  conjugated to~$ \sigma^m$, while~$\alpha _\lambda $ is conjugated
  to~$\alpha^m$ and~$\phi_\lambda $ is conjugated to~$\phi^m$.

\item Within~$S_n$, the generator~$\sigma_\lambda $ is conjugated to~$
  \sigma$, while~$\alpha _\lambda $ is conjugated to~$\alpha$
  and~$\phi_\lambda $ is conjugated to~$\phi$.

\item $\cell$ and~$\cell'$ have the same number of black vertices of a
  given degree, white vertices of a given degree, and faces of a given
  degree. 

\item The automorphism groups of~$\cell$ and~$\act{\lambda }{\cell}$
  are isomorphic.

\item The surfaces~$\topo{\cell}$ and~$\topo{\act{\lambda }{\cell}}$
  are homeomorphic.

\end{enumerate}
\end{thm}

There is an ingredient in the proof that will be used again later, so
we isolate it:

\begin{lem} \label{lem-galois-correspondence}
Let~$\cell$ be a regular dessin, and let~$\cell'$ be the intermediate
dessin corresponding to the subgroup~$H$
of~$Aut(\cell)$. Then~$\act{\lambda }{\cell}$ is regular,
and~$\act{\lambda }{\cell'}$ is its intermediate dessin corresponding
to the subgroup~$\lambda^*(H)$. 
\end{lem}

\begin{proof}
This is purely formal, given that the action of~$\lambda $ is via a
self-equivalence of the category~$\etq$ which preserve degrees (this
is the first point of the proposition, and it is obvious!). Clearly
regular objects must be preserved. If~$K/\qb(x)$ is an intermediate
extension of~$L/\qb(x)$ corresponding to~$H$, then the elements of~$H$
are automorphisms of~$L$ fixing~$K$, so the elements of~$\lambda^*(H)$
are automorphisms of~$\act{\lambda }{L}$ fixing~$\act{\lambda
}{K}$. Comparing degrees we see that~$\lambda^*(H)$ is precisely the
subgroup corresponding to~$\act{\lambda }{K}$.
\end{proof}

\begin{proof}[Proof of the theorem]
We need a bit of notation. Let~$\tilde \cell$ be the regular cover
of~$\cell$. Let us pick a dart~$d$ of~$\cell$ as a base-dart. This
defines an isomorphism between the cartographic group~$G$
and~$Aut(\tilde \cell)$, under which~$\sigma $ is identified
with~$\tilde \sigma $, and likewise for~$\alpha $ and~$\phi$. Finally,
let~$H$ be the stabilizer of the dart~$d$, so that in~$\sets_{\sigma,
  \alpha, \phi}$ our dessin is the object~$H \bs G$. The
subgroup~$H$ of~$G$ corresponds to~$\cell$ in the ``Galois
correspondence'' for~$\tilde \cell$.

By the lemma, $\act{\lambda }{\tilde \cell}$ is the regular closure
of~$\act{\lambda }{\cell}$, and the latter corresponds to the
subgroup~$\lambda^{*}(H)$. Therefore in~$\sets_{\sigma, \alpha, \phi}$
we can represent~$\act{\lambda }{\cell}$ by~$\lambda ^*(H) \bs
\lambda^*(G)$. In the category of~$G$-sets, this is isomorphic to~$H
\bs G$ {\em via}~$\lambda^*$. If we use the bijection~$H \bs G \to
\lambda^*(H) \bs \lambda^*(G)$ in order to number the elements
of~$\lambda^*(H) \bs \lambda^*(G)$, then we have arranged things so
that the cartographic groups for~$\cell$ and~$\act{\lambda }{\cell}$
coïncide as subgroups of~$S_n$.

This proves (1) and (2). Point (3) is a reformulation of the previous
proposition. To establish (4), we note that~$m$ is prime to the
order~$N$ of~$G$, and in particular it is prime to the order
of~$\sigma$. In this situation~$\sigma^m$ has the same cycle-type
as~$\sigma $ and is therefore conjugated to~$\sigma $
within~$S_n$. Likewise for~$\alpha $ and~$\phi$. Those cycle-types
describe the combinatorial elements refered to in (5).

Point (6) follows since the automorphism groups of~$\cell$
and~$\act{\lambda }{\cell}$ are both isomorphic to the centralizer
of~$G$ in~$S_n$.

Finally, point (7) is obtained by comparing Euler characteristics, as
in remark~\ref{rmk-euler-char}. 
\end{proof}

\begin{ex} 
We return to example~\ref{ex-galois-action}. While looking for an
explicit Belyi map, we found four candidates, falling into two Galois
orbits. Let us represent them again, with a numbering of the darts. 

\figurehere{1}{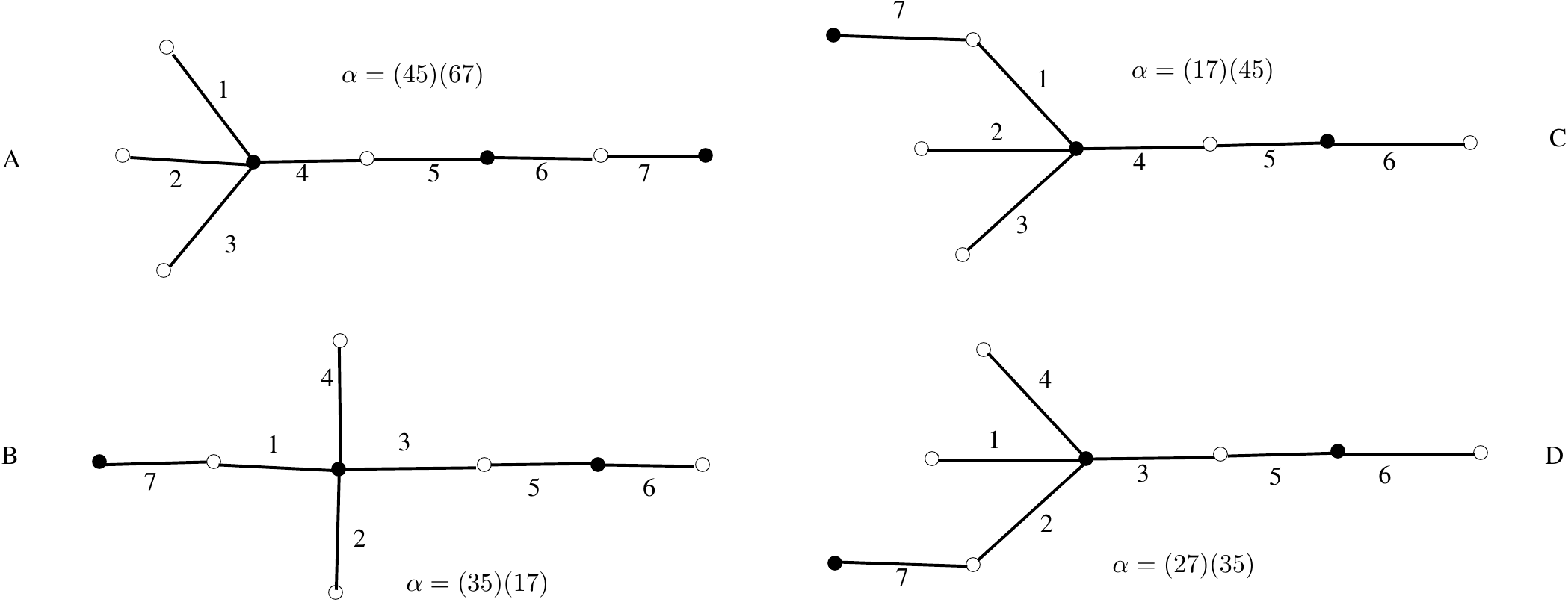}

In all four cases one has~$\sigma = (1234)(56)$, while~$\alpha $ is
given on the pictures. The following facts are obtained by asking
GAP: in cases A and B, the group generated by~$\sigma $ and~$\alpha $
is the alternating group~$A_7$ (of order 2520); in cases~$C$ and~$D$,
we get a group isomorphic to~$PSL_3(\mathbb{F}_2)$ (of order
168). This prevents A and B from being in the same orbit as C or D,
by the theorem, and suggests that A and B form one orbit, C and D
another. We have seen earlier that this is in fact the case.

Note that the cartographic groups for A and B are actually the same
subgroups of~$S_7$, and likewise for C and D. The theorem asserts
that this can always be arranged, though it does not really provide an
easy way of making sure that a numbering will be correct. With random
numberings of the darts, it is a consequence of the theorem that
the cartographic groups will be conjugated. In general the conjugation
will not preserve the distinguished generators, unless the two dessins
under consideration are isomorphic, cf
theorem~\ref{prop-iso-classes}.  

\end{ex}

\section{Towards the Grothendieck-Teichmüller group} \label{sec-GT}

In this section we define certain finite groups~$H_n$ for~$n \ge 1$,
and prove that there is an injection 
\[ \gal \longrightarrow \lim_n Out(H_n) \, .   \]
We further prove that the image lies in a certain subgroup, which we
call~$\GT$ and call the coarse Grothendieck-Teichmüller group. The
group~$\GT$ is an inverse limit of finite groups, and one can compute
approximations for it in finite time.

Beside these elementary considerations, we shall also use the language
of profinite groups, which has several virtues. It will show that our
constructions are independent of certain choices which seem arbitrary;
it will help us relate our construction to the traditional literature
on the subject; and it will be indispensable to prove a refinement of
theorem~\ref{thm-action-on-dessins-faithful}: the action of~$\gal$ on the set
of {\em regular} dessins is also faithful.

\subsection{The finite groups~$H_n$} \label{subsec-Hn}

Let~$F_2$ denote the free group on two generators, written~$\sigma $
and~$\alpha $. We encourage the reader to think of~$F_2$
simultaneously as~$\langle \sigma, \alpha \rangle$ and~$\langle
\sigma, \alpha, \phi ~|~ \sigma \alpha \phi = 1 \rangle$.

For any group~$G$ we shall employ the notation~$G ^{(n)}$ to denote
the intersection of all normal subgroups of~$G$ whose index is~$\le
n$. We define then~$H_n = F_2 / F_2 ^{(n)}$.  It is easily seen
that~$H_n$ is a finite group; moreover the intersection of all the
normal subgroups of~$H_n$ of index~$\le n$ is trivial, that is~$H_n
^{(n)} = \{ 1 \}$.

 In fact~$H_n$ is universal among the groups sharing these properties,
 as the following proposition makes precise (it is extracted
 from~\cite{helmut}, see \S7.1). The proof is essentially trivial.

\begin{prop} \label{prop-props-of-Hn}
\begin{enumerate}
\item For any finite group~$G$ of order~$\le n$ and~$g_1, g_2 \in G$,
  there is a homomorphism~$H_n \to G$ sending~$\sigma $ to~$g_1$
  and~$\alpha $ to~$g_2$.
\item If~$g_1, g_2$ are generators of a group~$G$ having the property
  that~$G ^{(n)} = \{ 1 \}$, then there is a surjective map~$H_n \to
  G$ sending~$\sigma $ to~$g_1$ and~$\alpha $ to~$g_2$.
\item If~$h_1, h_2$ are generators of~$H_n$, there is an automorphism
  of~$H_n$ sending~$\sigma $ to~$h_1$ and~$\alpha $ to~$h_2$.
\end{enumerate}
\end{prop}

(Here we have written~$\sigma $ and~$\alpha $ for the images in~$H_n$
of the generators of~$F_2$.)

In particular, there is a surjective map~$H_{n+1} \to H_n$. The kernel
of this map is~$H_{n+1} ^{(n)}$, which is characteristic ; it follows
that we also have maps~$Aut(H_{n+1}) \to Aut(H_n)$ as well as
$Out(H_{n+1}) \to Out(H_n) $.

Here is a concrete construction of~$H_n$. Consider all triples~$(G,
x, y)$ where~$G$ is a finite group of order~$\le n$ and~$x, y$ are
generators for~$G$, and consider two triples~$(G, x, y)$ and~$(G', x',
y')$ to be isomorphic when there is an isomorphism~$G \to G'$
taking~$x$ to~$x'$ and~$y$ to~$y'$. Next, pick representatives for the
isomorphism classes, say~$(G_1, x_1, y_1), \ldots, (G_N, x_N,
y_N)$. By the material above, this is equivalent to classifying all
the regular dessins on no more than~$n$ darts. Consider then 
\[ U = G_1 \times \cdots \times G_N \, ,   \]
and its two elements~$\sigma = (x_1, \ldots, x_N)$ and~$\alpha = (y_1,
\ldots, y_N)$. The subgroup~$K$ of~$U$ generated by~$\sigma $
and~$\alpha $ is then isomorphic to~$H_n$. Indeed, if~$G$ is any group
generated by two elements~$g_1, g_2$ satisfying~$G ^{(n)} = \{ 1 \}$,
by considering the projections from~$G$ to its quotients of order~$\le
n$ we obtain an injection of~$G$ into~$U$; under this injection~$g_1$,
resp.\ $g_2$, maps to an element similar to~$\sigma $, resp.\ $\alpha
$, except that some entries are replaced by~$1$'s, for those
indices~$i$ such that~$G_i$ is not a quotient of~$G$. As a result
there is a projection~$K \to G$ sending~$\sigma $ to~$g_1$ and~$\alpha
$ to~$g_2$. Since~$K$ satisfies the ``universal'' property (2) of
proposition~\ref{prop-props-of-Hn}, just like~$H_n$ does, these two
groups must be isomorphic.

The finite groups~$H_n$ will play a major role in what follows. Variants are possible: other collections of quotients of~$F_2$ could have been chosen, and we comment on this in \S\ref{subsec-variants}. We shall presently use the language of profinite groups, which allows a reformulation which is plainly independent of choices. Yet, in the sequel where elementary methods are preferred, and whenever we attempt a computation in finite time, the emphasis is on~$H_n$ or the analogous finite groups. The use of profinite groups is necessary, however, to prove theorem~\ref{thm-action-out-faithful}.

\begin{lem} \label{lem-inverse-limit-Hn}
The inverse limit~$\lim_n H_n$ is isomorphic to~$\hat F_2$, the
profinite completion of~$F_2$.
\end{lem}

\begin{proof}
By definition the profinite completion is 
\[ \hat F_2 = \lim F_2/N  \]
where the inverse limit is over all normal subgroups~$N$ of finite
index. Each such~$N$ contains some~$F_2 ^{(n)}$ for~$n$ large enough,
so the collection of subgroups~$F_2 ^{(n)}$ is ``final'' in the
inverse limit, implying the result.
\end{proof}

\begin{lem} \label{lem-out-f2-is-inv-lim-out-hn}
There is an isomorphism~$Out(\hat F_2) \cong \lim_n Out(H_n)$. 
\end{lem}

Note that~$Out(\hat F_2)$ is, by definition, $Aut_c(\hat F_2) / Inn(\hat
F_2)$ where~$Aut_c(\hat F_2)$ is the group of {\em continuous}
automorphisms of~$\hat F_2$. The proof will give a description
of~$Aut_c(\hat F_2)$ as an inverse limit of finite groups.

\begin{proof}
We will need the fact that normal subgroups of finite index in~$F_2$
are in bijection with open, normal subgroups of~$\hat F_2$ (which are
automatically closed and of finite index), under the closure
operation~$N \mapsto \bar N$: in fact the quotient map~$F_2 \to F_2/N$
extends to a map~$\hat F_2 \to F_2 /N$ whose kernel is~$\bar N$. It
follows easily that~$\bar N_1 \cap \bar N_2 = \overline{N_1 \cap
  N_2}$, where~$N_i$ has finite index in~$F_2$. In particular, the
closure of~$F_2 ^{(n)}$ in~$\hat F_2$, which is the kernel of~$\hat
F_2 \to H_n$, is preserved by all continuous automorphisms -- we call
it characteristic.

We proceed with the proof. Using the previous lemma we identify~$\hat
F_2$ and~$\lim_n H_n$. There is a natural map
\[ \lim_n Aut(H_n) \longrightarrow Aut_c( \lim_n H_n ) \, ,   \]
and since the kernel of~$\hat F_2 \to H_n$ is characteristic there is
also a map going the other way:
\[ Aut_c(\hat F_2) \longrightarrow \lim_n Aut(H_n) \, .   \]
%
These two maps are easily seen to be inverses to one another.

Next we show that the corresponding map
\[ \pi \colon \lim_n Aut(H_n) \longrightarrow \lim_n Out(H_n)  \]
is surjective. This can be done as follows. Suppose that a
representative~$\tilde \gamma_n \in Aut(H_n)$ of~$\gamma_n \in
Out(H_n)$ has been chosen. Pick any representative~$\tilde
\gamma_{n+1}$ of~$\gamma_{n+1}$. It may not be the case that~$\tilde
\gamma_{n+1}$ maps to~$\tilde \gamma_n$ under the map~$Aut(H_{n+1})
\to Aut(H_n)$, but the two differ by an inner automorphism of~$H_n$;
since~$H_{n+1} \to H_n$ is surjective, we can compose~$\tilde
\gamma_{n+1}$ with an inner automorphism of~$H_{n+1}$ to compensate
for this. This defines~$(\tilde \gamma_n)_{n \ge 1} \in \lim_n
Aut(H_n)$ by induction, and shows that~$\pi$ is surjective.

To study the kernel of~$\pi$, we rely on a deep theorem of
Jarden~\cite{ jarden}, which states that any automorphism of~$\hat
F_2$ which fixes all the open, normal subgroups is in fact inner. An
element~$\beta \in \ker(\pi)$ must satisfy this assumption: indeed
each open, normal subgroup of~$\hat F_2$ is the closure~$\bar N$ of a
normal subgroup~$N$ of finite index in~$F_2$, and each such subgroup
contains some~$F_2 ^{(n)}$ for some~$n$ large enough, so if~$\beta $
induces an inner automorphism of~$H_n$ it must fix~$\bar N$. We
conclude that the kernel of~$\pi$ is~$Inn(\hat F_2)$, and the lemma
follows.
\end{proof}



\subsection{A group containing~$\gal$}

We make use of the axiom of choice, and select an algebraic
closure~$\Omega $ of~$\qb(x)$.

The finite group~$H_n$ with its two generators gives a regular dessin,
and so also an extension of fields~$L_n/\qb (x)$ which is in~$\etq$;
it is Galois with $\Gal(L_n / \qb (x)) \cong H_n$. Now we may
choose~$L_n$ to be a subfield of~$\Omega $. What is more, $L_n$ is
then unique: for suppose we had~$L_n' \subset \Omega $ such that there
is an isomorphism of field extensions~$L_n \to L_n'$, then we would
simply appeal to the fact that any map~$L_n \to \Omega $ has its
values in~$L_n$, from basic Galois theory. In the same vein, we point
out that if~$L/\qb(x)$ is any extension which is isomorphic to~$L_n /
\qb(x)$, then any two isomorphisms~$L_n \to L$ differ by an element
of~$\Gal(L_n/\qb(x))$. From now on we identify once and for all~$H_n$
and~$\Gal(L_n/\qb(x))$.

Now let~$\lambda \in \Gal(\qb / \q )$. We have seen that~$\act{\lambda
}{L_n}$ is again regular (just like~$L_n$ is), and that it corresponds
to a choice of two new generators~$\sigma_\lambda $
and~$\alpha_\lambda $ of~$H_n$. However by (3) of
proposition~\ref{prop-props-of-Hn} there is an automorphism~$H_n \to
H_n$ such that~$\sigma \mapsto \sigma_\lambda $ and~$\alpha \mapsto
\alpha_\lambda $, and so~$L_n$ and~$\act{\lambda }{L_n}$ are
isomorphic. In other words there exists an isomorphism~$\iota \colon
L_n \to \act{\lambda }{L_n}$ of extensions of~$\qb(x)$, which is
defined up to pre-composition by an element of~$\Gal(L_n/\qb(x)) =
H_n$.

Given~$h \in H_n$, we may consider now the following diagram, which
{\em does not} commute.
\[ \begin{CD}
L_n @>{\iota }>> \act{\lambda }{L_n} \\
@V{h}VV                 @VV{\lambda ^*(h)}V \\
L_n @>{\iota }>> \act{\lambda }{L_n} \, . 
\end{CD}
  \]
The map~$\iota ^{-1} \circ \lambda ^*(h) \circ \iota $ depends on the
choice of~$\iota $, and more precisely it is defined up to conjugation
by an element of~$H_n$. As a result the automorphism~$h \mapsto \iota
^{-1} \circ \lambda ^*(h) \circ \iota$ of~$H_n$ induces a well-defined
element in~$ Out(H_n)$, which depends only on~$\lambda $.

\begin{thm}
There is an injective homomorphism of groups 
\[ \gt \colon \gal \longrightarrow \lim_n Out(H_n) \cong Out(\hat F_2) \, .   \]
\end{thm}

\begin{proof}
We have explained how to associate to~$\lambda \in \gal$ an element
in~$Out(H_n)$. First we need to prove that this gives a homomorphism 
\[  \gt_n \colon \gal \longrightarrow Out(H_n) \, ,   \]
for each fixed~$n$.  Assume that~$\gt_n(\lambda_i)$ is represented by
$ h \mapsto \iota_i^{-1} \circ \lambda_i^*(h) \circ \iota_i$, for~$i=
1, 2$. Then~$\gt_n (\lambda_1) \circ \gt_n (\lambda_2)$ is represented
by their composition, which is
\[ h \mapsto \iota_3^{-1} \circ (\lambda_1 \lambda_2)^* \circ \iota_3
\, ,   \]
where~$\iota_3 = \act{\lambda_1}{\iota_2} \circ
\iota_1$. Since~$\iota_3$ is an isomorphism~$L_n \to \act{\lambda_1
  \lambda_2}{L_n}$, we see that this automorphism represents~$ \gt_n
(\lambda_1 \lambda_2)$, so~$\gt_n (\lambda_1 \lambda_2) = \gt_n(\lambda_1)
\gt_n(\lambda_2)$, as requested.


Next we study the compatibility with the maps~$Out(H_{n+1}) \to
Out(H_n)$. The point is that~$L_n \subset L_{n+1}$, and that~$L_n$
corresponds to a {\em characteristic} subgroup of~$H_{n+1}$ in the
Galois correspondence (namely~$H_{n+1} ^{(n)}$). It follows that any
isomorphism~$L_{n+1} \to \act{\lambda } {L_{n+1}}$ must carry~$L_n$
onto~$\act{\lambda }{L_n}$. Together with the naturality
of~$\lambda^*$, this gives the desired compatibilities. 

Finally we must prove that~$\gt$ is injective. We have seen that the
action of~$\gal$ on dessins is faithful; so it suffices to shows
that whenever~$\gt (\lambda ) = 1$, the action of~$\lambda $ on
dessins is trivial.

To see this, pick any extension~$L$ of~$\qb(x)$, giving an object
in~$\etq$. It is contained in~$L_n$ for some~$n$, and corresponds to a
certain subgroup~$K$ of~$H_n$ in the Galois correspondence. By
lemma~\ref{lem-galois-correspondence}, $\act{\lambda }{L}$ corresponds
to~$\lambda ^*(K)$ as a subfield of~$\act{\lambda }{L_n}$. The
condition~$\gt(\lambda )= 1$ means that, if we identify~$\act{\lambda
}{L_n}$ with~$L_n$ by means of some choice of isomorphism~$\iota $
(which we may), the map~$\lambda^*$ becomes conjugation by a certain
element of~$H_n$. So~$\act{\lambda }{L}$ corresponds to a conjugate
of~$K$, and is thus isomorphic to~$L$ (this is part of the Galois
correspondence).
\end{proof}

\subsection{Action of~$Out(H_n)$ on dessins}

We seek to define a down-to-earth description of an action of~$\lim_n
Out(H_n)$ on (isomorphism classes of) dessins. In fact we only define
an action on {\em connected} dessins in what follows, and will not
recall that assumption. (It is trivial to extend the action to all
dessins if the reader wishes to do so.)


We work in~$\sets_{\sigma, \alpha , \phi}$, in which a typical
(connected) object is~$K \bs G$, where~$G$ is a finite group with two
generators~$\sigma $ and~$\alpha $ and~$K$ is a subgroup. Assume
that~$G$ has order~$\le n$. Then there is a surjective map~$p \colon
H_n \to G$, sending~$\sigma $ and~$\alpha $ to the elements bearing
the same name. We let~$N = \ker(p)$ and~$\bar K = p^{-1}(K)$.

Now suppose~$\gamma $ is an automorphism of~$H_n$. We can
consider~$\act{\gamma }{G} = H_n / \gamma (N)$, which we see as
possessing the distinguished generators~$\sigma $ and~$\alpha $, the
images under~$H_n \to H_n / \gamma (N)$ of the elements with the same
name. We certainly {\em do not} take~$\gamma (\sigma )$ and~$\gamma
(\alpha )$ as generators; on the other hand $\gamma $ induces an
isomorphism of groups~$G \to \act{\gamma }{G}$ which is {\em not}
compatible with the distinguished generators. Finally~$\act{\gamma
}{G}$ has the subgroup~$\act{\gamma }{K}$, the image of~$\gamma (\bar
K)$ under~$H_n \to H_n / \gamma (N)$. The object~$\act{\gamma }{K} \bs
\act{\gamma } {G}$ in~$\sets_{\sigma, \alpha, \phi}$ is the result of
applying~$\gamma $ to~$K \bs G$. Clearly this defines an action
of~$Out(H_n)$ on isomorphism classes of dessins whose cartographic
group has order~$\le n$.

\begin{lem}
Suppose~$\gamma \in Out(H_n)$ is of the form~$\gamma = \Gamma_n
(\lambda )$ for some~$\lambda \in \gal$. Then the action of~$\gamma $
on (isomorphism classes of) dessins agrees with that of~$\lambda $.
\end{lem}

\begin{proof}
We keep the notation introduced above, and write~$\cell$ for the
regular dessin defined by the finite group~$H_n$ with its two
canonical generators. The dessin~$X =K \bs G$ considered is the
intermediate dessin of~$\cell$ corresponding to the subgroup~$\bar K$
of~$Aut(\cell) = H_n$. Thus~$\act{\lambda } X$ corresponds to the
subgroup~$\lambda^*(\bar K)$ of~$Aut(\act{\lambda } \cell) =
\lambda^*(H_n)$. Picking an isomorphism~$\iota $ between~$\cell$
and~$\act{\lambda } {\cell}$ as before, we see that~$\act{\lambda } X$
is isomorphic~$\gamma  (\bar K) \bs H_n$ as requested.
\end{proof}

\begin{lem}
The actions defined above are compatible as~$n$ varies and can be
combined into a single action of~$\lim_n Out(H_n)$ on the isomorphism
classes of dessins.
\end{lem}

\begin{proof}
It suffices to prove that, for any integers~$n, s$, if we
pick~$\gamma_{n+s} \in Out(H_{n+s})$ and let~$\gamma_n$ be its image
under the projection~$Out(H_{n+s}) \to Out(H_n)$, then for any
dessin~$X$ whose cartographic group~$G$ has order~$\le n$ the
dessins~$\act{\gamma_{n+s}} X$ and~$\act{\gamma_n} X$ are
isomorphic. However this follows easily from the fact that the
projection~$p_{n+s} \colon H_{n+s} \to G$ factors as~$p_n \circ
\pi_{n+s}$, where we write~$\pi_{n+s} \colon H_{n+s} \to H_n$ for the
natural map.
\end{proof}

Now we seek to prove that the action of~$\lim_n Out(H_n)$ on dessins
is faithful. 

\begin{thm} \label{thm-action-out-faithful}
The group~$\lim_n Out(H_n) \cong Out(\hat F_2)$ acts faithfully on the
set of regular dessins. 
\end{thm}

\begin{proof}
Let~$\beta \in Aut(\hat F_2)$ correspond to~$\gamma = (\gamma_n)_{n
  \ge 1} \in \lim_n Out(H_n)$. If the action of this element is
trivial on the set of all regular dessins, then the
automorphism~$\beta $ must fix all open, normal subgroups of~$\hat
F_2$. However the theorem of Jarden already used in the proof of
lemma~\ref{lem-out-f2-is-inv-lim-out-hn} implies then that~$\beta $ is
an inner automorphism of~$\hat F_2$. As a result, $\gamma_n = 1$ for
all~$n$.
\end{proof}

Here it was necessary to see~$\lim_n H_n$ as~$Out(\hat F_2)$ to
conduct the proof (or more precisely, to be able to apply Jarden's
theorem which is stated in terms of~$\hat F_2$).

\begin{coro}
The group~$\gal$ acts faithfully on the set of regular dessins.
\end{coro}







\begin{ex} \label{ex-explicit-action-aut}
Suppose~$\gamma $ is an automorphism of~$H_n$ for which you have an
explicit formula, say
\[ \gamma (\sigma ) = \alpha \phi \alpha^{-1} \, , \qquad \gamma
(\alpha ) = \alpha \, .   \]
What is the effect of~$\gamma $ on dessins, explicitly? Discussing
this for regular dessins for simplicity, say you have~$G$, a finite
group of order~$\le n$ with two distinguished generators written as
always~$\sigma $ and~$\alpha $. Can we compute the effect of~$\gamma $
on~$(G, \sigma , \alpha )$ immediately?

The answer is that some care is needed. Looking at the definitions, we
write~$G = H_n / N$ for some uniquely defined~$N$, and the new dessin
is~$(H_n/ \gamma (N), \sigma, \alpha )$. If we want to write this more
simply, according to the principle that ``applying~$\gamma $ gives the
same group with new generators'', we exploit the isomorphism of groups
\[ G = H_n / N \longrightarrow H_n / \gamma (N)  \]
which is induced by~$\gamma $. Transporting the canonical generators
of~$H_n/ \gamma (N)$ to~$G$ {\em via} this isomorphism gives is
fact~$(G, \gamma^{-1}(\sigma ), \gamma^{-1}(\alpha ))$ (note the
inverses!).

In our case we compute~$\gamma ^{-1}(\sigma ) = \phi$,
$\gamma^{-1}(\alpha ) = \alpha $. In short 
\[ \act{\gamma }{(G, \sigma , \alpha )} = (G, \phi, \alpha )  \]
with, as ever, $\phi = (\sigma \alpha )^{-1}$. Incidentally, if we
compare this with example~\ref{ex-duality}, we see that the action
of~$\gamma $ is to turn a dessin into its ``dual''. 
\end{ex}

\subsection{The coarse Grothendieck-Teichmüller group}

Let us give a list of conditions describing a subgroup of~$\lim
\farout(H_n)$ containing the image of~$\gal$. 

\begin{lem} \label{lem-canonical-form-gt}
Let~$\gamma = \gt_n(\lambda )\in \farout(H_n)$, for some~$\lambda \in
\gal$. Then~$\gamma $ can be represented by an element of~$Aut(H_n)$,
still written~$\gamma $ for simplicity, and enjoying the following
extra properties: there exists an integer~$k$ prime to the order
of~$H_n$, and an element~$f \in [H_n, H_n]$, the commutator subgroup,
such that
\[ \gamma (\sigma ) = \sigma ^k \quad\textnormal{and}\quad \gamma (\alpha ) = f^{-1} \alpha^k
f \, .  \]
Moreover~$\gamma (\sigma \alpha )$ is conjugated to~$(\sigma \alpha
)^k$.
\end{lem}

\begin{proof}
This follows from proposition~\ref{prop-branch-cycle-argument} (the
branch cycle argument) applied to~$L_n$. More precisely, let us
write~$\sigma $ for~$\tilde \sigma $ and~$\alpha $ for~$\tilde \alpha
$, etc. Then there is an isomorphism~$\iota $ between~$L_n$
and~$\act{\lambda }{L_n}$, under which~$\sigma_\lambda \in
Aut(\act{\lambda } L_n)$ is identified with~$\sigma \in H_n$, and
similarly for~$\alpha_\lambda $ and~$\phi_\lambda $ ; as for~$\lambda
^*$, it becomes~$\gt_n(\lambda )$ when viewed in~$Out(H_n)$. Thus a
simple translation of the notation shows that~$\gamma (\sigma )$ is
conjugated to~$\sigma^k$, where~$k$ is determined by the action
of~$\lambda $ on roots of unity, while~$\gamma (\alpha )$ is
conjugated to~$\alpha^k$ and~$\gamma (\sigma \alpha) $ is conjugated
to~$(\sigma \alpha )^k$. By composing with an inner automorphism, we
may thus assume that~$\gamma (\sigma ) = \sigma ^k$.

Let~$g \in H_n$ be such that~$\gamma (\alpha ) = g^{-1} \alpha^k
g$. Every element of the abelian group~$H_n / [H_n, H_n]$ can be
written~$ \alpha ^j \sigma^i$ for some integers~$i, j$, so let us
write~$g= \alpha^j \sigma^i c_1$ for some~$c_1 \in [H_n,
  H_n]$. Further put~$\sigma^i c_1 = c_2 c_1 \sigma^i$; here~$c_2$ is
a commutator, so that~$f= c_2 c_1 \in [H_n, H_n]$. Thus~$g= \alpha^j f
\sigma^i$ and
\[ \gamma (\alpha ) = g \alpha^k g^{-1} = (\sigma^{-i} f^{-1} \alpha^{-j}) \alpha^k (\alpha^{j} f
\sigma^{i}) = \sigma^{-i} (f \alpha^k  f^{-1}) \sigma^{i} \, .  \]
By composing~$\gamma $ with conjugation by~$\sigma^i$, we obtain a
representative which is of the desired form.
\end{proof}

For each~$n$ there is an automorphism~$\delta_n$ of~$H_n$
satisfying~$\delta_n (\sigma ) = \alpha \phi \alpha^{-1} = \sigma^{-1}
\alpha^{-1}$, $\delta_n (\alpha ) = \alpha$, $\delta_n (\phi) =
\sigma$. We write~$\delta = (\delta _n)_{n \ge 1}$ for the
corresponding element of~$\lim_n Out(H_n)$. The letter~$\delta $ is
for duality, as the next lemma explains.

\begin{lem}
\begin{enumerate}
\item The dessin~$\act{\delta }{\cell}$ resulting from the application
  of~$\delta $ to an arbitrary dessin~$\cell$ is its
  ``dual''. If~$\cell$ corresponds to the surface~$S$ endowed with the
  Belyi map~$F \colon S \to \p$, then~$\act{\delta }{\cell}$
  corresponds to~$S$ endowed with~$1 / F$.

\item If~$\gamma = \gt(\lambda )\in \lim_n \farout(H_n)$ for~$\lambda \in
  \gal$, then~$\gamma $ and~$\delta $ commute.
\end{enumerate}
\end{lem}

Note that~$\delta $ squares to conjugation by~$\alpha $. Thus
in~$\farout(H_n)$, it is equal to its inverse, and
the letter~$\omega $ is often used in the literature
for~$\delta^{-1}$. 

\begin{proof}
(1) follows from the computations in example~\ref{ex-duality}
  and example~\ref{ex-explicit-action-aut}. 

Since the Galois action proceeds by the effect of~$\lambda \in \gal$
on the coefficients of the equations defining~$S$ as a curve, and the
coefficients of the rational fraction~$F$, the first point implies
that~$\act{\lambda \delta }{\cell} \cong \act{\delta \lambda }{\cell}$
for any dessin~$\cell$. Since the action of~$\lim_n Out(H_n)$ on
isomorphism classes of dessins is faithful, this implies~$\lambda
\delta = \delta \lambda $.
\end{proof}

Note that we have relied on the point of view of algebraic curves in
this argument.

Now we turn to the study of the automorphism of~$H_n$ usually
written~$\theta_n$ which satisfies~$\theta_n (\sigma ) = \alpha $
and~$\theta_n (\alpha ) = \sigma $. We write~$\theta = (\theta_n)_{n
  \ge 1}$ for the corresponding element of~$\lim_n Out(H_n)$.

\begin{lem}
\begin{enumerate}
\item The dessin~$\act{\theta}{\cell}$ resulting from the application
  of~$\theta $ to an arbitrary dessin~$\cell$ is simply obtained by
  changing the colours of all the vertices in~$\cell$. If~$\cell$
  corresponds to the surface~$S$ endowed with the Belyi map~$F \colon
  S \to \p$, then~$\act{\theta }{\cell}$ corresponds to~$S$ endowed
  with~$1 - F$.

\item If~$\gamma = \gt(\lambda ) \in \lim_n Out(H_n)$ for~$\lambda \in
  \gal$, then~$\gamma $ and~$\theta $ commute.
\end{enumerate}
\end{lem}

\begin{proof}
As the previous proof, based on example~\ref{ex-change-colours}.
\end{proof}

We come to the definition of the {\em coarse Grothendieck-Teichmüller
  group}, to be denoted~$\GT$. In fact, we start by defining the
subgroup~$\GT (n)$ of~$Out(H_n)$ comprised of
all the elements~$\gamma $ such that:
\begin{enumerate}
\item[(GT0)] $\gamma$ has a representative in $Aut(H_n)$, say~$\tilde
  \gamma$, for which there exists an integer~$k_n$ prime to the
  order of~$H_n$, and an element~$f_n \in [H_n, H_n]$, such that
\[ \tilde \gamma (\sigma ) = \sigma ^{k_n} \quad\textnormal{and}\quad \tilde \gamma (\alpha ) = f_n^{-1} \alpha^{k_n}
f_n \, .  \]

\item[(GT1)] $\gamma $ commutes with~$\theta_n $.
\item[(GT2)] $\gamma $ commutes with~$\delta_n $,
\end{enumerate}

Remark that conditions (GT2) and (GT0) together imply that~$\tilde
\gamma (\sigma \alpha )$ is conjugated to~$(\sigma \alpha )^{k_n}$.

We let~$\GT = \lim_n \GT (n)$. The contents of this section may thus
be summarized as follows, throwing in the extra information we have
from proposition~\ref{prop-branch-cycle-argument}:

\begin{thm} \label{thm-main-gt}
There is an injective homomorphism 
\[ \gt \colon \gal \longrightarrow \GT \, .   \]
Moreover, for~$\gamma = \gt(\lambda )$, the integer~$k_n$ can be taken
to be any integer satisfying 
\[ \lambda (\zeta_N) = \zeta_N^{k_n} \, .   \]
Here~$N$ is the order of~$H_n$, and~$\zeta_N = e^{\frac{2 i \pi} {N}}$.
\end{thm}

\subsection{Variants} \label{subsec-variants}

It should be clear that the groups~$H_n$ are not the only ones we
could have worked with. In fact, let~$\N$ be a collection of subgroups
of~$F_2$ with the following properties: \begin{itemize}
\item[(i)] each~$N \in \N$ has finite index in~$F_2$ ,
\item[(ii)] each~$N \in \N$ is characteristic (and in particular normal),
\item[(iii)] for any normal subgroup~$K$ in~$F_2$, there exists~$N \in
  \N$ such that~$N \subset K$.
\item[(iv)] for each~$N \in \N$, the group~$G= F_2 / N$ has the following property: given two pairs of generators~$(g_1, g_2)$ and~$(h_1, h_2)$ for~$G$, there exists an automorphism of~$G$ taking~$g_i$ to~$h_i$, for~$i= 1, 2$.
\end{itemize} 

\noindent So far we have worked with~$\N=$ the collection of all subgroups~$F_2
^{(n)}$ (for~$n \ge 1$). Other choices include: \begin{itemize}
\item For~$n \ge 1$, let~$F_2^{[n]}=$ the intersection of all normal
  subgroups of~$F_2$ of order {\em dividing}~$n$. Then take $\N=$ the
  collection, for all~$n \ge 1$, of all the groups~$F ^{[n]}$.

\item For~$G$ a finite group, let~$N_G=$ the intersection of all the
  normal subgroups~$N$ of~$F_2$ such that~$F_2 / N$ is isomorphic
  to~$G$ (the group~$G$ {\em not} having distinguished
  generators). Then take~$\N=$ the collection of all~$N_G$, where~$G$
  runs through representatives for the isomorphism classes of finite
  groups which can be generated by two elements. 
\end{itemize}

\noindent To establish condition (iv) in each case, one proves  a more ``universal'' property analogous to (2) of proposition~\ref{prop-props-of-Hn} for~$H_n$.

The reader will check that all the preceding material is based only on these four conditions, and the results below follow {\em mutatis mutandis}. First, as in~$\S\ref{subsec-Hn}$ we have
\[ \hat F_2 \cong \lim_{N \in \N} F_2 / N \, ,   \]
and 
\[ Out(\hat F_2) \cong \lim_{N \in \N} Out( F_2 / N ) \, .   \]
In particular we have maps~$\gal \to Out(F_2 / N)$ for~$N$ running
through~$\N$, and any non-trivial element of~$\gal$ has non-trivial
image in some~$Out(F_2 / N)$. 

Let us introduce the notation~$\GT(K)$, for any characteristic subgroup~$K$ of finite index in~$F_2$, to mean the subgroup of~$Out(F_2/K)$ of those elements satisfying (GT0) - (GT1) - (GT2). Note that~$N$ being characteristic, it makes sense to speak of~$\delta $ and~$\theta $ as elements of~$Out(F_2 / N)$. In the same fashion we define~$\GT(K)$, as a subgroup of~$Out(\hat F_2/K)$, when $K$ is open and characteristic in~$\hat F_2$.

With this terminology, one proves that the elements of~$Out(F_2 / N)$
coming from elements of~$\gal$ must in fact lie in~$\GT(N)$. If we let~$\GT(\N)$ denote the inverse limit of the groups~$\GT(N)$ for~$N \in \N$, then it is isomorphic to a subgroup of~$Out(\hat F_2)$ and we have an injection of~$\gal$ into~$\GT(\N)$.

The next lemma then proves that~$\GT(\N)$ is independent of~$\N$:

\begin{lem}
Let~$\beta \in Out(\hat F_2)$. Then~$\beta $ lies in~$\GT (\N)$ if and
only if for each open, characteristic subgroup~$K$ of~$\hat F_2$, the
induced element of~$Out(\hat F_2 / K)$ is in~$\GT( K)$.

In particular, the group~$\GT(\N)$, as a subgroup of~$Out(\hat F_2)$
is independent of the choice of~$\N$.
\end{lem}


\begin{proof}
The condition is clearly sufficient, as we see by letting~$K$ run
through the closures of the elements of~$\N$. 

To see that it is necessary, we only need to observe that~$K$ contains
the closure of an element~$N\in\N$, so~$\hat F_2 /K$ is a quotient
of~$F_2 / N$ and the automorphism induced by~$\beta $ on~$\hat F_2 /K$ is
also induced by an element of~$\GT( N)$; thus it must lie
in~$\GT(K)$.

This characterization of elements of~$\GT(\N)$ visibly does not make
any reference to~$\N$.
\end{proof}

In theory, all choices for~$\N$ are equally valid, and in fact no mention of {\em any} choice is necessary: one may state all the results of this section in terms of~$Out(\hat F_2)$, for example defining~$\GT$ by the characteristic property given in the lemma. In practice however, choosing a collection~$\N$ allows us to compute~$\GT(N)$ explicitly for some groups~$N \in \N$, and that is at least a baby step towards a description of~$\gal$. The difficulty of the computations will depend greatly on the choices we make. For example, with the groups~$F ^{(n)}$, the order of~$H_n$ increases very rapidly with~$n$, but the indexing set is very simple; with~$F ^{[n]}$, the order of~$F_2 / F_2 ^{[n]}$ is much less than the order of~$H_n$, but the inverse limits are more involved. In a subsequent publication, computations with the family~$\N$ of all the groups of the form~$N_G$ will be presented.

We conclude with yet another definition of~$\GT$ which does involve choosing a collection~$\N$. This is the traditional definition.

\subsection{Taking coordinates; the group~$\rGT_0$}

We start with a couple of observations about~$H_n$.

\begin{lem}
If~$k_1$ and~$k_2$ are integers such that~$\sigma ^{k_1}$
and~$\sigma^{k_2} $ are conjugate in~$H_n$, then~$k_1 \equiv k_2$
mod~$n$. Similarly for~$\alpha $.
\end{lem}

\begin{proof}
We use the map~$H_n \to C_n = \langle x \rangle$, where~$C_n$ is the
cyclic group of order~$n$, sending both~$\sigma $ and~$\alpha $
to~$x$. The image of~$\sigma ^{k_i}$ is~$x^{k_i}$ (for~$i= 1, 2$), and
conjugate elements of~$C_n$ are equal, so~$k_1 \equiv k_2$ mod~$n$.
\end{proof}

\begin{coro}
Let~$\gamma \in \GT (n)$. For~$i= 1, 2$, let~$\tilde \gamma_i$ be a 
representative for~$\gamma $ in~$Aut(H_n)$ such that~$\tilde \gamma_i
(\sigma )$ is conjugate to~$\sigma^{k_i}$. Then~$k_1 \equiv k_2$
mod~$n$. This defines a homomorphism 
\[ \GT (n) \longrightarrow (\z/n)^\times \, ,   \]
which we write~$\gamma \mapsto k(\gamma )$ (or sometimes~$k_n(\gamma
)$ for emphasis).

Letting~$n$ vary, we obtain a homomorphism 
\[ k \colon \GT \longrightarrow \hat\z^\times \, .   \]
\end{coro}

Here~$\hat \z = \lim_n \z / n \z$ is the profinite completion of the
ring~$\z$. 

\begin{prop} \label{prop-lift-a-la-drinfeld}
Let~$\gamma \in \GT$. Then~$\gamma $ has a lift~$\beta \in Aut(\hat
F_2)$ satisfying 
\[ \beta (\sigma ) = \sigma ^{k(\gamma )} \, , \qquad \beta (\alpha )
= f^{-1} \alpha ^{k(\gamma )} f \, ,   \]
for some~$f \in [\hat F_2, \hat F_2]$, the commutator subgroup. The
element~$f$ is unique, and as a result, so is~$\beta $. 
\end{prop}

\begin{proof}
Start with any lift~$\beta_0$. The elements~$\beta_0(\sigma )$
and~$\sigma ^{k(\gamma   )}$ are conjugate in every group~$H_n$,
so~$\beta_0( \sigma )$ is in the closure of the conjugacy class
of~$\sigma^{k(\gamma )}$. However this class is closed (the map~$x
\mapsto x \sigma^{k(\gamma )} x^{-1}$ is continuous and its image must
be closed since its source~$\hat F_2$ is compact). So~$\beta_0(\sigma
)$ is conjugated to~$\sigma^{k(\gamma )}$, and
likewise~$\beta_0(\alpha )$ is conjugated to~$\alpha^{k(\gamma
  )}$. Now, argue as in lemma~\ref{lem-canonical-form-gt} to obtain
the existence of a representative~$\beta $ as stated. 

We turn to the uniqueness. If~$f'$ can replace~$f$, then~$f= c_1 f'
c_2$ where~$c_2$ centralizes~$\sigma $ and~$c_1$ centralizes~$\alpha
$. However the centralizer of~$\sigma $ in~$\hat F_2$ is the (closed)
subgroup generated by~$\sigma $ and likewise for~$\alpha $. Since~$f$
and~$f'$ are assumed to be both commutators, we can reduce mod~$[\hat
  F_2, \hat F_2]$ and obtain a relation~$c_1 c_2 = 1$; the latter must
then hold true in any finite, abelian group on two generators~$\sigma
$ and~$\alpha $, and this is clearly only possible if~$c_1 = c_2 = 1$
in~$\hat F_2$.
\end{proof}

We observe at once:

\begin{coro}
The injection~$\gt\colon\gal \to Out(\hat F_2)$ lifts to an injection
$\tilde\gt\colon\gal \to Aut(\hat F_2)$. In particular, an element
of~$\gal$ can be entirely described by a pair~$(k, f) \in
\hat\z^\times \times [\hat F_2, \hat F_2]$.
\end{coro}

\begin{proof}
Let~$\tilde \gt(\lambda )$ be the lift of~$\gt(\lambda )$ described in
the proposition. The composition of two automorphisms of~$\hat F_2$ of
this form is again of this form, so~$\tilde \gt(\lambda ) \tilde
\gt (\mu )$ must be the lift of~$\gt(\lambda ) \gt(\mu ) =
\gt(\lambda \mu )$, that is, it must be equal to~$\tilde \gt(\lambda
\mu )$.
\end{proof}

We want to describe a group analogous to~$\GT$ in terms of the
pairs~$(k, f)$. There is a subtlety here, in that if we pick~$k \in
\hat\z^\times $ and~$f \in [\hat F_2, \hat F_2]$ arbitrarily, the
self-homomorphism~$\beta $ of~$\hat F_2$ satisfying
\[ \beta (\sigma ) = \sigma ^{k} \, , \qquad \beta (\alpha )
= f^{-1} \alpha ^{k} f  \tag{*}  \]
may not be an automorphism. Keeping this in mind, we define a
group~$\rGT_0$ now -- the notation is standard, and the index ``0'' is
not to be confused with our writing~$\GT (n)$ for~$n=0$; moreover the
notation does not refer to a profinite completion of some underlying
group~$\GT_0$. So let~$\rGT_0$ be the group of all
pairs~$(k, f) \in \hat\z^\times \times [\hat F_2, \hat F_2]$ such
that :
\begin{itemize}
\item Let $\beta $ be the self-homomorphism defined by (*);
  then~$\beta $ is an automorphism.
\item $\beta $ commutes with~$\delta $ in~$Out(\hat F_2)$.
\item $\beta $ commutes with~$\theta $ in~$Out(\hat F_2)$.
\end{itemize}
The composition law on~$\rGT_0$ is defined {\em via} the composition
of the corresponding automorphisms of~$\hat F_2$; one may recover~$k$
and~$f$ from~$\beta $, and indeed~$\rGT_0$ could have been defined as
a subgroup of~$Aut(\hat F_2)$, though that is not what has been
traditionally done in the literature. 

The definition of~$\rGT_0$ was given by Drinfeld
in~\cite{drinfeld}. The reader who is familiar with {\em loc.\ cit.}
may not recognize~$\rGT_0$ immediately behind our three conditions, so
let us add:

\begin{lem}
This definition of~$\rGT_0$ agrees with Drinfeld's. 
\end{lem}

\begin{proof}
This follows from~\cite{survey}, \S1.2, last theorem, stating that
``conditions (I) and (II)'' are equivalent with the commutativity
conditions with~$\theta $ and~$\delta $ respectively (the author using
the notation~$\omega $ for an inverse of~$\delta $ in~$Out(\hat
F_2)$).
\end{proof}

The natural map~$Aut(\hat F_2) \to Out(\hat F_2)$ induces a
map~$\rGT_0 \to \GT$. The existence and uniquess statements in
proposition~\ref{prop-lift-a-la-drinfeld} imply the surjectivity and
injectivity of this map, respectively, hence:

\begin{prop}
$\rGT_0$ and~$\GT$ are isomorphic.
\end{prop}

One may rewrite the main theorem of this section,
theorem~\ref{thm-main-gt}, as follows:

\begin{thm}
There is an injective homomorphism of groups 
\[ \gal \longrightarrow \rGT_0 \, .   \]
Composing this homomorphism with the projection~$\rGT_0 \to \hat
\z^\times$ gives the cyclotomic character of~$\gal$.
\end{thm} 

We conclude with a few remarks about the (real) {\em
  Grothendieck-Teichmüller group}. This is a certain subgroup
of~$\rGT_0$, denoted~$\rGT$, also defined by Drinfeld
in~\cite{drinfeld}. It consists of all the elements of~$\rGT_0$
satisfying the so-called ``pentagon equation'' (or ``condition
(III)''). 

Ihara in~\cite{ihara} was the first to prove the existence of an
injection of~$\gal$ into~$\rGT$. His method is quite different from
ours, and indeed proving the pentagon equation following our
elementary approach would require quite a bit of extra work.

Another noteworthy feature of Ihara's proof (beside the fact that it
refines ours by dealing with~$\rGT$ rather than~$\rGT_0$) is that it
does not, or at least not explicitly, refer to dessins d'enfants. It
is pretty clear that the original ideas stem from the material in the
{\em esquisse}~\cite{esquisse} on dessins, but the children's drawings
have disappeared from the formal argument. We hope to have
demonstrated that the elementary methods could be pushed quite a long
way.

\bibliography{myrefs}
\bibliographystyle{amsalpha}

\end{document}